\documentclass[reqno, 12pt]{amsart}

\allowdisplaybreaks

\usepackage[truedimen,margin=25truemm]{geometry}
\usepackage{amsmath,amsfonts,amssymb,graphicx,amsthm,color,yfonts,cite, latexsym}
\usepackage{paralist}
\usepackage{mathrsfs}
\usepackage{hyperref}
\usepackage{physics}
\usepackage{bbm}
\usepackage{bm}
\usepackage{mathtools}
\mathtoolsset{showonlyrefs}

\usepackage{tikz}
\usepackage{tikz-cd}
\usetikzlibrary{cd}
\usetikzlibrary{arrows.meta}
\usetikzlibrary{matrix, arrows, positioning}
\usetikzlibrary{shadows}

\newtheorem{theorem}{Theorem}

\newtheorem{proposition}[theorem]{Proposition}
\newtheorem{assumption}[theorem]{Assumption}
\newtheorem{lemma}[theorem]{Lemma}
\newtheorem{corollary}[theorem]{Corollary}

\theoremstyle{remark}
\newtheorem{remark}[theorem]{Remark}
\newtheorem{example}[theorem]{Example}

\makeatletter
\newcommand*{\rom}[1]{\expandafter\@slowromancap\romannumeral #1@}
\makeatother

\newcommand{\ls}{\lesssim}

\renewcommand{\lg}{\langle}
\newcommand{\rg}{\rangle}
\newcommand{\R}{\mathbb{R}}

\newcommand{\C}{\mathbb{C}}

\newcommand{\ep}{\varepsilon}
\newcommand{\al}{\alpha}

\newcommand{\CB}{\mathcal{B}}

\newcommand{\CD}{\mathcal{D}}

\newcommand{\CF}{\mathcal{F}}

\newcommand{\CM}{\mathcal{M}}

\newcommand{\CP}{\mathcal{P}}
\newcommand{\CQ}{\mathcal{Q}}
\newcommand{\CR}{\mathcal{R}}
\newcommand{\CS}{\mathcal{S}}

\newcommand{\CU}{\mathcal{U}}
\newcommand{\CV}{\mathcal{V}}
\newcommand{\CW}{\mathcal{W}}
\newcommand{\CX}{\mathcal{X}}
\newcommand{\CY}{\mathcal{Y}}

\newcommand{\SA}{\mathscr{A}}
\newcommand{\SB}{\mathscr{B}}
\newcommand{\SC}{\mathscr{C}}
\newcommand{\SD}{\mathscr{D}}
\newcommand{\SE}{\mathscr{E}}
\newcommand{\SF}{\mathscr{F}}
\newcommand{\SG}{\mathscr{G}}
\newcommand{\SH}{\mathscr{H}}
\newcommand{\SI}{\mathscr{I}}

\newcommand{\FS}{\mathfrak{S}}

\newcommand{\ga}{\gamma}
\newcommand{\de}{\delta}
\newcommand{\ps}{\psi}
\newcommand{\ph}{\varphi}
\newcommand{\ta}{\tau}
\renewcommand{\th}{\theta}
\newcommand{\lm}{\lambda}
\newcommand{\si}{\sigma}

\newcommand{\ze}{\zeta}
\newcommand{\rh}{\rho}
\newcommand{\ch}{\chi}

\newcommand{\De}{\Delta}
\newcommand{\Ph}{\Phi}
\newcommand{\Ps}{\Psi}

\newcommand{\pl}{\partial}
\newcommand{\wt}{\widetilde}
\newcommand{\wh}{\widehat}

\newcommand{\mx}{{\rm max}}

\newcommand{\Ck}[1]{\left\{#1\right\}}
\newcommand{\Dk}[1]{\left[#1\right]}
\newcommand{\K}[1]{\left(#1\right)}

\newcommand{\No}[1]{\left\| #1 \right\|}
\newcommand{\Bk}[1]{\Big[#1\Big]}
\newcommand{\Bck}[1]{\Big\{#1\Big\}}
\newcommand{\I}{\infty}
\newcommand{\sd}{\langle \hb \nabla \rangle}
\newcommand{\bx}{\langle x \rangle}

\newcommand{\tw}{\frac{1}{2}}
\newcommand{\na}{\nabla}
\newcommand{\hb}{\hbar}
\newcommand{\ov}{\overline}
\newcommand{\bt}{\lg t \rg}

\newcommand{\J}{\mathcal{J}^\hb}
\newcommand{\U}{\CU^\hb}

\newcommand{\W}{\CW^\hb}

\newcommand{\Wig}{\mathbf{Wig}^\hbar}

\newcommand{\gah}{\ga^\hb}
\newcommand{\rhh}{\rh^\hb}
\newcommand{\M}{\CM^\hb}
\newcommand{\D}{\CD^\hb}
\newcommand{\Id}{\textup{Id}}

\newcommand{\xh}{\frac{x}{\hb}}

\renewcommand{\Tr}{\operatorname{Tr}}

\newcommand{\bta}{\lg\ta\rg}
\newcommand{\sa}{\textup{sa}}

\newcommand{\Sp}{{\FS^2_\hb}}

\newcommand{\Jt}{\langle \J(t) \rangle}
\newcommand{\Jta}{\langle \J(\ta) \rangle}

\usepackage{wrapfig}
\usepackage{tikz}
\usetikzlibrary{arrows,calc,decorations.pathreplacing}
\definecolor{light-gray1}{gray}{0.90}
\definecolor{light-gray2}{gray}{0.80}
\definecolor{light-gray3}{gray}{0.60}

\numberwithin{equation}{section}

\numberwithin{theorem}{section}

\numberwithin{table}{section}

\numberwithin{figure}{section}

\makeatletter
\newcommand\thankssymb[1]{\textsuperscript{\@fnsymbol{#1}}}
\makeatother

\title[Uniform dispersive estimates]{Uniform dispersive estimates for the semi-classical Hartree equation with long-range interaction}
\author[S. HADAMA]{Sonae HADAMA}
\address{The University of Osaka, Japan.}
\email{hadama.sonae.sci@osaka-u.ac.jp}
\thanks{This work would not have come to fruition without the collaborative research \cite{Hadama Hong 2025} with Professor Y. Hong at Chung-Ang University. The author is deeply grateful to Mr. A. Hoshiya at the University of Tokyo for reading an early draft and providing numerous insightful comments. The author also wishes to express sincere thanks to Professors K. Nakanishi and N. Kishimoto for their many helpful suggestions and constructive feedback.}

\date{}

\begin{document}
\begin{abstract}
In this paper, we consider the Hartree equation with smooth but long-range interaction in the semi-classical regime, in three-dimensional space. We show that the density function of a small-data solution decays at the optimal rate. When the semi-classical parameter $\hbar \in (0,1]$ is fixed, our result is essentially covered by the recent work by Nguyen and You; however, the novelty of this paper is the uniformity with respect to $\hbar$. Namely, both the smallness condition for the initial data and the bounds for the solution are independent of $\hbar$. Our proof relies on three main ingredients. First, we prove the boundedness of finite-time wave operators modified by phase corrections. Second, we show an $L^1$--$L^\infty$ dispersive estimate for the modified propagator. Third, we give various kinds of commutator estimates for density operators. By combining them, we can apply the usual bootstrap argument to obtain the main result.
\end{abstract}
\maketitle
\tableofcontents

\section{Introduction}
\subsection{Setup of the problem}\label{subsec:intro}
	In this paper, we consider the Hartree equation in the three-dimensional space:
	\begin{equation}\label{eq:NLH}\tag{NLH}
		i\hb\pl_t\gah=\Dk{-\frac{\hbar^2}{2}\De_x+w\ast \rh^\hb(\gah),\gah},\quad \gah:\R_{\ge 0}\to\CB_\sa(L^2(\R^3)),
	\end{equation}
	where $w:\R^3 \to \R$ is a given function, $\CB_\sa(L^2(\R^3))$ is the space of all bounded self-adjoint operators on $L^2(\R^3)$, $[A,B]=AB-BA$ is the commutator, and $\hb \in(0,1]$ is the reduced Planck constant.
	We often write $A(x,x')$ for the integral kernel of the linear operator $A$ on $L^2(\R^3)$.
	We \textit{formally} define the (semi-classically scaled) density function of $A$ by $\rhh(A)(x)=(2\pi\hb)^{3}A(x,x)$.
	The rigorous definition of $\rhh(A)$ will be given in Section \ref{subsec:definition of density functoin}. 
	
It is well-known that the Hartree equation \eqref{eq:NLH} appears in many-body problems of quantum mechanics, 
and there is a vast body of mathematical literature on this equation.
When the initial data is given by a rank-one operator $\ga_0^\hb = |u_0 \rg \lg u_0|$ and $\hb=1$, \eqref{eq:NLH} is equivalent to
\begin{equation}\label{eq:NLH1}
	i\pl_t u = -\frac{1}{2} \De_x u + (2\pi)^3(w\ast |u|^2)u, \quad u:\R_{\ge 0}\times \R^3 \to \C 
\end{equation}
with initial condition $u(0)=u_0$.
Studying \eqref{eq:NLH1} is one of the most important directions in the mathematical study of \eqref{eq:NLH}.
Since there are numerous studies on \eqref{eq:NLH1}, it is impossible to provide a fully detailed and comprehensive review. We limit ourselves here to referring only to \cite{Hayashi Tsutsumi 1987,Hayashi Naumkin 1998, Hayashi Naumkin 1998 b, Hayashi et al 1998, Nakanishi 2002, Nakanishi 2002 b}.
In a different direction, several studies have analyzed \eqref{eq:NLH} around the translation-invariant stationary solutions.
See, for example, \cite{Chen et al 2017, Chen et al 2018, Lewin Sabin 2014, Lewin Sabin 2015, Lewin Sabin 2020, Hadama 2023, Hadama Hong 2024, Smith 2025, You 2024, Nguyen You 2023, Nguyen You 2025, Smith 2026}.

The global well-posedness of \eqref{eq:NLH} was shown in \cite{Bove et al 1974, Bove et al 1976, Chadam 1976, Zagatti 1992}.
For long-time behavior of solutions, a small-data scattering result was given in \cite{Pusateri Sigal 2021} for short-range interactions. In \cite{Pusateri Sigal 2021}, the optimal decay of the density function of the solution $\gah(t)$
\begin{equation}\label{eq:density function decay}
	\No{\rhh(\ga^\hb(t))}_{L^\I_x} \ls \frac{1}{\bt^3}
\end{equation}
is given when $\hb=1$ and the initial data is small, where $\bt := (1+|t|^2)^{1/2}$.
However, if we directly apply the argument in \cite{Pusateri Sigal 2021} to general $\hbar\in(0,1]$, then the smallness condition for initial data gets stronger as $\hbar\to 0$. More precisely, for any $\hb\in(0,1]$, we can actually prove \eqref{eq:density function decay} when the initial data $\gah_0$ is small. However, we immediately find that
\begin{equation}\label{eq:vanish}
	\|\gah_0\|_{\CX_\hb}\to 0 \mbox{ as } \hb \to 0
\end{equation}
for any acceptable family of initial data $(\gah_0)_{\hb\in(0,1]}$, where $\|\cdot\|_{\CX_\hb}$ is a natural norm in the semi-classical regime.

The recent work by Smith \cite{Smith 2025} proved \eqref{eq:density function decay} for regular interaction around a certain class of translation-invariant stationary solutions for all $\hb\in (0,1]$ \textit{when $\|\gah_0\|_{\CX_\hb}\le \ep_0$, where $\ep_0>0$ is a small but $\hb$--independent constant.}
Moreover, Hong and the author obtained a similar result in \cite{Hadama Hong 2025} near vacuum when $w(x)=\pm|x|^{-a}$ for $1<a<5/3$, that is, $w$ is a more singular short-range interaction.
The importance of these improvements will be discussed in Section \ref{subsec:importance}.

Although both \cite{Pusateri Sigal 2021} and \cite{Smith 2025,Hadama Hong 2025} dealt with short-range interaction $w$,
a modified scattering result for the Coulomb interactions was proved in \cite{Nguyen You 2025}, and the same optimal decay rate \eqref{eq:density function decay} is given when $\hb = 1$. 
In this paper, we prove \eqref{eq:density function decay} \textit{for long-range interaction $w$ when $\|\gah_0\|_{\CX_\hb}\le \ep_0$, where $\ep_0>0$ is a small but $\hbar$--independent constant.}

\subsection{Main result}
Before stating the main result, we prepare some notation.
First of all, if we write $A\ls B$, then it means $A \le CB$, where $C>0$ is an $\hb$-independent constant.
Let $\CB=\CB(L^2(\R^3))$ be the space of all bounded linear operators.
Define the Schatten--$r$ norm by
\begin{equation}
	\|A\|_{\FS^r} := \begin{dcases}
		&\Ck{\Tr\K{|A|^r}}^{1/r}, \quad \text{for } r \in [1,\I), \\
		&\|A\|_{\CB}, \quad \text{for } r =\I
	\end{dcases}
\end{equation}
for compact operators $A$.
Moreover, we define the semi-classically scaled Schatten--$r$ norm by
\begin{align}
	\|A\|_{\FS^r_\hb} := (2\pi\hb)^{3/r} \|A\|_{\FS^r}.
\end{align}
Throughout this paper, we use the following class of operators for initial data:
\begin{equation}\label{eq:initial data class}
	\begin{aligned}
		\CX_\hb^\si &:= \Big\{ A \in \CB(L^2(\R^3)) \mid A \mbox{ is self-adjoint and } \|A\|_{\CX^\si_\hb} < \I\Big\},
	\end{aligned}
\end{equation}
where the norm $\|\cdot\|_{\CX^\si_\hb}$ is defined by
\begin{equation}\label{eq:initial data}
	\begin{aligned}
		\|A\|_{\CX_\hb^\si} &:= \|A\|_{\FS^1_\hb}
			+ \sum_{X\in \{x,-i\hb\na_x\}}\No{\lg X\rg^{\si} A \lg X \rg^{\si}}_{\FS^2_\hb\cap \CB}
			+ \sum_{X\in \{x,-i\hb\na_x\}} \No{\Dk{\frac{X}{i\hb},A}}_{\FS^{1}_\hb}  \\
			&\qquad +\sum_{X,Y\in\{x,-i\hb\na_x\} } \No{\lg X\rg^{\si} \Dk{\frac{Y}{i\hb}, A}\lg X\rg^{\si}}_{\FS^2_\hb \cap \CB} \\
			&\qquad +\sum_{X,Y,Z\in\{x,-i\hb\na_x\} } \No{\lg X\rg^{\si/2} \Dk{\frac{Y}{i\hb},\Dk{\frac{Z}{i\hb},A}}\lg X\rg^{\si/2}}_{\FS^2_\hb}, 
	\end{aligned}
\end{equation}
for $\si\in(3/2,2)$.

\begin{remark}
	The norm $\|\cdot\|_{\CX^\si_\hb}$ defined in \eqref{eq:initial data} is natural in the semi-classical regime.
	See \cite{Hadama Hong 2025}.
\end{remark}

\begin{remark}
	The precise meaning of double commutator terms is, for example, 
	\begin{equation}
		\No{\sd^{\si/2}\Bk{\na,\Bk{\na,A}}\sd^{\si/2}}_{\FS^{2}_\hb}
		= \sum_{j,k=1}^3 \No{\sd^{\si/2}\Bk{\pl_{x_j},\Bk{\pl_{x_k},A}}\sd^{\si/2}}_{\FS^{2}_\hb}.
	\end{equation}
\end{remark}

For interaction $w(x)$, we assume the following condition:
\begin{align}\label{eq:ass} \tag{\textbf{A}}
	w \in C^3(\R^3), \quad \mbox{and} \quad \abs{\pl^\al w(x)} \ls \frac{1}{\bx^{1+|\al|}} \mbox{ for all } |\al|\le 3.
\end{align}

\begin{example}
	The Coulomb interaction $w(x)=\pm|x|^{-1}$ does not satisfy \eqref{eq:ass}.
	However, the Coulomb interaction regularized around the origin $w(x)= \pm\bx^{-1}$ satisfies \eqref{eq:ass}.
\end{example}

Now we state our main result.
\begin{theorem}\label{th:main}
	Let $d = 3$ and $3/2<\si<2$. 
	Assume that $w(x)$ satisfies \eqref{eq:ass}.
	Then, there exists $\ep_0>0$ such that the following holds:
	If the family of self-adjoint initial data $(\gah_0)_{\hb\in(0,1]}\subset \CX^\si_\hb$ satisfies $\sup_{\hb\in(0,1]}\No{\gah_0}_{\CX_\hb^\si} \le \ep_0$, 
	then, for any $\hb\in(0,1]$, there exists a unique global solution $\gah(t)\in C(\R_{\ge 0};\FS^1_\hb)$ to \eqref{eq:NLH} with $\gah(0)=\gah_0$ such that
	\begin{equation}\label{eq:main}
		\sup_{t\ge 0} \Ck{ \No{\rhh(\gah(t))}_{L^1_x} + \bt^3 \No{\rhh(\gah(t))}_{L^\I_x} }\ls \|\gah_0\|_{\CX_\hb^\si},
	\end{equation}
	where the implicit constant is independent of $\hb\in(0,1]$.
\end{theorem}

\begin{remark}
	In the main theorem, we did \textit{not} include the Coulomb interaction $w(x)=\pm|x|^{-1}$.
	Indeed, some convolution estimates (see Lemmas \ref{lem:potential deriv} and \ref{lem:potential deriv 2}) do not hold for the Coulomb interaction.
	However, the author believes that this is a technical issue. 
	Let $\ch(x)\in C_c^\I(\R^3)$ be a smooth cut-off function such that $\ch(x)\equiv1$ around the origin.
	Then, we can decompose
	\begin{align}
		w(x) = w_0(x)+w_1(x) := w(x)(1-\ch(x)) + w(x)\ch(x).
	\end{align}
	On the one hand, $w_0(x)$ satisfies \eqref{eq:ass}. Thus, it can be dealt with by the methods in this paper. 
	On the other hand, $w_1(x)$ is a short-range interaction; hence, in principle, we can expect that we can handle it as a perturbation (see \cite{Hadama Hong 2025}).
	This additional argument to include the Coulomb interaction will be clarified in future work.
\end{remark}

\begin{remark}
	Theorem \ref{th:main} claims the decay
	\begin{align}
		\No{\rhh(\gah(t))}_{L^\I_x} \ls \frac{1}{\bt^3},
	\end{align}
	where the implicit constant is independent of $\hb\in(0,1]$.
	The rate of decay seems natural because the free solution also decays at the same rate (see Lemma \ref{lem:free}).	
	However, the derivatives of the density decay faster than $|t|^{-3}$; for example, we obtain
	$$\No{\na \rhh(\gah(t))}_{L^\I_x}\ls \frac{1}{\bt^{4-\ep}},$$
	where $\ep>0$ is a small positive constant (see Proposition \ref{prop:a priori}).
\end{remark}

\begin{remark}\label{rmk:modified scattering}
	The solution obtained in Theorem \ref{th:main} exhibits modified scattering.
	Namely, there exists $\gah_+\in \FS^1_\hb$ such that
	\begin{equation}\label{eq:modified scattering}
		\lim_{t\to\I}\No{e^{i\Ps(t,-i\hb\na_x)} \U(t)^*\ga^\hb(t) \U(t) e^{-i\Ps(t,-i\hb\na_x)} -\gah_+ }_{\FS^2_\hb} = 0,
	\end{equation}
	where
	\begin{equation}
		\Ps(t,\xi) := \frac{1}{\hb}\int_0^t V(\ta,\ta\xi) d\ta, \quad V(t):= w \ast \rhh(\gah(t)).
	\end{equation}
	However, this paper could not remove the $\hbar$--dependence from the rate of convergence in \eqref{eq:modified scattering}.
	See Section \ref{subsec:proof} and Remark \ref{rmk:final}.
\end{remark}

\begin{remark}\label{rmk:new proof}
Fix $\hb=1$ and choose rank-one initial data $\ga_0 = |u_0\rg \lg u_0|$ with sufficiently nice and small $u_0$. Then, as a corollary of Theorem \ref{th:main} (and Remark \ref{rmk:modified scattering}), we obtain the modified scattering for the usual nonlinear Schr\"{o}dinger equation \eqref{eq:NLH1} with $u(0)=u_0$.
We can regard the argument in this paper as a new proof of the modified scattering for the usual long-range nonlinear Schr\"{o}dinger equation with a Hartree type nonlinearity. However, the proof is much more complicated than well-known proofs.
\end{remark}

\subsection{Importance of the problem}\label{subsec:importance}
To understand the importance of getting a uniform-in-$\hb$ smallness condition for the initial data, we need to review the semi-classical limit to the Vlasov equation
\begin{equation}\label{eq:Vlasov}
	\pl_t f + p \cdot (\na_q f) -(\na w \ast \rh(f)) \cdot \na_p f = 0, \quad f=f(t,q,p) :\R_{\ge 0}\times \R_q^3 \times \R_p^3 \to \R,
\end{equation}
where the density function is defined by
\begin{equation}
	\rh(f)(t,q) := \int_{\R^3} f(t,q,p) dp.
\end{equation}
It is well-known that the solution of the Hartree equation ``converges'' to the solution of the Vlasov equation as $\hb \to 0$, under various assumptions. More precisely, we have
\begin{equation}\label{eq:Wigner convergence}
	\Wig[\ga_0^\hb] \to f_0 \text{ as } \hb\to0 \implies \Wig[\gah(t)] \to f(t) \text{ as } \hb \to 0
\end{equation}
in appropriate topologies, where $\gah_0$ is the initial data of \eqref{eq:NLH}, $\ga^\hb(t)$ is a solution to \eqref{eq:NLH}, $f_0$ is the initial data of \eqref{eq:Vlasov}, and $f(t)$ is the solution to \eqref{eq:Vlasov}. Moreover, the Wigner transform is given by
\begin{equation}
	\Wig[\ga(t)](q,p) := \int_{\R^3} \ga \K{t,q+\frac{y}{2}, q-\frac{y}{2}} e^{-ipy/\hb}dy.
\end{equation}
This convergence is called the \textit{semi-classical limit}. Due to the abundance of research on this topic, the following list of references is far from exhaustive; for example, see \cite{Saffirio 2020a, Saffirio 2020b, Lafleche 2019, Lafleche 2021, Lafleche 2024, Chong et al 2022, Chong et al 2023, Chong et al 2024, Amour et al 2013a, Amour et al 2013b, Athanassoulis et al 2011, Benedikter et al 2014, Benedikter et al 2016 derivation, Benedikter et al 2016 semiclassical, Figalli et al 2012, Gasser et al 1998,Golse 2016, Golse Paul 2017, Golse Paul 2019, Lions Paul 1993, Graffi 2003, Lewin Sabin 2020, Markowich Mauser 1993, Pezzotti Pulvirenti 2009, Smith 2025,Lafleche Saffirio 2023, Han-Kwan Rousset 2025}.
Recently, the following ``commutative diagram'' was shown in \cite{Smith 2026, Hadama Hong 2025}:
\[
\begin{tikzcd}[row sep=large, column sep=large]
	\ga_0^\hb \mbox{ (initial data) }
	\arrow[r, "\textup{solve}"] 
	\arrow[d, "\hb\to0"'] 
	& \ga^\hb(t) \mbox{ (solution) }
	\arrow[r, "t \to \I"] 
	\arrow[d, "\hb\to0"'] 
	\arrow[dl, phantom, "\circlearrowleft" pos=0.45] 
	& \ga_+^\hb \mbox{ (scattering state) }
	\arrow[d, "\hb\to0"'] 
	\arrow[dl, phantom, "\circlearrowleft" pos=0.45] \\
	f_0 \mbox{ (initial data) }
	\arrow[r, "\textup{solve}"] 
	& f(t) \mbox{ (solution) }
	\arrow[r, "t \to \I"] 
	& f_+ \mbox{ (scattering state) }.
\end{tikzcd}
\]
This means that we have the convergence of the scattering states
\begin{equation}\label{eq:convergence of scattering states}
	\Wig[\ga_0^\hb] \to f_0 \text{ as } \hb\to0 \implies \Wig[\ga^\hb_+] \to f_+  \mbox{ as } \hb \to 0
\end{equation}
in appropriate topologies, where $\ga_+^\hb := \lim_{t \to \I} \U(-t)\ga^\hb(t) \U(t)$ and $f_+ := \lim_{t \to \I} \CU(-t) f(t)$. Here, we defined $\CU(t)f_0 := f_0(q-tp,p)$, that is, $f(t):= \CU(t)f_0$ is the solution to the free transport equation
\begin{equation}
	\pl_t f + p \cdot \na_q f = 0.
\end{equation}
In \cite{Hadama Hong 2025, Smith 2026}, the most essential tool for proving \eqref{eq:convergence of scattering states} is the estimate \eqref{eq:density function decay} with uniform-in-$\hbar$ smallness conditions for the initial data.
Indeed, if we do not have uniformity with respect to $\hbar$, then as we have already explained in Section \ref{subsec:intro}, we have 
$\gah_0 \to 0 \mbox{ as } \hb \to 0$.
Hence, under standard assumptions, we have $f_0 = \lim_{\hb\to0} \Wig[\gah_0] = 0$. 
This means that we can only include the trivial case on the classical side.

We have to note that this kind of correspondence of scattering states is known only when $w$ is a short-range interaction.
When $w$ is long-range, for example, Coulomb interaction $w(x)=\pm|x|^{-1}$ or Coulomb interaction regularized around the origin $w(x)=\pm \bx^{-1}$, we cannot expect scattering. 
However, we have modified scattering (see, for example, \cite{Nguyen You 2025, Hayashi Naumkin 1998, Hayashi Naumkin 1998 b, Hayashi et al 1998}).
Therefore, it is quite natural to ask whether the following ``commutative diagram'' is true or not when $w$ is long-range:
\[
\begin{tikzcd}[row sep=large, column sep=large]
	\ga_0^\hb \mbox{ (initial data) }
	\arrow[r, "\textup{solve}"] 
	\arrow[d, "\hb\to0"'] 
	& \ga^\hb(t) \mbox{ (solution) }
	\arrow[r, "t \to \I"] 
	\arrow[d, "\hb\to0"'] 
	\arrow[dl, phantom, "\circlearrowleft" pos=0.45] 
	& \ga_+^\hb \mbox{ (modified scattering state) }
	\arrow[d, "\hb\to0"'] 
	\arrow[dl, phantom, "\circlearrowleft" pos=0.45] \\
	f_0 \mbox{ (initial data) }
	\arrow[r, "\textup{solve}"] 
	& f(t) \mbox{ (solution) }
	\arrow[r, "t \to \I"] 
	& f_+ \mbox{ (modified scattering state) }.
\end{tikzcd}
\]
In other words, we would like to know whether the correspondence of the modified scattering states
\begin{equation}
	\Wig[\ga_0^\hb] \to f_0 \text{ as } \hb\to0 \implies \Wig[\ga^\hb_+] \to f_+ \mbox{ as } \hb \to 0
\end{equation}
holds or not. \textit{As an important first step to tackle this question, we prove the uniform estimate in this paper for long-range interaction $w$.}

\subsection{Difficulties of the problem}\label{subsec:difficulty}
First of all, we need to emphasize that the uniformity is not very trivial even in the free case.
When $\hb=1$, there is a lot of freedom in the choice of the norm $\|\cdot\|_\CX$ such that
\begin{equation}\label{eq:h=1}
	\sup_{t\ge 0} |t|^3 \No{\rh^1\K{\CU^1(t) \ga_0 \CU^1(t)^*}}_{L^\I_x} \ls \|\ga_0\|_{\CX}
\end{equation}
for the free case. For example, we can choose $\|\ga_0\|_\CX:= (2\pi)^3\|\bx^\si \ga_0 \bx^\si\|_{\FS^r}$ for any $\si>3/2$ and $r\in [1,\I]$ (see Lemma \ref{lem:free}).
Since simple scaling of \eqref{eq:h=1} implies
\begin{equation}\label{eq:h general}
	\sup_{t\ge 0} |t|^3 \No{\rhh\K{\U(t) \gah_0\U(t)^*}}_{L^\I_x} \ls (2\pi\hb)^{-3/r} \No{\bx^\si \gah_0 \bx^\si }_{\FS^r_\hb},
\end{equation}
we find that \eqref{eq:h general} is uniform-in-$\hbar$ only when $r=\I$.
In general, the optimality of the Schatten exponent is a necessary condition to obtain the uniform-in-$\hb$ estimates from the scaling;
hence we cannot waste anything about the Schatten exponent.

Moreover, in the nonlinear problem (even in the short-range case), a new difficulty emerges. If we use the Duhamel formula
\begin{align}\label{eq:Duhamel}
	&\text{(NLH) with } \gah(0)=\gah_0 \\
	&\Longleftrightarrow\gah(t) = \U(t)\ga_0^\hb\U(t)^* - \frac{i}{\hb} \int_0^t \U(t-\ta) \Bk{w\ast \rhh(\ga^\hb), \ga^\hb}(\ta) \U(\ta-t)d\ta,
\end{align}
then we find a $1/\hb$--factor in front of the Duhamel term.
Since $1/\hbar \to \I$ as $\hb\to0$, this is a very bad factor.
In principle, we can obtain additional $\hbar$ to cancel $1/\hbar$ from the commutator $\Dk{w\ast \rhh(\ga^\hb), \ga^\hb}$.
However, in practice, it is not easy at all to close the argument.
If we extract $\hb$ from the commutator, then we find another commutator term $\Dk{x/\hb, \ga^\hb}$ (see Lemma \ref{lem:commutator basic}).
To control an appropriate norm of $\Dk{x/\hb, \ga^\hb}$, if we use the Duhamel formula again, then we find a double commutator term $\Dk{x/\hb,\Dk{x/\hb,\gah(t)}}$. Hence, if we naively repeat using the Duhamel formula, we cannot close the argument.

Finally, we explain why we cannot apply the usual treatment for the nonlinear Schr\"{o}dinger equations with long-range nonlinearity.
The decomposition of the integral kernel of $e^{it\De/2}$
\begin{equation}\label{eq:decomposition of eitDe}
	\frac{e^{-|x-x'|^2/(2it)}}{(2\pi i t)^{3/2}} = \frac{1}{(2\pi it)^{3/2}} + \frac{1}{(2\pi it)^{3/2}} \K{ e^{-|x-x'|^2/(2it)}-1 } = A(x,x') + B(x,x')
\end{equation}
is one of the most basic but important arguments in the study of the long-range NLS (see, for example, \cite{Nguyen You 2025, Hayashi Naumkin 1998}). 
Since $A(x,x')$ is just a constant, we can easily deal with it.
Moreover, $B(x,x')$ gives us an additional decay from 
\begin{equation}\label{eq:additinal decay from the decomposition}
	\abs{B(x,x')}\le \frac{1}{|2\pi t|^{3/2}} \frac{|x-x'|^2}{2t}.
\end{equation}
However, when we work in the semi-classical regime, we need to replace all the $t$ in \eqref{eq:decomposition of eitDe} and \eqref{eq:additinal decay from the decomposition} by $t\hb$. Hence, additional decay coming from $B(x,x')$ is $|x-x'|^2/(2t\hb)$ and we find a $1/\hb$--factor here again.
Since the author does not know how we can deal with this bad factor, 
in this paper, we develop a new method to handle long-range interactions.

\subsection{Ideas and novelties of the proof}\label{subsec:idea}
\subsubsection{Review of the argument in \cite{Hadama Hong 2025}}
Hong and the author proved a result similar to Theorem \ref{th:main} for short-range interaction in \cite{Hadama Hong 2025}. We give a brief review of its ideas.

Let $u(t)=\U_V(t)\phi$ be the solution to the following linear Schr\"{o}dinger equation with potential $V(t,x)$:
\begin{equation}
	i\hb\pl_t u = -\frac{\hb^2}{2}\De_x u+ V(t,x)u, \quad u:\R_{\ge 0}\times\R^3 \to \C
\end{equation}
with initial condition $u(0) = \phi$.
As already explained, if we use the Duhamel formula \eqref{eq:Duhamel} directly, then it is rather difficult to deal with the bad factor $1/\hb$ of the right-hand side. Hence, we use the alternative formulation 
\begin{equation}\label{eq:f}
	\gah(t) = \U_{w \ast \rhh(\gah)}(t) \gah_0 \U_{w \ast \rhh(\gah)}(t)^*.
\end{equation}
Taking the density function of both sides, we obtain a closed equation for the density function $\rhh(\gah(t))$:
\begin{align}
	\rhh(\gah(t)) = \rhh\K{\U_{w \ast \rhh(\gah)}(t) \gah_0 \U_{w \ast \rhh(\gah)}(t)^*}.
\end{align}
Since we can recover the density operator $\gah(t)$ from its density function $\rhh(\gah(t))$ via the formula \eqref{eq:f}, all the arguments were reduced to the estimates for the density functions.

The crucial estimate in \cite{Hadama Hong 2025} is the following a priori estimate for the density function.
\textit{If $V=w \ast \rhh(\gah)$ belongs to a suitable function space and its norm is small}, then we obtain
\begin{equation}\label{eq:previous key}
	\No{\rhh\K{\U_V(t)\gah_0\U_V(t)^*}}_{L^r_x} \ls \frac{1}{\bt^{3/r'}} \K{\No{\bx^{\si/r'}\gah_0 \bx^{\si/r'}}_{\FS^r_\hb} + \No{\sd^{\si/r'}\gah_0 \sd^{\si/r'}}_{\FS^r_\hb}}
\end{equation}
for any $3/2<\si< 2$ and $r\in[1,\I]$.
The idea of the proof of \eqref{eq:previous key} is simple. 
First, we prove \eqref{eq:previous key} when $V=0$.
After that, we estimate 
\begin{align}
	&\No{\rhh\K{\U_V(t)\gah_0\U_V(t)^*}}_{L^r_x} 
	= \No{\rhh\K{\U(t)\W_V(t)\gah_0\W_V(t)^*\U(t)^*}}_{L^r_x}  \\
	&\ls \frac{1}{\bt^{3/r'}} \bigg(\No{\bx^{\si/r'} \W_V(t)\gah_0\W_V(t)^* \bx^{\si/r'}}_{\FS^r_\hb}  + \No{\sd^{\si/r'}\W_V(t)\gah_0\W_V(t)^* \sd^{\si/r'}}_{\FS^r_\hb}\bigg),
\end{align}
where $\W_V(t) := \U(-t) \U_V(t)$ is the (finite-time) wave operator.
Second, we can prove that the boundedness of wave operators
\begin{align}
	&\sup_{\hb\in(0,1]}\No{\bx^{\si/r'} \W_V(t) \bx^{-\si/r'}}_\CB \ls 1, \label{eq:previous wave operator boundedness 1}\\
	&\sup_{\hb\in(0,1]}\No{\sd^{\si/r'} \W_V(t) \sd^{-\si/r'}}_{\CB} \ls 1 \label{eq:previous wave operator boundedness 2}
\end{align}
holds \textit{if $V$ belongs to a suitable function space and its norm is small}. 
Therefore, we obtain an a priori estimate
\begin{align}
	\No{\rhh\K{\U_V(t)\gah_0\U_V(t)^*}}_{L^r_x}
	&\ls \frac{1}{\bt^{3/r'}} \No{\bx^{\si/r'} \W_V(t) \bx^{-\si/r'}}_{\CB}^2 \No{\bx^{\si/r'} \gah_0 \bx^{\si/r'}}_{\FS^r_\hb}\\
	&\quad + \frac{1}{\bt^{3/r'}} \No{\sd^{\si/r'}\W_V(t) \sd^{-\si/r'}}_\CB^2 \No{\sd^{\si/r'} \gah_0 \sd^{\si/r'}}_{\FS^r_\hb} \\
	&\ls \frac{1}{\bt^{3/r'}} \K{\No{\bx^{\si/r'}\gah_0 \bx^{\si/r'}}_{\FS^r_\hb} + \No{\sd^{\si/r'}\gah_0 \sd^{\si/r'}}_{\FS^r_\hb}}.
\end{align}
Since local well-posedness of \eqref{eq:NLH} is easy, once we get an a priori estimate, we can go through the standard bootstrap argument and prove the uniform estimate.

\subsubsection{Idea of this article}\label{subsubsec:idea}
Now we will explain the idea of this article.
First of all, we need to note that the argument in \cite{Hadama Hong 2025} works well because we can expect that $V=w \ast \rhh(\gah)$ belongs to a suitable function space if $w$ is a short-range interaction.
If $w$ is long-range, we can never expect that $V$ belongs to a suitable function space, and we cannot expect the boundedness of the wave operator with weight \eqref{eq:previous wave operator boundedness 1}.
However, by using so-called phase correction, we obtain a similar boundedness.
More precisely, instead of \eqref{eq:previous wave operator boundedness 1}, we prove
\begin{equation}\label{eq:previous wave operator boundedness 1'}
	\sup_{\hb\in(0,1]} \No{\bx^\si e^{i\Ps(t,-i\hb\na_x)} \W_V(t) \bx^{-\si}}_\CB \ls 1
\end{equation}
for a ``worse'' class of potential $V$, where $\Ps(t,\xi):\R_{\ge 0}\times \R^3 \to \R$ is an appropriate function and $e^{i\Ps(t,-i\hb\na_x)}$ is the phase correction.
The most significant contribution of this paper is that we prove \eqref{eq:previous wave operator boundedness 1'} by precise commutator calculations (see Proposition \ref{prop:wave operator boundedness}).

Since we use phase correction, we also need to prove 
\begin{align}\label{eq:free modified}
	&\No{\rhh\K{\U(t)e^{-i\Ps(t,-i\hb\na_x)}\gah_0 e^{i\Ps(t,-i\hb\na_x)}\U(t)^*}}_{L^r_x} \\
	&\quad \ls \frac{1}{\bt^{3/r'}} \K{\No{\bx^{\si/r'}\gah_0 \bx^{\si/r'}}_{\FS^r_\hb} + \No{\sd^{\si/r'}\gah_0 \sd^{\si/r'}}_{\FS^r_\hb}}
\end{align}
in order to apply the same argument as in \cite{Hadama Hong 2025}.
Estimate \eqref{eq:free modified} comes from an $L^1$--$L^\I$ dispersive estimate
\begin{align}\label{eq:Linfty L1}
	\No{ \U(t) e^{-i\Ps(t,-i\hb\na_x)} }_{L^1_x \to L^\I_x} \ls \frac{1}{|t\hb|^{3/2}}.
\end{align}
Initially, the author suspected that \eqref{eq:Linfty L1} was already known, but he could not find any appropriate references. Hence, we give a self-contained proof of \eqref{eq:Linfty L1} as an application of the stationary phase methods (see Proposition \ref{prop:perturbed dispersive}).

Connecting the above arguments, we have
\begin{align}
	&\No{\rhh\K{\U_V(t)\gah_0\U_V(t)^*}}_{L^r_x} \\
	&\quad =\No{\rhh\K{\U(t)e^{-i\Ps(t,-i\hb\na_x)}e^{i\Ps(t,-i\hb\na_x)}\W_V(t)\gah_0\W_V(t)^*e^{-i\Ps(t,-i\hb\na_x)}e^{i\Ps(t,-i\hb\na_x)}\textbf{}\U(t)^*}}_{L^r_x}  \\
	&\quad \ls \frac{1}{\bt^{3/r'}} \bigg(\No{\bx^{\si/r'} e^{i\Ps(t,-i\hb\na_x)} \W_V(t)\gah_0\W_V(t)^* e^{-i\Ps(t,-i\hb\na_x)}\bx^{\si/r'}}_{\FS^r_\hb} \\
	&\qquad \qquad \qquad + \No{\sd^{\si/r'}e^{i\Ps(t,-i\hb\na_x)}\W_V(t)\gah_0\W_V(t)^* e^{-i\Ps(t,-i\hb\na_x)}\sd^{\si/r'}}_{\FS^r_\hb}\bigg) \\
	&\quad \ls \frac{1}{\bt^{3/r'}} \bigg(\No{\bx^{\si/r'} e^{i\Ps(t,-i\hb\na_x)} \W_V(t) \bx^{-\si/r'}}_\CB^2 \No{\bx^{\si/r'} \gah_0 \bx^{\si/r'}}_{\FS^r_\hb} \\
	&\qquad \qquad \qquad + \No{\sd^{\si/r'}\W_V(t)\sd^{-\si/r'}}_\CB^2 \No{\sd^{\si/r'} \gah_0 \sd^{\si/r'}}_{\FS^r_\hb}\bigg) \\
	&\quad \ls \frac{1}{\bt^{3/r'}} \bigg(\No{\bx^{\si/r'} \gah_0 \bx^{\si/r'}}_{\FS^r_\hb}+ \No{\sd^{\si/r'} \gah_0 \sd^{\si/r'}}_{\FS^r_\hb}\bigg).
\end{align}

In \cite{Hadama Hong 2025}, the above argument was enough to apply the standard bootstrap arguments because we did not need to estimate $\na \rhh(\gah(t))$ nor $\na^2 \rhh(\gah(t))$. However, in this paper, we cannot avoid dealing with them to obtain \eqref{eq:previous wave operator boundedness 1'}. By the basic formula (see \eqref{eq:density deriv formula})
\begin{align}
	\pl_{x_j} \rhh(\U(t)A(t)\U(t)^*) = \rhh\K{\U(t) \Bk{\pl_{x_j}, A(t)} \U(t)^*} = \frac{1}{t}\rhh\K{\U(t)\Bk{\frac{x_j}{i\hb}, A(t)}\U(t)^*} 
\end{align}
for all $j=1,\dots,3$, we find that we cannot close the argument in terms of density functions. Namely, we need to estimate various commutators of density operators, and this requires a bit of hard calculation (see Sections \ref{sec:single} and \ref{sec:double}).
Moreover, since we are studying the global behavior of the solution, we need global-in-time inequalities; however, to the best of the author's knowledge, this type of commutator estimate is not written in the literature. We carefully compute them and finally get the desired bound.

Following the above argument, we derive the main theorem using the standard bootstrap method. 

\begin{remark}
	In Section \ref{subsec:difficulty}, we said we cannot close the commutator estimates if we naively use the Duhamel formula. In this paper, thanks to the identity \eqref{eq:f}, we can avoid the blowup of the number of commutator terms. 
\end{remark}

\begin{remark}
	We should emphasize that a very naive estimate
	\begin{equation}
		\No{\Bk{A,B}}_\CX \le \|AB\|_\CX + \|BA\|_\CX
	\end{equation}
	for a norm $\|\cdot \|_\CX$ is insufficient in our setting.
	Indeed, when we encounter one commutator, we have already lost one $\hbar$-factor. In other words, we do not encounter the commutator $[A,B]$ alone, but we find one commutator with one $1/\hbar$.
	What we need to estimate is not $[A,B]$, but  $\frac{1}{\hbar}[A,B]$.
	One can find this situation in the Duhamel formula \eqref{eq:Duhamel}.
\end{remark}

\subsubsection{Summary of main contributions}
Finally, we summarize the novelties of this article.
\begin{itemize}
	\item First of all, we provide the boundedness of the wave operator with weights \eqref{eq:previous wave operator boundedness 1'} (see Proposition \ref{prop:wave operator boundedness}). Such an estimate might be useful not only for the setting of this article but also for other areas. For example, we might be able to apply Proposition \ref{prop:wave operator boundedness} or its proof to the study of the modified scattering of more general nonlinear Schr\"{o}dinger equations.
	
	For the proof, we choose
\begin{equation}\label{eq:phase correction}
	\Ps(t,\xi):= \frac{1}{\hb}\int_0^t V(\ta,\ta\xi)d\ta
\end{equation}
as the phase correction. One can find this kind of choice of $\Ps$ in, for example, \cite{Nguyen You 2025}.
The advantage of $\Ps$ defined by \eqref{eq:phase correction} is that it explicitly includes $V$. Hence, we can compute the commutator 
$$\Bk{x,e^{i\Ps(t,-i\hb\na_x)}\W_V(t)}$$
precisely, and extract the cancellation between $e^{i\Ps(t,-i\hb\na_x)}$ and $\W_V(t)$.
	\item As explained in Section \ref{subsubsec:idea}, we cannot avoid estimating various kinds of commutators.
	The main (or practically the only) tool to estimate the commutator is the following inequality (see Lemma \ref{lem:commutator basic}):
	\begin{align}
		\No{\Bk{f(x), A}}_{\FS^r_\hbar} \le \No{\nabla f}_{\CF L^1} \No{\Bk{x,A}}_{\FS^r_\hb}.
	\end{align}
	Hence, we often need to estimate $\No{\rhh(\gah(t))}_{\CF L^1}$, not $\No{\rhh(\gah(t))}_{L^\I_x}$. 
	We give a simple but new estimate for the free propagator in Lemma \ref{lem:free Fourier}. It enables us to deal with the $\CF L^1$--norm, but with some $1/\hbar$--factor. However, fortunately, we sometimes find additional $\hbar$ in front of the $\CF L^1$--norm (see, for example, Lemma \ref{lem:key estimate}). 
	\item In Sections \ref{sec:single} and \ref{sec:double}, we give several global-in-time estimates of the commutators. To the best of the author's knowledge, none of these estimates are written in the literature. 
	\item We give an $L^1$--$L^\I$ dispersive estimate for the propagator with phase correction in Proposition \ref{prop:perturbed dispersive}. The author suspects it might be already known, but he was not able to find it.
\end{itemize}

\subsection{Structure of this paper}
This paper is organized as follows.
In Section \ref{sec:preliminaries}, we give basic definitions, notations, and elementary lemmas that will be used later.
In Section \ref{sec:L1LI}, we prove an $L^1$--$L^\I$ dispersive estimate for the modified propagator (Proposition \ref{prop:perturbed dispersive}). As a corollary, we give uniform-in-$\hbar$ dispersive estimates for the phase-corrected propagators without any potential $V(t,x)$ (Corollary \ref{cor:free}).
Moreover, we give a new bound for the $\CF L^1$--norm of the density function (Lemma \ref{lem:free Fourier}).  
In Section \ref{sec:wave}, we deal with the most important ingredient in this paper: the uniform-in-$\hb$ boundedness of the wave operators (Proposition \ref{prop:wave operator boundedness}). 
In Section \ref{sec:single}, we prove various single commutator estimates for the general Schatten--$r$ norm to control the derivatives of the density functions.
In Section \ref{sec:double}, we prove various double commutator estimates for the Hilbert--Schmidt norm.
In Section \ref{sec:proof}, first we prove the key a priori estimate. Then, we prove local well-posedness of \eqref{eq:NLH} and complete the proof of the main theorem by the standard bootstrap argument.

\section{Preliminaries}\label{sec:preliminaries}
In this section, we gather preliminaries. In Section \ref{subsec:definition and notation}, we give some necessary definitions and notations.
In Section \ref{subsec:definition of density functoin}, we give a rigorous definition of the density function $\rhh(\gah(t))$, which was formally given in the introduction. 
In Section \ref{subsec:useful identities}, we collect some useful identities, which are necessary to recover $\hbar$--loss coming from the Duhamel formula. 
In Section \ref{subsec:useful estimates}, we give elementary inequalities for linear operators on $L^2(\R^3)$. 

\subsection{Definitions and notations}\label{subsec:definition and notation}
\subsubsection{Basic notations}
Define the Japanese bracket by $\bx := (1+|x|^2)^{1/2}$.
We write $A \ls B$ if $A \le C B$, where $C>0$ is an $\hb$-independent constant.

\subsubsection{Function and operator spaces}
We write $C_c^\I(\R^N)$ for the space of all $C^\I$--functions with compact support. 
We write $\CS(\R^N)$ for the space of all rapidly decreasing functions (Schwartz functions).
Moreover, $L^{p,q}(\R^N)$ is the Lorentz space. For the definitions and basic properties of the Lorentz spaces, see, for example, \cite{Grafakos book}.
We define the Fourier--Lebesgue space $\CF L^p(\R^N)$ by $\|f\|_{\CF L^p} = \|\wh{f}\|_{L^p_\xi}$.

Let $\CB = \CB(L^2(\R^3))$ be the space of all bounded linear operators on $L^2(\R^3)$.
We define the Schatten--$r$ class by
\[
\FS^r := \Ck{A\in \CB(L^2(\R^3)) : A \text{ is compact and } \|A\|_{\FS^r} < \I},
\]
where
\[
\|A\|_{\FS^r}:= \Bck{\Tr\Dk{(A^*A)^{r/2}}}^{1/r} \mbox{ for } r \in [1,\I),
\quad \|A\|_{\FS^\I}:= \|A\|_{\CB}.
\]
When $r=2$, we have
\begin{align}\label{eq:Hilbert Schmidt}
	\|A\|_{\FS^2} = \|A(x,x')\|_{L^2_{x,x'}}.
\end{align}
See, for example, \cite{Simon book} and \cite[Appendix D]{Hytonen et al book} for more details on this material.
In this paper, we work in the semi-classical regime. Hence, we need to introduce the semi-classically scaled Schatten--$r$ norm
\begin{equation}
	\|A\|_{\FS^r_\hb} := (2\pi\hb)^{3/r} \|A\|_{\FS^r}.
\end{equation}

\subsubsection{Specific functions and operators}
There are some different definitions of the Fourier transform $\CF$ and the inverse Fourier transform $\CF^{-1}$, but in this paper, we use the following definitions:
	\begin{align}
		&\CF[f](\xi) = \wh{f}(\xi) = \frac{1}{(2\pi)^{3/2}}\int_{\R^3} e^{-ix\xi} f(x)dx, 
		\quad \CF^{-1}[f](x) = \check{f}(x) = \frac{1}{(2\pi)^{3/2}}\int_{\R^3} e^{ix\xi} f(\xi)d\xi.
	\end{align}
We define $\U(t):= e^{it\hb\De/2}=\CF^{-1} e^{-it\hb|\xi|^2/2} \CF$ as the propagator of the free Schr\"{o}dinger equation. Moreover, $u(t):= \U_V(t,s)\phi$ for $t, s \ge 0$ is the solution to 
\[\begin{dcases}
	&i\hb\pl_t u = -\frac{\hb^2}{2}\De_x u + V(t,x)u, \quad u:\R_{\ge 0}\times \R^3 \to \C, \\
	&u(s)=\phi.
\end{dcases}\]
We always assume that $V$ is real-valued.
For simplicity of notation, we write $\U_V(t):= \U_V(t,0)$.
The (finite-time) wave operator, which is one of the most important tools in this paper, is defined by $\W_V(t) := \U(-t)\U_V(t)$.
We write $\J(t) := x + it\hb\na$. 
A direct calculation shows 
\begin{equation}\label{eq:J}
	\J(t)=\U(t)x\U(-t) = \M(t)it\hb\na \M(-t),
\end{equation}
where $\M(t) := e^{-|x|^2/(2it\hb)}$.
Hence, we define 
\begin{equation}\label{eq:f(J)}
\lg\J(t)\rg^s := \U(t)\bx^s\U(-t) = \M(t)\lg t\hb\na \rg^s\M(-t)
\end{equation}
for any $s \in \R$.
The vector field $\J(t)$ is widely used in the literature.  See, for example, \cite{Hayashi Naumkin 1998, Hayashi Naumkin 1998 b, Hayashi Tsutsumi 1987, Hayashi et al 1998}.

\subsection{Density functions of linear operators}\label{subsec:definition of density functoin}
In the introduction, we formally defined $\rhh(A)$; however, we provide a rigorous definition here.
We called the definition $\rhh(A)(x):= (2\pi\hb)^3 A(x,x)$ formal because the diagonal set $\{(x,y)\in\R^{3+3}\mid x=y\}$ is a null set of $\R^{3+3}$.
In this paper, we consider $\rhh(A)$ only for trace class operators $A \in \FS^1$. 
Then, we can use the singular value decomposition
\begin{align}
	A = \sum_{n=1}^\I a_n|u_n\rg \lg v_n|,
\end{align}
where $(a_n)_{n=1}^\I \in \ell^1$ satisfies $a_1 \ge a_2 \ge \cdots \ge a_n \downarrow 0$, and $(u_n)_{n=1}^\I$, $(v_n)_{n=1}^\I$ are orthonormal systems in $L^2_x$. 
We define
\begin{align}\label{eq:densty rigorous definition}
	\rhh(A)(x) := (2\pi\hb)^3\sum_{n=1}^\I a_n u_n(x)\ov{v_n(x)}.
\end{align}
It is easy to see that $\rhh(A)$ is a well-defined mathematical object in $L^1_x$ because 
\begin{align}\label{eq:density L1 estimate}
	\|\rhh(A)\|_{L^1_x} \le  (2\pi\hb)^3\sum_{n=1}^\I |a_n|\|u_n v_n\|_{L^1_x} \le  (2\pi\hb)^3\|(a_n)_{n=1}^\I\|_{\ell^1} < \I.
\end{align}
We can define $\rhh(A)$ when $A$ is not of the trace class (see Remark \ref{rmk:general density}).
\begin{remark}
	The trivial bound 
	\begin{align}\label{eq:density L1 bound}
		\|\rhh(A)\|_{L^1_x} \le \|A\|_{\FS^1_\hb}
	\end{align}
	follows from the above argument, because $\|A\|_{\FS^1_\hb} =  (2\pi\hb)^3\|(a_n)_{n=1}^\I\|_{\ell^1}$ (see \cite{Simon book, Hytonen et al book}).
\end{remark}

Moreover, this rigorous definition is compatible with the formal definition because
\begin{align}
	A(x,x') = \sum_{n=1}^\infty a_n u_n(x)\ov{v_n(x')}
	\quad \mbox{and}\quad  A(x,x) ``=\mbox{''} \sum_{n=1}^\infty a_n u_n(x)\ov{v_n(x)}.
\end{align}

We sometimes use the following property of the density function, which follows from a direct calculation with the rigorous definition.
\begin{lemma}\label{lem:duality}
	For any $f\in C_c^\I(\R^3)$ and $A \in \FS^1_\hb$, it holds that
	\begin{align}
		\int_{\R^3} f(x) \rhh(A)(x) dx = (2\pi\hb)^3 \Tr\Big[f(x)A\Big].
	\end{align}
\end{lemma}

\subsection{Basic identities}\label{subsec:useful identities}
\subsubsection{MDFM decomposition}
We sometimes use the so-called MDFM decomposition 
\begin{equation}\label{eq:MDFM}
	\U(t) = \M(t) \D(t) \CF \M(t),
\end{equation}
where 
\begin{align}
	&\M(t) = e^{-|x|^2/(2it\hb)}, \quad 
	(\D(t)u)(x):= \frac{1}{(it\hb)^{3/2}}u\K{\frac{x}{t\hb}}.
\end{align}
We can obtain \eqref{eq:MDFM} by a direct calculation.
It has been used in a large body of literature. For example, see \cite{Hayashi Naumkin 1998, Hayashi Naumkin 1998 b, Hayashi Tsutsumi 1987, Hayashi et al 1998}.

\subsubsection{Derivatives of the density function}
By the rigorous definition of $\rhh(A)$ and the Leibniz rule, we obtain $\pl_{x_j} \rhh(A) = \rhh\K{\Dk{\pl_{x_j}, A}}$.
In particular, we have
\begin{align}\label{eq:density deriv formula 0}
\pl_{x_j} \rhh\K{\U(t)\gah_0\U(t)^*} = \rhh\K{\U(t)\Bk{\pl_{x_j}, \gah_0}\U(t)^*}.
\end{align}
However, we have another expression
\begin{align}\label{eq:density deriv formula}
	\pl_{x_j} \rhh\K{\U(t)\gah_0\U(t)^*} = \frac{1}{t}\rhh\K{\U(t)\Bk{\frac{x_j}{i\hb}, \gah_0}\U(t)^*}.
\end{align}

\begin{proof}[Proof of \eqref{eq:density deriv formula}]
Note that 
\begin{align}
	\rhh\K{\U(t)\gah_0 \U(t)^*}(x) = \hb^3 \iint_{\R^3\times \R^3} e^{ix(\xi-\xi')} e^{-it\hbar(|\xi|^2-|\xi'|^2)/2} \wh{\gah_0}(\xi,\xi') d\xi d\xi',
\end{align}
where $\wh{\gah_0}:= \CF \gah_0 \CF^{-1}$.
Hence, the integration by parts implies
\begin{align}
	&\pl_{x_j} \rhh\K{\U(t)\gah_0 \U(t)^*}(x)
	= \hb^3 \iint_{\R^3\times\R^3} e^{ix(\xi-\xi')} i(\xi_j-\xi_j') e^{-it\hb(|\xi|^2-|\xi'|^2)/2} \wh{\gah_0}(\xi,\xi') d\xi d\xi' \\
	&\quad = \frac{-\hb^2}{t} \iint_{\R^3\times\R^3} e^{ix(\xi-\xi')} \K{\pl_{\xi_j} + \pl_{\xi'_j}} \K{ e^{-it\hb(|\xi|^2-|\xi'|^2)/2}} \wh{\gah_0}(\xi,\xi') d\xi d\xi' \\
	&\quad = \frac{\hb^2}{t} \iint_{\R^3\times\R^3}  e^{-it\hb(|\xi|^2-|\xi'|^2)/2} \K{\pl_{\xi_j} + \pl_{\xi'_j}}\K{ e^{ix(\xi-\xi')} \wh{\gah_0}(\xi,\xi') }d\xi d\xi'\\
	&\quad = \frac{\hb^2}{t} \iint_{\R^3\times \R^3}  e^{-it\hb(|\xi|^2-|\xi'|^2)/2} e^{ix(\xi-\xi')}\K{\pl_{\xi_j} + \pl_{\xi'_j}}\K{  \wh{\gah_0}(\xi,\xi') }d\xi d\xi'\\
	&\quad = \frac{1}{t} \rhh\K{\U(t)\Bk{\frac{x_j}{i\hb},\gah_0} \U(t)^*}(x).
\end{align}
\end{proof}

\subsubsection{Commutator identities}
We often use the Jacobi identity:
\begin{equation}\label{eq:Jacobi}
	[A,[B,C]] + [B,[C,A]] + [C,[A,B]] = 0.
\end{equation}
We can prove \eqref{eq:Jacobi} by a direct computation.
We use the following commutator formulas of the wave operator:
\begin{lemma}[Lemma 3.3 in \cite{Hadama Hong 2025}]\label{lem:commutator wave}
Let $V(t,x)$ be real-valued. For any $f(\xi)$, the commutator formula
\begin{align}\label{eq:commutator f W}
	\Bk{f(-i\hb\nabla), \W_V(t)} = -\frac{i}{\hb} \W_V(t) \int_0^t \U_V(\ta)^* \Big[f(-i\hb\na), V(\ta)\Big] \U_V(\ta)d\ta
\end{align}
holds. In particular, we have
\begin{align}\label{eq:commutator nabla W}
	&\Bk{\nabla, \mathcal{W}_V^\hbar(t)}
	=-\frac{i}{\hb}\W_V(t)\int_0^t \U_V(\ta)^*(\nabla V)(\ta)\U_V(\ta) d\ta.
\end{align}
Moreover, it holds that
	\begin{align}\label{eq:commutator x W}
		&\Bk{x, \mathcal{W}_V^\hbar(t)}
		=\W_V(t) \int_0^t \U_V(\ta)^* \ta(\nabla V)(\ta)\U_V(\ta)d\ta.
	\end{align}
\end{lemma}

\begin{remark}
	In the proof in \cite{Hadama Hong 2025}, the proof for \eqref{eq:commutator x W} and \eqref{eq:commutator nabla W} is only given. However, the same proof works for the more general formula \eqref{eq:commutator f W}.
\end{remark}

To obtain the commutator estimates, we need the following lemmas.
\begin{lemma}\label{lem:commutator nabla W}
	Let $V(t,x)$ be real-valued and set $\gah(t):= \U_V(t)\gah_0\U_V(t)^*$. Then, we have
\begin{equation}\label{eq:nabla single commutator}
	\begin{aligned}
		&\Bk{\pl_{x_j},\W_V(t)\gah_0 \W_V(t)^*}
		= \W_V(t) \Bk{\pl_{x_j},\ga_0^\hb}\W_V(t)^* \\
		&\qquad 
		-\frac{i}{\hb} \W_V(t) \int_0^t \U_V(\ta)^* \Bk{\pl_{x_j} V(\ta), \ga^\hb(\ta)} \U_V(\ta) d{\color{red}\ta} \W_V(t)^*
	\end{aligned}
\end{equation}
and
\begin{align}\label{eq:nabla double commutator}
\begin{aligned}
	&\Bk{\pl_{x_k},\Bk{\pl_{x_j},\W_V(t)\gah_0 \W_V(t)^*}} 
	= \W_V(t)\Bk{\pl_{x_k}, \Bk{\pl_{x_j}, \gah_0 }} \W_V(t)^* \\
	&\qquad - \frac{i}{\hb} \W_V(t)\int_0^t \U_V(\ta)^* \Bk{\pl_{x_j} V(\ta), \Bk{\pl_{x_k}, \gah(\ta)}} \U_V(\ta)d\ta\W_V(t)^* \\
	&\qquad - \frac{i}{\hb} \W_V(t)\int_0^t \U_V(\ta)^* \Bk{\pl_{x_k} V(\ta), \Bk{\pl_{x_j}, \gah(\ta)}} \U_V(\ta)d\ta\W_V(t)^*\\
	&\qquad - \frac{i}{\hb} \W_V(t)\int_0^t \U_V(\ta)^* \Bk{\pl_{x_k,x_j}^2 V(\ta), \gah(\ta)} \U_V(\ta)d\ta\W_V(t)^*
\end{aligned}
\end{align}
for all $j,k=1,\dots, 3$.
\end{lemma}

\begin{lemma}\label{lem:commutator x W}
	Let $V(t,x)$ be real-valued and set $\gah(t):= \U_V(t)\gah_0\U_V(t)^*$. Then, we have
\begin{equation}\label{eq:weight single commutator}
	\begin{aligned}
		&\Dk{\frac{x_j}{i\hb},\W_V(t)\gah_0 \W_V(t)^*}
		= \W_V(t) \Dk{\frac{x_j}{i\hb},\ga_0^\hb}\W_V(t)^* \\
		&\qquad 
		-\frac{i}{\hb} \W_V(t) \int_0^t \U_V(\ta)^* \Bk{\ta \pl_{x_j} V(\ta), \ga^\hb(\ta)} \U_V(\ta) d\ta \W_V(t)^*
	\end{aligned}
\end{equation}
and
\begin{align}\label{eq:weight double commutator}
	\begin{aligned}
		&\Dk{\frac{x_k}{i\hb},\Dk{\frac{x_j}{i\hb},\W_V(t)\gah_0 \W_V(t)^*}} 
		= \W_V(t)\Dk{\frac{x_k}{i\hb}, \Dk{\frac{x_j}{i\hb}, \gah_0 }} \W_V(t)^* \\
		&\qquad - \frac{1}{\hb^2} \W_V(t)\int_0^t \U_V(\ta)^* \Dk{\ta \pl_{x_j} V(\ta), \Dk{\J_k(\ta), \gah(\ta)}} \U_V(\ta)d\ta\W_V(t)^* \\
		&\qquad - \frac{1}{\hb^2} \W_V(t)\int_0^t \U_V(\ta)^* \Dk{\ta \pl_{x_k} V(\ta), \Dk{\J_j(\ta), \gah(\ta)}} \U_V(\ta)d\ta\W_V(t)^*\\
		&\qquad - \frac{i}{\hb} \W_V(t)\int_0^t \U_V(\ta)^* \Bk{\ta^2 \pl_{x_k,x_j}^2 V(\ta), \gah(\ta)} \U_V(\ta)d\ta\W_V(t)^*.
	\end{aligned}
\end{align}
\end{lemma}

\begin{lemma}\label{lem:commutator mixed}
	Let $V \in C(\R_{\ge 0}; W^{2,\I}_x)$ and $\gah(t):= \U_V(t)\gah_0\U_V(t)^*$. Then, we have
	\begin{align}\label{eq:mixed double commutator}
		\begin{aligned}
			&\Dk{\pl_{x_k},\Dk{\frac{x_j}{i\hb},\W_V(t)\gah_0 \W_V(t)^*}} 
			= \W_V(t)\Dk{\pl_{x_k}, \Dk{\frac{x_j}{i\hb}, \gah_0 }} \W_V(t)^* \\
			&\qquad - \frac{i}{\hb} \W_V(t)\int_0^t \U_V(\ta)^* \Dk{\ta \pl_{x_j} V(\ta), \Dk{\pl_{x_k}, \gah(\ta)}} \U_V(\ta)d\ta\W_V(t)^* \\
			&\qquad - \frac{1}{\hb^2} \W_V(t)\int_0^t \U_V(\ta)^* \Dk{\pl_{x_k} V(\ta), \Dk{\J_j(\ta), \gah(\ta)}} \U_V(\ta)d\ta\W_V(t)^*\\
			&\qquad - \frac{i}{\hb} \W_V(t)\int_0^t \U_V(\ta)^* \Bk{\ta \pl_{x_k,x_j}^2 V(\ta), \gah(\ta)} \U_V(\ta)d\ta\W_V(t)^*.
		\end{aligned}
	\end{align}
\end{lemma}

\begin{proof}[Proof of Lemmas \ref{lem:commutator nabla W}, \ref{lem:commutator x W}, and \ref{lem:commutator mixed}]
	We can give a rigorous proof of Lemmas \ref{lem:commutator nabla W}, \ref{lem:commutator x W}, and \ref{lem:commutator mixed} by an elementary but complicated algebraic calculation with Lemma \ref{lem:commutator wave}.
	However, we show formal but more intuitive calculations here. We only prove Lemma \ref{lem:commutator nabla W} because the same proof works for Lemmas \ref{lem:commutator x W} and \ref{lem:commutator mixed}.
	
	First, note that
	\begin{align}
		\begin{dcases}
			&i\hb\pl_t \W_V(t) = \U(t)^*V(t)\U(t)\W_V(t),  \\
			&\W_V(0) = \Id_{L^2_x}.
		\end{dcases}
	\end{align}
	Hence we have
	\begin{align}
		\begin{dcases}
		&i\hb\pl_t\W_V(t)\gah_0\W_V(t)^* = \Bk{\U(t)^*V(t)\U(t), \W_V(t)\gah_0\W_V(t)^*}, \\
		&\W_V(0)\gah_0\W_V(0)^* = \gah_0.
		\end{dcases}
	\end{align}
	Let $A(t) := \Bk{\pl_{x_j}, \W_V(t)\gah_0\W_V(t)^*}$ and $\gah(t):= \U_V(t)\gah_0\U_V(t)^*$.
	Applying the Jacobi identity \eqref{eq:Jacobi} for the commutator, we have
	\begin{align}
		i\hb\pl_t A(t) 
		&= \Bk{\pl_{x_j}, \Bk{\U(t)^* V(t) \U(t), \W_V(t)\gah_0\W_V(t)^*} } \\
		&=  \Bk{\U(t)^* V(t) \U(t), \Bk{\pl_{x_j}, \W_V(t)\gah_0\W_V(t)^*} } \\
		&\quad +  \Bk{\Bk{\pl_{x_j}, \U(t)^* V(t) \U(t)}, \W_V(t)\gah_0\W_V(t)^*}\\
		&=\Bk{\U(t)^* V(t) \U(t), A(t)}+\U(t)^* \Bk{\pl_{x_j}V(t), \gah(t)} \U(t).
	\end{align}
	Therefore, by the Duhamel formula, we obtain \eqref{eq:nabla single commutator}.
	
	Next, we prove \eqref{eq:nabla double commutator}. 
	Let $B(t):= \Bk{\pl_{x_k}, \Bk{\pl_{x_j}, \W_V(t)\gah_0\W_V(t)^*}}=\Bk{\pl_{x_k}, A(t)}$.
	Then, by the Jacobi identity \eqref{eq:Jacobi} for the commutator, we have
	\begin{align}
		i\hb\pl_t B(t) &= \Bk{\pl_{x_k}, i\hb\pl_t A(t)} \\
		&= \Bk{\pl_{x_k}, \Bk{\U(t)^* V(t) \U(t), A(t)}}+ \Bk{\pl_{x_k}, \U(t)^* \Bk{\pl_{x_j}V(t), \gah(t)} \U(t)} \\
		&= \Bk{\U(t)^* V(t) \U(t), \Bk{\pl_{x_k}, A(t)}}+\Bk{\Bk{\pl_{x_k}, \U(t)^* V(t) \U(t)}, A(t)}\\
		&\qquad + \U(t)^* \Bk{\pl_{x_k}, \Bk{\pl_{x_j}V(t), \gah(t)} }\U(t) \\
		&= \Bk{\U(t)^* V(t) \U(t), B(t)}+\U(t)^*\Bk{\pl_{x_k} V(t), \Bk{\pl_{x_j}, \gah(t)}} \U(t)\\
		&\qquad + \U(t)^* \Bk{\pl_{x_j}V(t), \Bk{\pl_{x_k}, \gah(t)} }\U(t)
		+ \U(t)^* \Bk{\pl^2_{x_j,x_k}V(t), \gah(t)}\U(t).
	\end{align}
	Therefore, the Duhamel formula implies \eqref{eq:nabla double commutator}.	
\end{proof}

\subsubsection{A variant of the Duhamel formula}
By the Duhamel principle, we have
\begin{align}
	A(t,s) &:= \U_V(t,s) A \U_V(s,t)\\
	 &= \U(t-s) A \U(s-t) -\frac{i}{\hb} \int_s^t \U(t-\ta)\Bk{V(\ta), A(\ta,s)}\U(\ta-t) d\ta.
\end{align}
However, we have another useful expression:
\begin{equation}\label{lem:Duhamel}
	\begin{aligned}
		A(t,s) &= \U(t-s)A\U(s-t) \\
		&\qquad  -\frac{i}{\hb} \int_s^t \U_V(t,\ta) \Bk{V(\ta), \U(\ta-s) A \U(s-\ta)} \U_V(\ta,t) d\ta.
	\end{aligned}
\end{equation}

\begin{proof}[Proof of \eqref{lem:Duhamel}]
	Note that we have the following two representations
	\begin{align*}
		&\U_V(t,s) = \CU^\hb(t-s) - \frac{i}{\hb} \int_s^t \U(t-\ta) V(\ta) \U_V(\ta,s) d\ta, \\
		&\U_V(t,s) = \CU^\hb(t-s) - \frac{i}{\hb} \int_s^t \CU_V^\hb(t,\ta) V(\ta) \CU^\hb(\ta-s) d\ta.
	\end{align*}
	Let $\wt{A}(t) := \U(t)A \U(t)^*$.
	Then, we obtain
	\begin{align*}
		\CU_V^\hb(t,s)A\U_V(s,t)
		&= \wt{A}(t-s)
		-\frac{i}{\hb} \int_s^t \CU_V^\hb(t,\ta)V(\ta)\wt{A}(\ta-s) \CU^\hb(\ta-t) d\ta \\
		&\quad - \frac{i}{\hb}  \int_t^s \CU^\hb(t-\ta')\wt{A}(\ta'-s) V(\ta')\CU_V^\hb(\ta',t)d\ta' \\
		&\quad - \frac{1}{\hb^2} \int_s^t d\ta \int_t^s d\ta' \CU_V^\hb(t,\ta)V(\ta)\wt{A}(\ta-s) \CU^\hb(\ta-\ta') V(\ta') \CU_V^\hb(\ta',t) \\
		&= \wt{A}(t-s)
		-\frac{i}{\hb} \int_s^t \CU_V^\hb(t,\ta)V(\ta)\wt{A}(\ta-s) \CU^\hb(\ta-t) d\ta \\
		&\quad - \frac{i}{\hb}  \int_t^s \CU^\hb(t-\ta')\wt{A}(\ta'-s) V(\ta')\CU_V^\hb(\ta',t)d\ta' \\
		&\quad - \frac{1}{\hb^2} \int_s^t d\ta \int_t^\ta d\ta' \CU_V^\hb(t,\ta)V(\ta)\wt{A}(\ta-s) \CU^\hb(\ta-\ta') V(\ta') \CU_V^\hb(\ta',t) \\
		&\quad - \frac{1}{\hb^2} \int_t^s d\ta' \int_{\ta'}^t d\ta \CU_V^\hb(t,\ta)V(\ta)\CU^\hb(\ta-\ta') \wt{A}(\ta'-s) V(\ta') \CU_V^\hb(\ta',t) \\
		&= \wt{A}(t-s) - \frac{i}{\hb} \int_s^t \CU_V^\hb(t,\ta)\Bk{V(\ta), \wt{A}(\ta-s)}\CU_V^\hb(\ta,t) d\ta.
	\end{align*}
\end{proof}	

\subsection{Basic estimates}\label{subsec:useful estimates}

\subsubsection{Elementary inequalities of Schatten class operators}
The following inequality is well-known in functional analysis:
\begin{align}
	\|AB\|_{\CB} \le \|A\|_{\CB} \|B\|_{\CB},
\end{align}
but this is the special case of the H\"{o}lder inequality for the Schatten norm:
\begin{align}\label{eq:Holder}
	\|AB\|_{\FS_\hb^r} \le \|A\|_{\FS^{r_1}_\hb} \|B\|_{\FS^{r_2}_\hb},
\end{align}
where $r,r_1,r_2\in[1,\I]$ satisfy $1/r=1/r_1+1/r_2$. See, for example, \cite{Simon book} and \cite[Appendix D]{Hytonen et al book}.
Moreover, the following Kato--Seiler--Simon inequality (see \cite{Seiler Simon 1975} and \cite[Theorem 4.1]{Simon book}) is useful:
For all $r \in [2,\I]$, it holds that
\begin{align}\label{eq:Kato Seiler Simon}
	\|f(x)g(-i\hb\na)\|_{\FS^r_\hb} \ls \|f\|_{L^r_x} \|g\|_{L^r_x}.
\end{align}
Finally, we mention the triangle inequality just in case:
\begin{align}\label{eq:triangle inequality}
	(2\pi\hb)^3|\Tr(A)| \le \|A\|_{\FS^1_\hb}.
\end{align}

\subsubsection{Commutator estimates}
We often use the commutator estimates to recover the loss of $\hbar$ by the Duhamel formula.
The following lemma is one of the most basic tools in this paper:
\begin{lemma}\label{lem:commutator basic}
	The commutator estimate
	\begin{align}\label{eq:commutator basic}
		\No{\Bk{f(x), A}}_{\FS^r_\hb} \le \No{\na f}_{\CF L^1} \No{\Bk{x, A}}_{\FS^r_\hb}
	\end{align}
	holds for any $r\in[1,\I]$.
	Moreover, we can improve the estimate when $r=2$, that is,
	\begin{align}\label{eq:commutator basic 2}
		\No{\Bk{f(x), A}}_{\FS^2_\hb} \le \|\na f\|_{L^\I_x} \No{\Bk{x, A}}_{\FS^2_\hb}.
	\end{align}
\end{lemma}

\begin{remark}\label{rmk:weight}
	Note that $\|\na \bx^s\|_{\CF L^1} <\I$ if $s < 1$.
	Hence, we have
	\begin{align}\label{eq:commutator bx s}
		\No{\Bk{\bx^s, A}}_{\FS^r_\hb} \ls_s \No{\Bk{x, A}}_{\FS^r_\hb}
	\end{align}
	for any $s<1$ and $r\in[1,\I]$.
\end{remark}

\begin{proof}
	By the fundamental theorem of calculus, we have
	\begin{align}
		\Bk{f, A}(x,x') &= \K{f(x)-f(x')} A(x,x') = \int_0^1 (\na f)((1-\th)x + \th x') d\th (x-x') A(x,x') \\
		&= \frac{1}{(2\pi)^{3/2}}\int_0^1 \int_{\R^3} \wh{\na f}(\xi)e^{i((1-\th)x + \th x')\xi}  \Bk{x,A}(x,x')d\xi d\th.
	\end{align}
	Therefore, we obtain
	\begin{align}
		\No{\Bk{f(x),A}}_{\FS^r_\hb} &\le \int_{\R^3} \abs{\wh{\na f}(\xi)}
		\sup_{\theta \in [0,1]}\No{e^{i(1-\th)x\xi}\Bk{x, A}(x,x') e^{i\th x'\xi}}_{\FS^r_\hb}  d\xi  \\
		&\le \big\|\wh{\na f}\big\|_{L^1_\xi} \No{\Bk{x,A}}_{\FS^r_\hb}.
	\end{align}
	Moreover, when $r=2$, we have
	\begin{align}
		\No{\Bk{f(x),A}}_{\FS^2_\hb}
		&\le (2\pi\hb)^{3/2} \No{\int_0^1 (\na f)((1-\th)x + \th x') d\th (x-x') A(x,x')}_{L^2_{x,x'}}  \\
		&\le \No{\na f}_{L^\I_x} (2\pi\hb)^{3/2} \No{(x-x') A(x,x')}_{L^2_{x,x'}}
		=  \No{\na f}_{L^\I_x}  \No{\Bk{x,A}}_{\FS^2_\hb}.
	\end{align}
\end{proof}

\subsubsection{``Sobolev embedding'' for the density function}
The pointwise-in-$t$ boundedness of density functions comes from the following lemma:
\begin{lemma}\label{lem:Sobolev}
	Let $d=3$ and $\si>3/2$. 
	Then, for any $r\in[1,\I]$, we have
	\begin{align}\label{eq:Sobolev}
	 	&\|\rhh(A)\|_{L^r_x} \ls \|\sd^{\si/r'}A\sd^{\si/r'}\|_{\FS^r_\hb}.
	\end{align}
\end{lemma}

\begin{remark}\label{rmk:general density}
	First, we can prove \eqref{eq:Sobolev} for all $A\in \FS^1_\hb$.
	Since $A \mapsto \rhh(A)$ is linear and $\sd^{-\si/r'}\FS^1_\hb\sd^{-\si/r'}$ is dense in $\FS^r_\hb$, we can uniquely define the density $\rhh(A)$ for all $A \in \sd^{-\si/r'}\FS^r_\hb\sd^{-\si/r'}$ and the bound \eqref{eq:Sobolev} holds for all $A \in \sd^{-\si/r'}\FS^r_\hb\sd^{-\si/r'}$.
\end{remark}

\begin{remark}\label{rmk:sobolev scaling}
	Note that \eqref{eq:Sobolev} can be reduced to the case when $\hb=1$ by scaling.
	Assume that we have \eqref{eq:Sobolev} when $\hb = 1$.
	Define the unitary operator by $T_\hb f(x):= \hb^{-3/2} f(x/\hb)$.
	Let $A \in \FS^r$ and define $B:= T_\hb^{-1} A T_\hb$.
	Then, we have
	\begin{equation}
		\rhh(A)(x) = \rh^1(B)(x/\hb).
	\end{equation}
	Hence, we have
	\begin{equation}
		\|\rhh(A)\|_{L^r_x} = \hb^{3/r} \|\rh^1(B)\|_{L^r_x} \ls \hb^{3/r} \|\lg\na\rg^{\si/r'}B\lg\na\rg^{\si/r'}\|_{\FS^r}.
	\end{equation}
	Since
	\begin{equation}
		\lg\na\rg = T_\hb^{-1} \sd T_\hb, 
	\end{equation}
	we have
	\begin{equation}
		\hb^{3/r} \|\lg\na\rg^{\si/r'}B\lg\na\rg^{\si/r'}\|_{\FS^r} = \hb^{3/r}\|\sd^{\si/r'} A \sd^{\si/r'}\|_{\FS^r} \sim  \|\sd^{\si/r'} A \sd^{\si/r'}\|_{\FS^r_\hb}. 
	\end{equation}
	Collecting the above argument, we obtain \eqref{eq:Sobolev} for general $\hb \in (0,1]$.
\end{remark}

\begin{proof}
	By Remark \ref{rmk:sobolev scaling}, we can assume $\hb=1$.
	Let $f \in C_c^\I(\R^3)$, and define $f_0,f_1$ by $f_0:= |f|^{1/2}$ and $f=f_0f_1$.
	Then, by Lemma \ref{lem:duality}, \eqref{eq:triangle inequality}, cyclicity of the trace, and \eqref{eq:Holder}, we have
	\begin{align}
	&\abs{\int_{\R^3} f(x) \rh^1(A)(x)dx} 	= (2\pi)^3 \abs{\Tr\Bk{\lg\na\rg^{-\si/r'}f_0(x) f_1(x) \lg\na\rg^{-\si/r'} \lg\na\rg^{\si/r'} A \lg\na\rg^{\si/r'} }}\\
		&\qquad \qquad \le (2\pi)^3\No{\lg\na\rg^{-\si/r'} f_0(x)}_{\FS^{2r'}} \No{f_1(x)\lg\na\rg^{-\si/r'} }_{\FS^{2r'}}
		\No{\lg\na\rg^{\si/r'} A \lg\na\rg^{\si/r'}}_{\FS^r}.
	\end{align}
	By the Kato--Seiler--Simon inequality \eqref{eq:Kato Seiler Simon}, we have
	\begin{align}
		&\No{\lg\na\rg^{-\si/r'} f_0(x)}_{\FS^{2r'}}\ls \No{\lg\xi\rg^{-\si/r'}}_{L^{2r'}_\xi} \|f_0\|_{L^{2r'}_x}\ls \|f_0\|_{L^{2r'}_x}, \\
		&\No{f_1(x)\lg\na\rg^{-\si/r'}}_{\FS^{2r'}}\ls \No{\lg\xi\rg^{-\si/r'}}_{L^{2r'}_\xi} \|f_1\|_{L^{2r'}_x}\ls\|f_1\|_{L^{2r'}_x}.
	\end{align}
	Therefore, we obtain
	\begin{align}
	\abs{\int_{\R^3} f(x) \rh^1(A)(x)dx}
	&\ls \|f\|_{L^{r'}_x}(2\pi)^3\No{\lg\na\rg^{\si/r'} A \lg\na\rg^{\si/r'}}_{\FS^r}.
	\end{align}
	By the duality argument, we obtain \eqref{eq:Sobolev} with $\hb=1$.
\end{proof}

\section{Dispersive estimate for the modified propagators}\label{sec:L1LI}
In this section, we aim to prove an $L^1$--$L^\I$ dispersive estimate for the modified propagator $\U(t)e^{-i\Ps(t,-i\hb\na_x)}$ (Proposition \ref{prop:perturbed dispersive}). As a corollary, we obtain the uniform-in-$\hbar$ dispersive estimates for the density function (Corollary \ref{cor:free}).

\subsection{Free case}
First, we review the uniform estimate for the free propagator:
\begin{lemma}[Proposition 3.1 in \cite{Hadama Hong 2025}]\label{lem:free}
	Let $d=3$ and $\si>3/2$. Then, we have
	\begin{align}
		&\No{\rhh(\U(t)\gah_0\U(t)^*)}_{L^r_x}
		  \ls \No{\sd^{\si/r'} \gah_0 \sd^{\si/r'}}_{\FS^r_\hb}, \label{eq:free bounded}\\
		&\No{\rhh(\U(t)\gah_0\U(t)^*)}_{L^r_x}
		  \ls \frac{1}{|t|^{3/r'}} \No{\bx^{\si/r'} \gah_0 \bx^{\si/r'}}_{\FS^r_\hb} \label{eq:free decay}
	\end{align}
	for any $r\in[1,\I]$, where the implicit constants are independent of $\hb$.
\end{lemma}

\begin{remark}
	The estimate \eqref{eq:free bounded} is equivalent to Lemma \ref{lem:Sobolev}.
	Hence, \eqref{eq:free bounded} is equivalent to the case when $\hb=1$ (see Remark \ref{rmk:sobolev scaling}).
	More direct scaling implies that \eqref{eq:free decay} is equivalent to the case when $\hb=1$.
\end{remark}

Since Lemma \ref{lem:free} is not enough for our analysis, we will give a generalized version of Lemma \ref{lem:free} as a corollary of Proposition \ref{prop:perturbed dispersive}.
However, before discussing the generalization of Lemma \ref{lem:free}, we give estimates that can bound the $\CF L^1$--norm, which is stronger than the $L^\I_x$--norm.
To the best of the author's knowledge, this is a new observation even if its proof is simple. 
\begin{lemma}\label{lem:free Fourier}
	Let $d=3$ and $\si \in (3/2,2)$. Then, we have 
	\begin{align}
		&\No{\rhh(\U(t)\gah_0\U(t)^*)}_{\CF L^1}  \ls \frac{1}{\hb^{3/2}}\No{\sd^\si \gah_0\sd^\si}_{\FS^2_\hb}, \label{eq:FL1 bounded} \\
		&\No{\rhh(\U(t)\gah_0\U(t)^*)}_{\CF L^1} \ls \frac{1}{\hb^{3/2}|t|^3} \No{\bx^\si \gah_0 \bx^\si}_{\FS^2_\hb}, \label{eq:FL1 decay}
	\end{align}
	where the implicit constants are independent of $\hb$.
\end{lemma}

\begin{remark}\label{rmk:scaling FL1}
	By the same argument as Remark \ref{rmk:sobolev scaling}, we find that \eqref{eq:FL1 bounded} is reduced to the case $\hb=1$.
	More direct scaling implies that \eqref{eq:FL1 decay} is also equivalent to the case when $\hb=1$.
\end{remark}

\begin{proof}
By Remark \ref{rmk:scaling FL1}, we can assume that $\hb=1$.
	
\smallskip \noindent  \underline{\textbf{Step 1: Decay estimate.}}
	By a simple scaling, \eqref{eq:FL1 decay} is equivalent to the case $\hb=1$. 
	Let $f \in C_c^\I(\R^3)$. Then, we have by Lemma \ref{lem:duality}, cyclicity of the trace, \eqref{eq:triangle inequality}, and \eqref{eq:Holder} that 
	\begin{align}
		\abs{\int_{\R^3} f(\xi) \CF \rh^1(\CU^1(t)\ga_0 \CU^1(t)^*) d\xi}
		&= (2\pi)^3 \abs{\Tr\Dk{\bx^{-\si} \CU^1(t)^* \wh{f}(x)\CU^1(t) \bx^{-\si} \bx^\si \ga_0 \bx^\si}} \\
		&\le (2\pi)^3 \No{\bx^{-\si} \CU^1(t)^* \wh{f}(x) \CU^1(t) \bx^{-\si}}_{\FS^2} \|\bx^\si\ga_0 \bx^\si\|_{\FS^2}.
	\end{align}
	By the MDFM decomposition \eqref{eq:MDFM}, we have
	\begin{align}
		\CU^1(t)^* \wh{f}(x) \CU^1(t) = \CM^1(-t)\wh{f}(-it\na_x)\CM^1(t).
	\end{align}
	Therefore, we have by \eqref{eq:Hilbert Schmidt}
	\begin{align}
		&(2\pi)^3\No{\bx^{-\si} \CU^1(t)^* \wh{f}(x) \CU^1(t) \bx^{-\si}}_{\FS^2}
		= (2\pi)^3\No{\bx^{-\si} \wh{f}(-it\na_x) \bx^{-\si}}_{\FS^2} \\
		&\quad =\frac{(2\pi)^{3/2}}{|t|^3} \No{\bx^{-\si}\lg x'\rg^{-\si} f\K{\frac{x-x'}{t}}}_{L^2_{x,x'}}
		\le \frac{(2\pi)^{3/2}}{|t|^3} \No{\bx^{-\si}}_{L^2_x}^2 \|f\|_{L^\I_x}.
	\end{align}
	From the above, by the duality argument, we obtain
	\begin{align}
		\No{\rh^1(\CU^1(t)\ga_0\CU^1(t)^*)}_{\CF L^1}
		\ls \frac{1}{|t|^3} (2\pi)^{3/2}\No{\bx^\si \ga_0 \bx^\si}_{\FS^2}.
	\end{align}
	
	\smallskip \noindent  \underline{\textbf{Step 2: Boundedness.}}
	Similarly, by Lemma \ref{lem:duality}, cyclicity of the trace, and \eqref{eq:Holder}, we obtain
	\begin{align}
		\abs{\int_{\R^3} f(\xi) \CF \rh^1(\CU^1(t)\ga_0 \CU^1(t)^*) d\xi} 
		&= (2\pi)^3 \abs{\Tr\Dk{\lg\na\rg^{-\si} \CU^1(t)^* \wh{f}(x)\CU^1(t) \lg\na\rg^{-\si} \lg\na\rg^\si \ga_0 \lg\na\rg^\si}} \\
		&\le (2\pi)^{3/2}\No{\lg\na\rg^{-\si} \wh{f}(x) \lg\na\rg^{-\si}}_{\FS^2} (2\pi)^{3/2}\No{\lg\na\rg^\si\ga_0 \lg\na\rg^\si}_{\FS^2}.
	\end{align}
	Note that
	\begin{align}
		&(2\pi)^{3/2}\No{\lg\na\rg^{-\si} \wh{f}(x) \lg\na\rg^{-\si}}_{\FS^2}
		= (2\pi)^{3/2} \No{\lg\xi\rg^{-\si} \CF \wh{f}(x)\CF^{-1} \lg\xi\rg^{-\si}}_{\FS^2} \\
		&\quad = \No{\lg\xi\rg^{-\si}\lg\xi'\rg^{-\si} f(\xi-\xi') }_{L^2_{\xi,\xi'}}
		\ls \No{\lg\xi\rg^{-\si}}_{L^2_\xi}^2 \|f\|_{L^\I_\xi} \ls \|f\|_{L^\I_\xi}.
	\end{align}
	By the duality argument, we obtain
	\begin{align}
		\No{\rh^1(\CU^1(t)\ga_0\CU^1(t)^*)}_{\CF L^1}
		\ls (2\pi)^{3/2}\No{\lg\na\rg^\si \ga_0\lg\na\rg^\si}_{\FS^2}.
	\end{align}
\end{proof}

\subsection{\texorpdfstring{$L^1$--$L^\I$}{L1-L∞} estimate}
The most fundamental property of the free propagator $\U(t)$ is
\begin{align}\label{eq:free L1 Linfty}
	\No{\U(t)}_{L^1_x \to L^\I_x} \ls \frac{1}{|t\hb|^{3/2}}.
\end{align}
However, \eqref{eq:free L1 Linfty} is not enough for our analysis. 
Hence, we prove the following generalization.
\begin{proposition}\label{prop:perturbed dispersive}
	Let $d=3$. Assume that $\Ps(t,\xi)\in C((0,\I)\times \R^3)$ is real-valued, $\Ps(t,\xi) \in C^2(\R^3)$ for all $t\in (0,\I)$, and satisfies
	\begin{equation}\label{eq:assumption for a}
		\|\na_\xi^{2}\Ps(t,\xi)\|_{L^\I_\xi} \le \tw\abs{\frac{t}{\hb}},
		\quad \|\na_\xi^{3}\Ps(t,\xi)\|_{L^3_\xi} \ls \abs{\frac{t}{\hb}}
	\end{equation}
	for all $t \in (0,\I)$ and $\hb \in (0,1]$.
	Then, we have
	\begin{equation}\label{eq:L1Linfty decay}
	\No{\U(t)e^{- i\Ps(t,-i\hb\nabla)}}_{L^1_x \to L^\I_x}
	\ls \frac{1}{|t\hb|^{3/2}}
	\end{equation}
	for any $t\in(0,\I)$ and $\hb \in (0,1]$, where the implicit constant is independent of $\hb$.
\end{proposition}

\begin{remark}
	The author suspects that similar results to Proposition \ref{prop:perturbed dispersive} were already known; however, the author was not able to find such a reference. Thus, we give a self-contained proof here.
\end{remark}

Before proving Proposition \ref{prop:perturbed dispersive}, we give a lemma.
Lemma \ref{lem:bijective} is a corollary of the global inverse function theorem by Hadamard, but we give its proof here for readers' convenience.
\begin{lemma}\label{lem:bijective}
	Let $\ps:\R^3 \to \R$ be a $C^2$--function. If $\na^2\ps \ge 1/2$, then $\na \ps:\R^3 \to \R^3$ is bijective.
\end{lemma}
\begin{proof}
By the assumption, we have
    \begin{equation}\label{eq:lower bounds for ps}
	\begin{aligned}
		&\ps(x) \ge \ps(y) + (\na \ps)(y)\cdot (x-y) + \frac{1}{4} |x-y|^2, \\
		&\ps(y) \ge \ps(x) + (\na \ps)(x)\cdot (y-x) + \frac{1}{4} |x-y|^2.
	\end{aligned}
\end{equation}
Hence, adding two inequalities, we obtain
\begin{equation}
	(\na \ps (x) - \na \ps(y)) \cdot (x-y) \ge \tw |x-y|^2.
\end{equation}
Therefore, if $\na \ps (x) = \na \ps(y) $, then we have 
\begin{equation}
	0\ge \tw |x-y|^2,
\end{equation}
which implies $x=y$. Therefore, $\na \ps$ is injective.

Next, we prove the surjectivity. Fix an arbitrary $y\in\R^3$. Let $F(x):= \ps(x)-x\cdot y$.
By \eqref{eq:lower bounds for ps}, we have
\begin{equation}
	F(x) \ge \ps(0) + \Ck{(\na \ps)(0) - y}\cdot x + \frac{1}{4} |x|^2 \ge \frac15 |x|^2 - C
\end{equation}
for a sufficiently large constant $C=C(\ps,y)$.
Hence, there exists $x_* \in \R^3$ such that $F(x_*) = \min_{x\in\R^3} F(x)$.
Since $F \in C^2(\R^3;\R)$, we have $\na \ps(x_*) - y =\na F(x_*) = 0$, which implies $\na \ps(x_*)=y$.
Therefore, $\na \ps $ is surjective.
\end{proof}

\begin{proof}[Proof of Proposition \ref{prop:perturbed dispersive}]
\smallskip \noindent  \underline{\textbf{Step 0: Elimination of $t\hb$ by scaling.}}
	First, we remove $t\hb$ from the estimate \eqref{eq:L1Linfty decay}.
	Since
	\begin{equation}
		\No{\CF^{-1} e^{i\ph(\lm\xi)}\CF }_{L^1_x \to L^\I_x} = \frac{1}{\lm^{3}} \No{\CF^{-1} e^{i\ph(\xi)}\CF }_{L^1_x \to L^\I_x}
	\end{equation}
	holds for any phase function $\ph$ and $\lm>0$, we have 
	\begin{align}
		\No{\U(t)e^{- i\Ps(t,-i\hb\nabla)}}_{L^1_x \to L^\I_x}
		     &= \No{\CF^{-1}e^{-it\hb|\xi|^2/2 - i\Ps(t,\hb\xi)} \CF}_{L^1_x \to L^\I_x} \\
		     &=  \frac{1}{|t\hb|^{3/2}}\No{\CF^{-1}e^{-i|\xi|^2/2 - i\Ps\K{t,\sqrt{\hb/t}\xi}} \CF}_{L^1_x \to L^\I_x}.
	\end{align}
	For the phase function, we have
	\begin{align}
		&\|\na^2\Ps\|_{L^\I_\xi} \le \tw \abs{\frac{t}{\hb}} \Longleftrightarrow \No{\na^2\Dk{\Ps\K{t,\sqrt{\hb/t}\xi}}}_{L^\I_\xi} \le \frac{1}{2}, \\
		&\|\na^3 \Ps\|_{L^{3}_\xi} \ls \frac{t}{\hb} \Longleftrightarrow \No{\na^3 \Dk{\Ps(t,\sqrt{\hb/t}\xi)}}_{L^{3}_\xi} \ls 1. 
	\end{align}
	Therefore, it suffices to show that
	\begin{equation}
		\|\na^2 \Ps\|_{L^\I_\xi} \le \tw \text{ and } \|\na^3 \Ps\|_{L^{3}_\xi} \ls 1
		\implies \No{\CF^{-1}e^{-i|\xi|^2/2 - i\Ps(\xi)} \CF}_{L^1_x \to L^\I_x} \ls 1.
	\end{equation}

\smallskip \noindent  \underline{\textbf{Step 1: Reduction of the proof to the estimate of $B_R$.}}

\smallskip \noindent  \textbf{$\blacklozenge$ Setup of the proof.}
First, note that
	\begin{align*}
		e^{i\De/2 - i\Ps(-i\na_x)}u_0(x)
		&= \frac{1}{(2\pi)^{3}} \int_{\R^3} \K{\lim_{R \to \I}\int_{\R^3} \ph_R(\xi) e^{i(x-x')\xi-i|\xi|^2/2-i\Ps(\xi)} d\xi} u_0(x') dx' \\
		&=:\frac{1}{(2\pi)^{3}}\K{\lim_{R \to \I} S_R }\ast u_0(x),
	\end{align*}
	where $\ph_R(\xi)=\ph(\xi/R)$ and $\ph \in C_c^\I(\R^3)$ satisfies $0\le \ph\le 1$ and $\ph \equiv 1$ on $\{|\xi|\le 1\}$.
	The desired bound is
	\begin{equation*}
		\No{\lim_{R\to \I} S_R}_{L^\I_x} \ls 1.
	\end{equation*}

\smallskip \noindent  \textbf{$\blacklozenge$ Unique critical point.}	
	Define the phase function by
	$$\ps (\xi) := x\xi-\frac{|\xi|^2}{2}-\Ps(\xi).$$
	Then, we have
	\begin{equation*}
		\begin{aligned}
			&\na \ps(\xi) = x -\xi - \na \Ps(\xi), \\
			&\na^2 \ps(\xi) = -\Id_{\R^3} -  \na^2 \Ps(\xi).
		\end{aligned}
	\end{equation*}
	By the assumptions for $\Ps$, we have
	\begin{equation}
		-\xi^T (\na^2 \ps)(\xi) \xi
		\ge|\xi|^2 - \|\na^2 \Ps\|_{L^\I_\xi} |\xi|^2
		\ge \frac{|\xi|^2}{2}.
	\end{equation}
	Hence, Lemma \ref{lem:bijective} implies that $-\na\ps:\R^3_\xi \to \R^3_\xi$ is bijective for all $x \in \R^3$.
	Therefore, there exists a unique critical point $\xi_c = \xi_c(x)$ such that
	$\na \ps(\xi_c)=0$ for all $x \in \R^3$.

\smallskip \noindent  \textbf{$\blacklozenge$ Decomposition of $S_R$.}
Let $\ch\in C_c^\I(\R^3)$ satisfy $\ch\equiv1$ on $\{|\xi|\le 1\}$	and $0\le \ch \le 1$.
	We decompose $S_R = A_R + B_R$, where
	\begin{align}
	&A_R:=\int_{\R^3} e^{i\ps(\xi)}\ph_R(\xi)\chi(\xi-\xi_c)d\xi, \\
	&B_R:=\int_{\R^3} e^{i\ps(\xi)}\ph_R(\xi)\Ck{1-\chi(\xi-\xi_c)}d\xi.
	\end{align}
	Therefore, we have
	\begin{align}
	\No{\lim_{R\to \I} S_R}_{L^\I_x} \le \No{\lim_{R\to\I}A_R}_{L^\I_x} + \No{\lim_{R \to \I} B_R}_{L^\I_x}. 
	\end{align}

\smallskip \noindent  \textbf{$\blacklozenge$ Estimate of $A_R$.}	
	We can estimate $A_R$ easily. Indeed, we have
	\begin{equation}
		|A_R| \le \int_{\R^3} \abs{\ch(\xi-\xi_c)} d\xi \ls 1.
	\end{equation}
	Hence we have
	\begin{align}
		 \No{\lim_{R\to\I}A_R}_{L^\I_x}\le \limsup_{R\to\I}\|A_R\|_{L^\I_x} \ls 1.
	\end{align}
	In the rest of this proof, we estimate $B_R$.
	
\smallskip \noindent  \underline{\textbf{Step 2: Estimate of $B_R$.}}
	The integration by parts implies
	\begin{align*}
		B_R &= -i\int_{\R^3}  \na \Dk{e^{i\ps(\xi)}} \cdot \frac{ \na \ps(\xi)}{|\na \ps(\xi)|^2} \ph_R(\xi) \Ck{1 - \chi(\xi-\xi_c)} d\xi \\
		&= -i\sum_{j=1}^3 \int_{\R^3} \pl_{\xi_j} \Dk{e^{i\ps(\xi)}} \frac{\pl_{\xi_j} \ps(\xi)}{|\na \ps(\xi)|^2}  \ph_R(\xi) \Ck{1 - \chi(\xi-\xi_c)} d\xi \\
		&= i \sum_{j=1}^3 \int_{\R^3}  e^{i\ps(\xi)} \pl_{\xi_j} \Dk{\frac{\pl_{\xi_j}\ps(\xi)}{|\na \ps(\xi)|^2}  }
		\ph_R(\xi)\Ck{1 - \chi(\xi-\xi_c) } d\xi \\
		&\quad + i \sum_{j=1}^3 \int_{\R^3}  e^{i\ps(\xi)} \frac{\pl_{\xi_j}\ps(\xi)}{|\na \ps(\xi)|^2}
		(\pl_{\xi_j}	\ph_R)(\xi) \Ck{1 - \chi(\xi-\xi_c)} d\xi\\
		&\quad + i \sum_{j=1}^3 \int_{\R^3}  e^{i\ps(\xi)} \frac{\pl_{\xi_j}\ps(\xi)}{|\na \ps(\xi)|^2}
		\ph_R(\xi)\pl_{\xi_j}
		\Dk{1 - \chi(\xi-\xi_c)} d\xi =: C_R+D_R+E_R.
	\end{align*}

\smallskip \noindent  \underline{\textbf{Step 2.1: Estimate of $E_R$.}}
First, we show that $|\na \ps(\xi)| \ge 1/2 $ if $1 \le  |\xi-\xi_c|$.
Indeed, since $\na \ps(\xi_c)=0$, we have
\begin{equation}\label{eq:lower bound}
	\begin{aligned}
		|\na \ps(\xi)|
		&= |x-\xi-\na \Ps(\xi)| \\
		&= |-(\xi-\xi_c) + \K{x -\xi_c - \na \Ps(\xi_c) }  + \na \Ps(\xi_c) - \na \Ps(\xi)| \\
		&\ge |\xi-\xi_c| -  \|\na^2 \Ps\|_{L^\I_\xi}|\xi_c- \xi| 
		\ge \tw |\xi_c-\xi| \ge \tw .
	\end{aligned}
\end{equation}
Hence, we have
\begin{align}
	|E_R| &\ls \sum_{j=1}^3 \int_{\R^3} \frac{1}{|\na \ps(\xi)|}
	\abs{\pl_{\xi_j} \chi(\xi-\xi_c)}d\xi \ls \No{\na \ch(\xi)}_{L^1_\xi} \ls 1.
\end{align}

\smallskip \noindent  \underline{\textbf{Step 2.2: Estimate of $C_R$.}}
Note that
\begin{equation}
	\begin{aligned}
		\sum_{j=1}^3 \pl_{\xi_j} \Dk{ \frac{\pl_{\xi_j} \ps}{|\na \ps|^2}}
		&= \frac{\De_\xi \ps}{|\na \ps|^2} - \frac{2 (\na \ps)^T (\na^2 \ps) \na \ps}{|\na \ps|^4} \\
		&= \frac{-1}{|\na \ps|^2}
		- \frac{\De \Ps(\xi)}{|\na \ps|^2}
		+ \frac{2 (\na \ps)^T (\na^2 \Ps)(\xi) \na \ps}{|\na \ps|^4}
		=:\Ph(\xi).
	\end{aligned}
\end{equation}
Therefore, we have
\begin{align}
	C_R &= i \int_{\R^3}  e^{i\ps(\xi)}\Ph(\xi) \ph_R(\xi)\Ck{1 - \chi(\xi-\xi_c) } d\xi \\
	&= - \sum_{j=1}^3 \int_{\R^3}\pl_{\xi_j} \Dk{e^{i\ps(\xi)}}  \frac{\pl_{\xi_j} \ps(\xi)}{|\na \ps(\xi)|^2}  
	\frac{\ph_R(\xi)}{|\na \ps(\xi)|^2} \Ck{1 - \chi(\xi-\xi_c) } d\xi \\
	&\quad + i \int_{\R^3}  e^{i\ps(\xi)}\ph_R(\xi) \K{\frac{-\De \Ps(\xi)}{|\na \ps|^2}
		+  \frac{2 (\na \ps)^T (\na^2 \Ps)(\xi) \na \ps}{|\na \ps|^4}}\Ck{1 - \chi(\xi-\xi_c) } d\xi\\
	&=: C_{R,1}+ iC_{R,2}.
\end{align}

\smallskip \noindent  \textbf{$\blacklozenge$ Estimate of $C_{R,1}$.}
For $C_{R,1}$, we have
\begin{align}
	C_{R,1} &= - \sum_{j=1}^3 \int_{\R^3} \pl_{\xi_j}	\Dk{e^{i\ps(\xi)}} \frac{\pl_{\xi_j} \ps(\xi)}{|\na \ps(\xi)|^4} 
	\ph_R(\xi)\Ck{1 - \chi(\xi-\xi_c) } d\xi \\
	&=  \sum_{j=1}^3 \int_{\R^3} e^{i\ps(\xi)} \pl_{\xi_j}  \Dk{\frac{\pl_{\xi_j} \ps(\xi)}{|\na \ps(\xi)|^4}} \ph_R(\xi)
	\Ck{1 - \chi(\xi-\xi_c)} d\xi \\
	&\quad +\sum_{j=1}^3 \int_{\R^3} e^{i\ps(\xi)} \frac{\pl_{\xi_j} \ps(\xi)}{|\na \ps(\xi)|^4}
	(\pl_{\xi_j}\ph_R)(\xi)  \Ck{1 - \chi(\xi-\xi_c) } d\xi\\
	&\quad + \sum_{j=1}^3 \int_{\R^3} e^{i\ps(\xi)} \frac{\pl_{\xi_j} \ps(\xi)}{|\na \ps(\xi)|^4}
	\ph_R(\xi)\pl_{\xi_j}  \Dk{1 - \chi(\xi-\xi_c) } d\xi=:C_{R,1,1} + C_{R,1,2} + C_{R,1,3}.
\end{align}
To estimate $C_{R,1,1}$, note the following: When $|\xi-\xi_c| \ge 1$, by \eqref{eq:lower bound} and $|\na^2 \ps| \ls 1$, we have
\begin{equation}\label{eq:divergence estimate}
	\begin{aligned}
		&\abs{\sum_{j=1}^3 \pl_{\xi_j}  \Dk{\frac{\pl_{\xi_j} \ps}{|\na \ps|^4}}}
		=\abs{\frac{\De \ps}{|\na \ps|^4} - \frac{4(\na \ps)^T (\na^2 \ps) (\na \ps)}{|\na \ps|^6}} \\
		&\quad \ls \frac{|\na^2 \ps|}{|\na \ps|^4} \ls \frac{1}{ |\xi-\xi_c|^4}
		\ls \frac{1}{1+|\xi-\xi_c|^4}.
	\end{aligned}
\end{equation}
Therefore, we have
\begin{equation}
	\abs{C_{R,1,1}}\ls \int_{\R^3} \frac{1}{1+|\xi-\xi_c|^4} d\xi \ls 1.
\end{equation}
For $C_{R,1,2}$, we have
\begin{align}
	\abs{C_{R,1,2}} \ls \frac{1}{R} \int_{\R^3} \frac{|\na \ph(\xi/R)|}{1+|\xi-\xi_c|^3}  d\xi 
	\le \|\na \ph\|_{L^3_x} \No{\frac{1}{1+|\xi|^3}}_{L^{3/2}_\xi} \ls 1.
\end{align}
For $C_{R,1,3}$, since $|\na \ps| \ge 1/2$, we have
\begin{equation}
	\abs{C_{R,1,3}} \ls \int_{\R^3} \frac{1}{|\na \ps(\xi)|^3}
	\abs{\na \ch (\xi-\xi_c) } d\xi \ls 1.
\end{equation}

\smallskip \noindent  \textbf{$\blacklozenge$ Estimate of $C_{R,2}$.}
\begin{align}
	C_{R,2} &= \int_{\R^3}  e^{i\ps(\xi)}\ph_R(\xi) \frac{-\De \Ps(\xi)}{|\na \ps|^2} \Ck{1 - \chi(\xi-\xi_c) } d\xi \\
		&\quad + \int_{\R^3}  e^{i\ps(\xi)}\ph_R(\xi) \frac{2 (\na \ps)^T (\na^2 \Ps)(\xi) \na \ps}{|\na \ps|^4}\Ck{1 - \chi(\xi-\xi_c) } d\xi
		=:  C_{R,2,1} +  C_{R,2,2}.
\end{align}
We only estimate $C_{R,2,1}$ because the same argument works for $C_{R,2,2}$.
We have
\begin{align}
	\abs{C_{R,2,1}}
	    &\sim \abs{\int_{\R^3}  e^{i\ps(\xi)}\ph_R(\xi) \frac{\De \Ps(\xi)}{|\na \ps|^2} \Ck{1 - \chi(\xi-\xi_c) } d\xi} \\
	    &= \abs{\sum_{j=1}^{3} \int_{\R^3} \pl_{\xi_j}[e^{i\ps(\xi)}]\frac{\pl_{\xi_j} \ps(\xi)}{|\na \ps|^4} \ph_R(\xi) \De \Ps(\xi)  \Ck{1 - \chi(\xi-\xi_c) } d\xi}\\
	    &\le \abs{\int_{\R^3} e^{i\ps(\xi)}\sum_{j=1}^{3}\pl_{\xi_j}\Dk{\frac{\pl_{\xi_j} \ps(\xi)}{|\na \ps|^4}} \ph_R(\xi) \De \Ps(\xi)  \Ck{1 - \chi(\xi-\xi_c) } d\xi}\\
	    &\quad +\abs{\sum_{j=1}^{3} \int_{\R^3} e^{i\ps(\xi)}\frac{\pl_{\xi_j} \ps(\xi)}{|\na \ps|^4} \pl_{\xi_j}[\ph_R(\xi)] \De \Ps(\xi)  \Ck{1 - \chi(\xi-\xi_c) } d\xi} \\
	    &\quad + \abs{\sum_{j=1}^{3}\int_{\R^3} e^{i\ps(\xi)}\frac{\pl_{\xi_j} \ps(\xi)}{|\na \ps|^4} \ph_R(\xi) \pl_{\xi_j}\De \Ps(\xi) \Ck{1 - \chi(\xi-\xi_c) } d\xi} \\
	    &\quad + \abs{\sum_{j=1}^{3}\int_{\R^3} e^{i\ps(\xi)}\frac{\pl_{\xi_j} \ps(\xi)}{|\na \ps|^4} \ph_R(\xi) \De \Ps(\xi) (\pl_{\xi_j}\chi)(\xi-\xi_c)d\xi}
	    =:\sum_{k=1}^4 C_{R,2,1}^{(k)}.
\end{align}
For $C_{R,2,1}^{(1)}$, by \eqref{eq:divergence estimate}, we have
\begin{align}
	C_{R,2,1}^{(1)}
	&\ls \int_{\R^3} \frac{1}{1+|\xi-\xi_c|^4} d\xi \ls 1.
\end{align}
For $C_{R,2,1}^{(2)}$, by \eqref{eq:lower bound}, we have
\begin{align}
	C_{R,2,1}^{(2)} &\ls \int_{\R^3} \frac{|1-\ch(\xi-\xi_c)|}{|\na \ps|^3} \frac{1}{R}|\na\ph(\xi/R)|d\xi \ls \int_{\R^3} \frac{1}{1+|\xi-\xi_c|^3} \frac{1}{R}|\na\ph(\xi/R)| d\xi \ls 1.
\end{align}
For $C_{R,2,1}^{(3)}$, by \eqref{eq:lower bound}, we have
\begin{align}
	C_{R,2,1}^{(3)} &\ls \int_{\R^3} \frac{|1 - \chi(\xi-\xi_c)|}{|\na \ps|^3} |\na \De \Ps(\xi)|  d\xi\ls \int_{\R^3} \frac{1}{1+|\xi - \xi_c|^3} |\na^3 \Ps(\xi)| d\xi \ls 1.
\end{align}
For $C_{R,2,1}^{(4)}$, by \eqref{eq:lower bound}, we have
\begin{align}
	C_{R,2,1}^{(4)} &\ls \int_{\R^3}|\na \ch(\xi-\xi_c)| d\xi \ls 1.
\end{align}

\smallskip \noindent  \underline{\textbf{Step 2.3: Estimate of $D_R$.}}
By the integration by parts, we have
\begin{align}
	D_R
	&= \sum_{j,k=1}^3 \int_{\R^3} \pl_{\xi_k} \Dk{e^{i\ps(\xi)}} \frac{\pl_{\xi_k} \ps(\xi)}{|\na \ps(\xi)|^2}   \frac{\pl_{\xi_j}\ps(\xi)}{|\na \ps(\xi)|^2}
	(\pl_{\xi_j}\ph_R)(\xi) \Ck{1 - \chi(\xi-\xi_c)} d\xi \\
	&= -\sum_{j,k=1}^3 \int_{\R^3} e^{i\ps(\xi)} \pl_{\xi_k} \Dk{\frac{\pl_{\xi_j}\ps(\xi) \pl_{\xi_k} \ps(\xi)}{|\na \ps(\xi)|^4}}  
	(\pl_{\xi_j}\ph_R)(\xi) \Ck{1 - \chi(\xi-\xi_c)} d\xi \\
	&\quad - \sum_{j,k=1}^3 \int_{\R^3} e^{i\ps(\xi)} \frac{\pl_{\xi_j} \ps(\xi)\pl_{\xi_k} \ps(\xi)}{|\na \ps(\xi)|^4}  
	(\pl_{\xi_j}\pl_{\xi_k}\ph_R)(\xi) \Ck{1 - \chi(\xi-\xi_c)} d\xi \\
	&\quad + \sum_{j,k=1}^3 \int_{\R^3} e^{i\ps(\xi)} \frac{\pl_{\xi_j} \ps(\xi)\pl_{\xi_k} \ps(\xi)}{|\na \ps(\xi)|^4}  
	(\pl_{\xi_j}\ph_R)(\xi) (\pl_{\xi_k}\chi)(\xi-\xi_c) d\xi 
	=: D_{R,1} + D_{R,2} + D_{R,3}.
\end{align}

\smallskip \noindent  \textbf{$\blacklozenge$ Estimate of $D_{R,1}$.}
For $D_{R,1}$, we have
\begin{align}
	\abs{D_{R,1}} &\ls \frac{1}{R}\int_{\R^3} \frac{|\na^2 \ps(\xi)|}{|\na \ps(\xi)|^3}  
	\abs{\na \ph(\xi/R)} \abs{1 - \chi(\xi-\xi_c)} d\xi \\
	&\ls \frac{1}{R}\int_{\R^3} \frac{1}{1+|\xi-\xi_c|^3}  
	\abs{\na \ph(\xi/R)} d\xi \ls  \No{\na \ph}_{L^3_\xi}.
\end{align}

\smallskip \noindent  \textbf{$\blacklozenge$ Estimate of $D_{R,2}$.}
For $D_{R,2}$, we have
\begin{align}
	\abs{D_{R,2}} &\ls \frac{1}{R^2} \int_{\R^3} \frac{1}{|\na \ps(\xi)|^2}  
	\abs{\na^2 \ph(\xi/R)} \abs{1 - \chi(\xi-\xi_c)}d\xi \\
	&\ls  \frac{1}{R^2} \int_{\R^3} \frac{1}{1 +  |\xi-\xi_c|^2}  
	\abs{\na^2 \ph(\xi/R)} d\xi \ls \|\na^2 \ph\|_{L^{3/2}_\xi}.
\end{align}

\smallskip \noindent  \textbf{$\blacklozenge$ Estimate of $D_{R,3}$.}
For $D_{R,3}$, we have
\begin{align}
	\abs{D_{R,3}} &\ls \frac{1}{R} \int_{\R^3} \frac{1}{|\na \ps(\xi)|^2} 
	 \abs{\na\chi(\xi-\xi_c)} d\xi
	 \ls \frac{1}{R} \int_{\R^3} \abs{\na \chi(\xi-\xi_c)} d\xi
	 \ls \frac{1}{R} \le 1.
\end{align}

\end{proof}

\subsection{Uniform-in-\texorpdfstring{$\hb$}{h} estimates for modified propagators}
As a corollary of Proposition \ref{prop:perturbed dispersive}, we show a generalization of Lemma \ref{lem:free}.
\begin{corollary} \label{cor:free}
Suppose the same assumptions as in Proposition \ref{prop:perturbed dispersive}.
Moreover, let $\si > 3/2$ and $r\in [1,\I]$. Then, we have
	\begin{align}
		&\No{\rhh\K{\U(t)e^{-i\Ps(t,-i\hb\na_x)} \ga_0^\hb e^{i\Ps(t,-i\hb\na_x)} \U(t)^*}}_{L^r_x}
		\ls \No{\sd^{\si/r'}\ga_0^\hb \sd^{\si/r'}}_{\FS^r_\hb}, \label{eq:bounded} \\
		&\No{\rhh\K{\U(t)e^{-i\Ps(t,-i\hb\na_x)} \ga_0^\hb e^{i\Ps(t,-i\hb\na_x)} \U(t)^*}}_{L^r_x}
		\ls \frac{1}{|t|^{3/r'}}\No{\bx^{\si/r'} \ga_0^\hb\bx^{\si/r'}}_{\FS^r_\hb},\label{eq:decay}
	\end{align}
	where the implicit constants are independent of $\hb$.
\end{corollary}

\begin{remark}\label{rmk:scaling}
	The estimate \eqref{eq:bounded} is equivalent to \eqref{eq:Sobolev}.
	In particular, we can reduce \eqref{eq:Sobolev} to the case when $\hb=1$.
	By direct scaling, the proof of \eqref{eq:decay} can be reduced to the case when $\hb=1$.
\end{remark}

\begin{proof}
	By Remark \ref{rmk:scaling}, we can assume $\hb=1$.
	Since we know the trivial estimate
	\begin{equation}
		\|\rh^1(A)\|_{L^1_x} \le (2\pi)^3\|A\|_{\FS^1}, 
	\end{equation}
	we can assume $r=\I$.
	Indeed, if we get \eqref{eq:decay} for $r=1$ and $r=\I$, then the complex interpolation implies that \eqref{eq:decay} holds for general $r \in [1,\I]$.
	
	Let $f \in C_c^\I(\R^d)$ and decompose $f=f_0 f_1$ with $f_0:= |f|^{1/2}$.
	Then, we have
	\begin{align*}
		&\abs{\int_{\R^3} f(x) \rh^1\K{\CU^1(t)e^{-i\Ps(t,-i\na_x)} \ga_0 e^{i\Ps(t,-i\na_x)} \CU^1(t)^*} dx} \\
		&\quad = (2\pi)^3 \Tr\Bk{ \bx^{-\si} e^{i\Ps(t,-i\na_x)}\CU^1(t)^* f(x) \CU^1(t) e^{-i\Ps(t,-i\na_x)} \bx^{-\si} \bx^\si \ga_0 \bx^\si} \\
		&\quad \le (2\pi)^{3}\|\bx^\si \ga_0\bx^\si\|_{\CB}\prod_{n=0,1}\No{ \bx^{-\si} e^{i\Ps(t,-i\na_x)}\CU^1(t)^* f_n }_{\FS^2}.
	\end{align*}
	Note that the Dunford--Pettis theorem \cite[Theorem 2.2.5]{Dunford Pettis 1940} says that
	\begin{equation}
		\|A(x,x')\|_{L^\I_{x,x'}} = \|A\|_{L^1_x \to L^\I_x}.
	\end{equation}
	Therefore, Proposition \ref{prop:perturbed dispersive} and \eqref{eq:Hilbert Schmidt} imply
	\begin{align*}
		&\No{ \bx^{-\si} e^{i\Ps(t,-i\na_x)}\CU^1(t)^* f_n(x) }_{\FS^2} 
		\le \|\bx^{-\si}\|_{L^2_x} \|f_n\|_{L^2_x} \No{e^{i\Ps(t,-i\na_x)}\CU^1(t)^* (x,x')}_{L^\I_{x,x'}} \\
		&\quad = \|\bx^{-\si}\|_{L^2_x} \|f\|_{L^1_x}^{1/2} \No{e^{i\Ps(t,-i\na_x)}\CU^1(t)^*}_{L^1_x \to L^\I_x} 
		\ls  \frac{\|f\|_{L^1_x}^{1/2}}{|t|^{3/2}}
	\end{align*}
	for $n=0,1$. 
	Collecting the above argument with the duality argument, we obtain
	\begin{equation}
		\No{\rh^1\K{\CU^1(t)e^{-i\Ps(t,-i\na_x)} \ga_0 e^{i\Ps(t,-i\na_x)} \CU^1(t)^*}}_{L^\I_x}
		\ls \frac{1}{|t|^{3}}\No{\bx^{\si} \ga_0\bx^{\si}}_{\CB}.
	\end{equation}
\end{proof}

\section{Boundedness of wave operators with phase correction}\label{sec:wave}
This is the most important section of this paper.
The goal of this section is to prove the following proposition:
\begin{proposition}\label{prop:wave operator boundedness}
	Let $d=3$, $0\le s \le 2$. Then, there exists a small absolute number $\de>0$ such that the following holds:
	Assume that $V(t,x)\in L_t^\I([0,T];W^{4,\I}_x)$ is real-valued and satisfies
	\begin{align}
		\|\pl^\al V(t)\|_{L^\I_x}
		\ls \begin{dcases}
			& \bt^{-1} \quad \mbox{if } |\al|=0, \\
			& \bt^{-2} \quad \mbox{if } |\al|=1,\\
			& \bt^{-3+\de} \quad \mbox{if } |\al|=2,\\
			& \bt^{-4+\de} \quad \mbox{if } |\al|=3, \\
			& \bt^{-4+\de} \quad \mbox{if } |\al|=4
		\end{dcases}
	\end{align}
	for all $0\le t \le T$, and
	\begin{align}
		\|\pl^\al V(t)\|_{\CF L^1}\ls
			\begin{dcases}
				&\bt^{-2+\de}\quad \mbox{if } |\al|=1, \\
				&\bt^{-3+\de}\quad \mbox{if } |\al|=2,\\
				&\min\K{\bt^{-7/2+\de}, \hb^{-3/2}\bt^{-4+\de}} \quad \mbox{if } |\al|=3
			\end{dcases}
	\end{align}
	for all $0\le t \le T$.
	Define $\Ps(t,\xi)$ by
	\begin{align}
		\Ps(t,\xi):= \frac{1}{\hb}\int_0^t V(\ta,\ta\xi)d\ta.
	\end{align}
	Then, we have
	\begin{align}
		\sup_{\hb\in(0,1]}\sup_{0\le t \le T}\No{\bx^s e^{i\Ps(t,-i\hb\na_x)} \W_{V}(t) \bx^{-s}}_{\CB} \ls 1
	\end{align}
for all $0\le s\le 2$. Moreover, all implicit constants are independent of $T>0$.
\end{proposition}

\begin{remark}
	Later, we choose $V(t)=w\ast \rhh(\gah(t))$ and this $V(t)$ satisfies the assumptions in Proposition \ref{prop:wave operator boundedness}.
\end{remark}

Before proving Proposition \ref{prop:wave operator boundedness}, we give a much easier estimate.
\begin{lemma}\label{lem:easy}
Let $d=3$ and $\ep\in(0,1)$.
Assume that $V(t,x)\in L_t^\I([0,T];W^{2,\I}_x)$ is real-valued and satisfies
\begin{align}
	\|\pl^\al V(t)\|_{L^\I_x}
	\ls \begin{dcases}
		& \bt^{-1} \quad \mbox{if } |\al|=0, \\
		& \bt^{-2} \quad \mbox{if } |\al|=1,\\
		& \bt^{-3+\ep} \quad \mbox{if } |\al|=2
	\end{dcases}
\end{align}
for all $0\le t \le T$. Then, we have
	\begin{align}
		&\No{\sd^s \W_V(t) \sd^{-s}}_\CB + \No{\sd^s \W_V(t)^* \sd^{-s}}_\CB \ls 1\quad \mbox{for all } 0 \le s \le 2, \label{eq:regularity}\\
		&\No{\bx^s \W_V(t) \bx^{-s}}_\CB + \No{\bx^s \W_V(t)^* \bx^{-s}}_\CB \\
		&\qquad \qquad \qquad \qquad \qquad \qquad \ls \begin{dcases}\label{eq:weight}
			(\log(t+2))^{s}, \quad \mbox{if } 0 \le s \le 1, \\
			(\log(t+2))^{2s-1}\bt^{\ep(s-1)}\quad \mbox{if } 1\le s \le 2.
		\end{dcases}
	\end{align}
\end{lemma}
\begin{proof}[Proof of Lemma \ref{lem:easy}]
	We only prove \eqref{eq:weight} because the same proof works for \eqref{eq:regularity}.
	First, we have
	\begin{align}
		\No{\bx \W_V(t) \bx^{-1}}_\CB + \No{\bx^{-1} \W_V(t) \bx}_\CB
		\ls 1 + \No{\Bk{x, \W_V(t)}}_\CB.
	\end{align} 
	By \eqref{eq:commutator x W}, we obtain
	\begin{align}
		\No{\Bk{x, \W_V(t)}}_\CB = \No{\int_0^t \U_V(\ta)^*\ta\na V(\ta) \U_V(\ta) d\ta}_\CB 
		\ls \int_0^t \frac{\ta}{\bta^2} d\ta \le \log(t+2).
	\end{align}
	Hence, the complex interpolation implies \eqref{eq:weight} for all $0\le s\le 1$.
		
	Next, we consider the case $s=2$. Note that 
	\begin{align}
		\No{\bx^2 \W_V(t) \bx^{-2}}_\CB
		&\ls \No{\bx \W_V(t) \bx^{-1}}_\CB + \No{\bx \Bk{x,\W_V(t)} \bx^{-1}}_\CB \\
		&\ls \log(t+2) + \No{\bx \Bk{x,\W_V(t)} \bx^{-1}}_\CB.
	\end{align}
	By using \eqref{eq:commutator x W} and boundedness when $-1\le s \le 1$, we have
	\begin{align}
		&\No{\bx \Bk{x, \W_V(t)}\bx^{-1}}_\CB
		= \No{\bx\W_V(t)\int_0^t \U_V(\ta)^*\ta\na V(\ta) \U_V(\ta) d\ta \bx^{-1}}_\CB  \\
		&\qquad \ls \int_0^t \log(t+2)\No{\bx \W_V(\ta)^* \U(\ta)^* \ta\na V(\ta) \U(\ta) \W_V(\ta) \bx^{-1}}_\CB d\ta \\
		&\qquad \ls \log(t+2) \int_0^t (\log(\ta+2))^2 \No{ \lg\J(\ta)\rg \ta\na V(\ta) \lg \J(\ta)\rg^{-1} }_\CB d\ta.
	\end{align}
	By \eqref{eq:f(J)}, we have
	\begin{align}
		&\No{ \lg\J(\ta)\rg \ta\na V(\ta) \lg \J(\ta)\rg^{-1} }_\CB
		= \No{\lg\ta\hb\na\rg \ta\na V(\ta) \lg \ta\hb\na\rg^{-1} }_\CB \\
		&\quad \le \ta\No{\na V(\ta)}_{L^\I_x} + \hb\ta^2 \No{\na^2 V(\ta)}_{L^\I_x}
		\ls \frac{1}{\bta} + \frac{\hb}{\bta^{1-\ep}}.
	\end{align}
	Therefore, we obtain
	\begin{align}
		&\No{\bx \Bk{x, \W_V(t)}\bx^{-1}}_\CB \ls (\log(t+2))^3 \bt^\ep.
	\end{align}
	Collecting the above argument, we obtain
	\begin{align}
		\No{\bx^2 \W_V(t) \bx^{-2}}_\CB
		\ls (\log(t+2))^3 \bt^\ep.
	\end{align}
	Hence, again the complex interpolation implies \eqref{eq:weight} with $1\le s \le 2$.
	We can bound the adjoint term $\|\bx^s \W_V(t)^* \bx^{-s}\|_\CB$ in the same way.
\end{proof}

\begin{remark}
	As you can see, the boundedness in Lemma \ref{lem:easy} includes the growth with respect to $t\ge 0$; hence, it is not sufficient in our analysis. This is the motivation to prove Proposition \ref{prop:wave operator boundedness}. However, Lemma \ref{lem:easy} is still useful, and in fact, it is used in the proof of Proposition \ref{prop:wave operator boundedness}.
\end{remark}

In the sequel, we prove Proposition \ref{prop:wave operator boundedness}. 
In Section \ref{subsec:key cancellation lemma}, we prove the most important estimate in the proof of Proposition \ref{prop:wave operator boundedness}. In Section \ref{subsec:decomposition}, we give a useful decomposition of the operator to get cancellation. Finally, in Sections \ref{subsec:si=1} and \ref{subsec:si=2}, 
we provide a complete step-by-step proof.

\subsection{Key cancellation lemma and its corollary}\label{subsec:key cancellation lemma}
In the proof of Proposition \ref{prop:wave operator boundedness}, the quantity defined by
\begin{equation}\label{eq:Af}
	\CP[f](t):= \No{\Ck{f(t,-it\hb\na_x) - \M(-t) f(t,-it\hb\na_x) \M(t)}\bx^{-1}}_{\CB} \\
\end{equation}
plays an essential role.
The next is the fundamental lemma about $\CP[f](t)$.
\begin{lemma}\label{lem:key estimate}
	We have the following three bounds of $\CP[f](t)$.
	\begin{align}
		&\CP[f](t) \ls \lg t\hb\rg \No{\na f(t)}_{\CF L^1}, \label{eq:Af 1} \\
		&\CP[f](t) \ls \frac{1}{\lg t\hb\rg^{1/2}} \K{ \|f(t)\|_{L^\I_x} + |t\hb|\|\na f(t)\|_{L^\I_x} + |t\hb|^2\|\na^2 f(t)\|_{L^\I_x} }, \label{eq:Af 2} \\
		&\CP[f](t) \ls \frac{1}{\lg t\hb\rg^{1/2}} \K{\|f(t)\|_{L^\I_x} + |t\hb|\|\na f(t)\|_{\CF L^1_x}}. \label{eq:Af 3}
	\end{align}
\end{lemma}

\begin{proof}
\underline{\textbf{Step 1: Proof of \eqref{eq:Af 1}.}}
First, we have
\begin{align*}
	\CP[f](t) &\le \No{\Bk{\bx^{-1}, f(t,-it\hb\na_x)}}_\CB
	+ \No{\Bk{\M(t)\bx^{-1}, f(t,-it\hb\na_x) }}_{\CB} =: \SA + \SB.
\end{align*}
On the one hand, for $\SA$, we have by \eqref{eq:commutator basic}
\begin{align*}
	\SA&\le t\hb \|\na f(t)\|_{\CF L^1} \No{\Bk{x, \lg\na\rg^{-1}}}_\CB
	\ls \lg t\hb\rg \|\na f(t)\|_{\CF L^1}.
\end{align*}
On the other hand, for $\SB$, we have
\begin{align*}
	\SB 
	&\le t\hb\|\na f(t)\|_{\CF L^1} \No{\Bk{\na,\M(t)\bx^{-1} }}_\CB  \ls \lg t\hb\rg \|\na f(t)\|_{\CF L^1}.
\end{align*}
Collecting all the above, we obtain
\begin{equation}\label{eq:bound 1'}
	\CP[f](t)\ls  \lg t\hb\rg \|\na f(t)\|_{\CF L^1}.
\end{equation}

\smallskip \noindent  \underline{\textbf{Step 2: Proof of \eqref{eq:Af 2}.}}
	Next, we have
	\begin{align*}
		\CP[f](t)&\le \No{\K{1-\M(-t)}f(t,-it\hb\na_x)\bx^{-1}}_\CB \\
		&\qquad  + \No{\M(-t) f(t,-it\hb\na_x)\K{1-\M(t)}\bx^{-1}}_{\CB} \\
		&\le \No{\bx^{-1} \K{1-\M(-t)}f(t,-it\hb\na_x)}_\CB  +\No{f(t,-it\hb\na_x)\K{1-\M(t)}\bx^{-1}}_{\CB}\\
		&\qquad +\No{\bx^{-1}\K{1-\M(-t)} \Bk{\bx,f(t,-it\hb\na_x)\bx^{-1}}}_\CB\\
		&\ls \No{\bx^{-1} \K{1-\M(-t)}f(t,-it\hb\na_x)}_\CB \\
		&\qquad +\No{\bx^{-1} \K{1-\M(-t)}\Bk{\bx ,f(t,-it\hb\na_x)}\bx^{-1}}_\CB=: \SC + \SD.
	\end{align*}
	On the one hand, since
	\begin{align}
		\abs{\bx^{-1} (1-\M(t))} \ls \frac{1}{\lg t\hb\rg^{1/2}},
	\end{align}
	we have
	\begin{equation*}
		\SC \ls \frac{\|f(t)\|_{L^\I_x}}{\lg t\hb \rg^{1/2}}.
	\end{equation*}
	On the other hand, Lemma \ref{lem:commutator basic} implies
	\begin{align}
		\SD&\ls \frac{1}{\lg t\hb \rg^{1/2}} \K{\No{\Bk{\bx^{1/2}, f(t,-it\hb\na_x)}}_{\CB}
			+ \No{\Bk{\bx^{1/2}, \Bk{\bx^{1/2}, f(t,-it\hb\na_x)}}}_{\CB}} \\
		&\ls \frac{1}{\lg t\hb \rg^{1/2}} \K{ |t\hb|\No{\na f(t)}_{L^\I_x} + |t\hb|^2\No{\na^2 f(t)}_{L^\I_x} }.
	\end{align}
	Hence, we obtain
	\begin{equation}\label{eq:bound 2'}
		\CP[f](t) \ls \frac{1}{\lg t\hb \rg^{1/2}} 
		\K{\|f(t)\|_{L^\I_x} + |t\hb| \|\na f(t)\|_{L^\I_x}
		 +  |t\hb|^2 \|\na^2 f(t)\|_{L^\I_x}}.
	\end{equation}

\smallskip \noindent  \underline{\textbf{Step 3: Proof of \eqref{eq:Af 3}.}}
The proof of \eqref{eq:Af 3} is almost the same as the proof of \eqref{eq:Af 2}.
The only difference is that we estimate $\SD$ by
\begin{align}
	\SD \ls \frac{1}{\lg t\hb\rg^{1/2}} \No{\Bk{f(t,t\hb\xi),\lg\na\rg}}_\CB \le \frac{t\hb\|\na f(t)\|_{\CF L^1}}{\lg t\hb\rg^{1/2}} \No{\Bk{x, \lg\na\rg}}_\CB \ls \frac{t\hb\|\na f(t)\|_{\CF L^1}}{\lg t\hb\rg^{1/2}}.
\end{align}
\end{proof}

As a corollary of Lemma \ref{lem:key estimate}, we obtain the following estimates.
\begin{corollary}\label{cor:decay rate}
	Under the same assumptions as Proposition \ref{prop:wave operator boundedness}, we have the following estimates:
	\begin{equation}\label{eq:A nabla V}
		\CP[\na V](t)\ls
		\begin{dcases}
			\bt^{-7/3+\de/3} \quad \text{for all }  t \in [0,\min(\hb^{-2},T)], \\
			\lg t\hb\rg^{-1/2} \bt^{-2} + \sqrt{\hb}\bt^{-5/2+\de}\quad \text{for all } t \in [0,T]
		\end{dcases}
	\end{equation}
	and
	\begin{align}\label{eq:A nabla2 V}
		\CP[\na^2 V](t)
		\ls\begin{dcases}
		\bt^{-19/6+\de} \quad \text{for all } t \in [0,\min(\hb^{-2},T)], \\
		\lg t\hb\rg^{-1/2} \bt^{-3+\de} + \min(\sqrt{\hb}\bt^{-3+\de}, \hb^{-1} \bt^{-7/2+\de}) \quad \text{for all } t \in [0,T].
		\end{dcases}
	\end{align}
\end{corollary}

\begin{proof}
	\smallskip \noindent  \underline{\textbf{Step 1: Proof of \eqref{eq:A nabla V}.}}
	When $0\le t \le 1/\hb^2$, by the assumption and Lemma \ref{lem:key estimate}, we have
	\begin{equation}
		\begin{aligned}
  \CP[\na V]&\ls \No{\na^2 V(t)}_{\CF L^1}^{1/3}\K{\No{\na V(t)}_{L^\I_x} + |t\hb|\No{\na^2 V(t)}_{L^\I_x} + |t\hb|^2\No{\na^3 V(t)}_{L^\I_x}}^{2/3}  \\
			&\ls \frac{1}{\bt^{1-\de/3}} \K{\frac{1}{\bt^{2}} +\frac{|t\hb|}{\bt^{3-\de}} + \frac{|t\hb|^2}{\bt^{4-\de}} }^{2/3} \\
			&\ls \frac{1}{\bt^{1-\de/3}} \K{\frac{1}{\bt^{2}} +\frac{\sqrt{t}}{\bt^{3-\de}} + \frac{t}{\bt^{4-\de}} }^{2/3} 
			\ls \frac{1}{\bt^{7/3-\de/3}}.
		\end{aligned}
	\end{equation}
	Moreover, by the assumption and Lemma \ref{lem:key estimate}, we have
	\begin{align}
		\CP[\na V](t) &\ls \frac{1}{\lg t\hb\rg^{1/2}} \K{\No{\na V(t)}_{L^\I_x} + |t\hb|\No{\na^2 V(t)}_{\CF L^1}}
		\ls \frac{1}{\lg t\hb\rg^{1/2}\bt^{2}} + \frac{\sqrt{\hb}}{\bt^{5/2-\de}}.
	\end{align}
	
	\smallskip \noindent  \underline{\textbf{Step 2: Proof of \eqref{eq:A nabla2 V}.}}
	When $0\le t \le 1/\hb^2$, by the assumption and Lemma \ref{lem:key estimate}, we have
	\begin{equation}
		\begin{aligned}
			\CP[\na^2 V](t)&\ls \No{\na^3 V(t)}_{\CF L^1}^{1/3}\K{\No{\na^2 V(t)}_{L^\I_x} + |t\hb|\No{\na^3 V(t)}_{L^\I_x} + |t\hb|^2\No{\na^4 V(t)}_{L^\I_x}}^{2/3} \\
			&\ls \frac{1}{\bt^{7/6-\de/3}} \K{\frac{1}{\bt^{3-\de}} + \frac{\sqrt{t}}{\bt^{4-\de}} + \frac{t}{\bt^{4-\de}}}^{2/3}
			\ls \frac{1}{\bt^{19/6-\de}}. 
		\end{aligned}
	\end{equation}
	Finally, by the assumption and Lemma \ref{lem:key estimate}, we have
	\begin{equation}
		\begin{aligned}
			\CP[\na^2 V](t)&\ls \frac{1}{\lg t\hb\rg^{1/2}}
			\K{\No{\na^2 V(t)}_{L^\I_x} + |t\hb|\No{\na^3 V(t)}_{\CF L^1}} \\
			&\ls \frac{1}{\lg t\hb\rg^{1/2}}
			\K{\frac{1}{\bt^{3-\de}} + \frac{|t\hb|}{\bt^{7/2-\de}} }
			\le \frac{1}{\lg t\hb\rg^{1/2}\bt^{3-\de}} + \frac{\sqrt{\hb}}{\bt^{3-\de}}.
		\end{aligned}
	\end{equation}
	Another estimate is
	\begin{align}
		\CP[\na^2 V](t)	&\ls \frac{1}{\lg t\hb\rg^{1/2}}
		\K{\frac{1}{\bt^{3-\de}} + \frac{|t\hb|}{\hb^{3/2}\bt^{4-\de}} }
		\le \frac{1}{\lg t\hb\rg^{1/2}\bt^{3-\de}}
		+ \frac{1}{\hb\bt^{7/2-\de}}.
\end{align}
\end{proof}

\subsection{Decomposition of the commutator}\label{subsec:decomposition}
In the proof of Proposition \ref{prop:wave operator boundedness}, it is very important to consider the following commutator:
\begin{align}
	\Bk{x, e^{i\Ps(t,-i\hb\na_x)}\W_{V}(t)}
	= \Bk{x, e^{i\Ps(t,-i\hb\na_x)}}\W_{V}(t) +  e^{i\Ps(t,-i\hb\na_x)}\Bk{x,\W_{V}(t)} =: A(t)+ B(t).
\end{align}
A direct calculation shows the decomposition of $A(t)$ and $B(t)$.
\begin{lemma}\label{lem:decomposition of A}
	For $A(t)$, we can write $A(t)=A_1(t)+A_2(t)+A_3(t)$, where
	\begin{equation}\label{eq:A}
		\begin{aligned}
			&A_1(t):= - e^{i\Ps(t,-i\hb\na_x)} \W_{V}(t) \int_0^t \ta (\na V)(\ta,-i\ta\hb\na_x) d\ta, \\
			&A_2(t):= -\frac{i}{\hb} e^{i\Ps(t,-i\hb\na_x)}  \W_{V}(t)\int_0^t\int_0^{\ta}d\ta' d\ta \U_{V}(\ta')^* \Bk{V(\ta'), \ta(\na V)(\ta,-i\ta\hb\na_x) } \U_{V}(\ta'),  \\
			&A_3(t):=- \frac{i}{\hb} e^{i\Ps(t,-i\hb\na_x)}  \W_{V}(t)\int_0^t \int_{\ta}^td\ta' d\ta
			\U_{V}(\ta')^* \Bk{V(\ta'), \ta(\na V)(\ta,-i\ta\hb\na_x) } \U_{V}(\ta').  
		\end{aligned}
	\end{equation}
\end{lemma}
\begin{proof}
	Since $\Dk{x,\CF^{-1}m(\xi)\CF}=\CF^{-1} i\na_\xi m(\xi) \CF$, we have
	\begin{equation}
		\begin{aligned}
			A(t) &= - e^{i\Ps(t,-i\hb\na_x)} \int_0^t \ta (\na V)(\ta,-i\ta\hb\na_x) d\ta \W_{V}(t)\\
			&= - e^{i\Ps(t,-i\hb\na_x)} \W_{V}(t) \int_0^t \ta (\na V)(\ta,-i\ta\hb\na_x) d\ta \\
			&\quad - e^{i\Ps(t,-i\hb\na_x)} \int_0^t \ta \Bk{(\na V)(\ta,-i\ta\hb\na_x), \W_{V}(t)} d\ta.
	\end{aligned}
	\end{equation}
	By \eqref{eq:commutator f W} in Lemma \ref{lem:commutator wave},  we obtain
	\begin{equation}
		\begin{aligned}
			A(t)&= - e^{i\Ps(t,-i\hb\na_x)} \W_{V}(t) \int_0^t \ta (\na V)(\ta,-i\ta\hb\na_x) d\ta \\
			&\quad - \frac{i}{\hb} e^{i\Ps(t,-i\hb\na_x)}  \W_{V}(t)\int_0^t
			\int_0^{\ta} \U_{V}(\ta')^* \Bk{V(\ta'), \ta(\na V)(\ta,-i\ta\hb\na_x) } \U_{V}(\ta') d\ta' d\ta\\
			&\quad - \frac{i}{\hb} e^{i\Ps(t,-i\hb\na_x)}  \W_{V}(t)\int_0^t
			\int_{\ta}^t \U_{V}(\ta')^* \Bk{V(\ta'), \ta(\na V)(\ta,-i\ta\hb\na_x) } \U_{V}(\ta') d\ta' d\ta.
		\end{aligned}
	\end{equation}
\end{proof}

Moreover, we have the following decomposition of $B(t)$.
\begin{lemma}\label{lem:decomposition of B}
	For $B(t)$, we can write $B(t)=B_1(t) + B_2(t)$, where
	\begin{equation}\label{eq:B}
		\begin{aligned}
			&B_1(t)= e^{i\Ps(t,-i\hb\na_x)} \W_{V}(t)\int_0^t \U(\ta)^* \ta (\na V)(\ta) \U(\ta) d\ta, \\
			&B_2(t) := \frac{i}{\hb} e^{i\Ps(t,-i\hb\na_x)} \W_{V}(t) \int_0^t d\ta\int_0^{\ta} d\ta'  \\
			&\qquad \qquad \cdot \U_{V}(\ta')^*
			\Bk{V(\ta'), \U(\ta'-\ta) \ta (\na V)(\ta) \U(\ta-\ta')} \U_{V}(\ta').
		\end{aligned}	
	\end{equation}
\end{lemma}
\begin{proof}[Proof of Lemma \ref{lem:decomposition of B}]
By \eqref{lem:Duhamel}, we have
	\begin{equation*}
		\begin{aligned}
			&\U_{V}(\ta)^* \ta (\na V)(\ta) \U_{V}(\ta)
			= \U(\ta)^* \ta (\na V) \U(\ta) \\
			&\quad  + \frac{i}{\hb} \int_0^{\ta} \U_{V}(\ta')^* \Bk{V(\ta'), \U(\ta'-\ta) \ta (\na V)(\ta) \U(\ta-\ta')} \U_{V}(\ta') d\ta'.
		\end{aligned}
	\end{equation*}
	Hence, by \eqref{eq:commutator x W}, we have
	\begin{equation*}
		\begin{aligned}
			B(t) &= e^{i\Ps(t,-i\hb\na_x)} \W_{V}(t)\int_0^t \U_{V}(\ta)^* \ta (\na V)(\ta) \U_{V}(\ta) d\ta \\
			&= e^{i\Ps(t,-i\hb\na_x)} \W_{V}(t)\int_0^t \U(\ta)^* \ta (\na V)(\ta) \U(\ta) d\ta \\
			&\quad + \frac{i}{\hb} e^{i\Ps(t,-i\hb\na_x)} \W_{V}(t) \int_0^t d\ta  \int_0^{\ta} d\ta'\\
			&\qquad \qquad \qquad \qquad \U_{V}(\ta')^*\Bk{V(\ta'), \U(\ta'-\ta) \ta (\na V)(\ta) \U(\ta-\ta')} \U_{V}(\ta').
		\end{aligned}
	\end{equation*}
\end{proof}

\subsection{Proof of Proposition \ref{prop:wave operator boundedness} when \texorpdfstring{$0\le s\le 1$}{0=<s=<1}}\label{subsec:si=1}
In this section, we prove Proposition \ref{prop:wave operator boundedness} when $0 \le s \le 1$.
\subsubsection{Setup of the proof.}
When $s=0$, it is clear that 
\begin{align}
	\No{\bx^s e^{i\Ps(t,-i\hb\na_x)} \W_{V}(t) \bx^{-s}}_{\CB}=1.
\end{align}
Hence, by the complex interpolation, it suffices to consider the case $s=1$. 
Since
\begin{align}
	\No{\bx e^{i\Ps(t,-i\hb\na_x)} \W_{V}(t) \bx^{-1}}_{\CB}
	&\ls 1
	+ \No{\Bk{x,  e^{i\Ps(t,-i\hb\na_x)} \W_{V}(t)}\bx^{-1}}_\CB 
	=:1 + \CQ(t),
\end{align}
our goal is to prove 
\begin{equation}\label{eq:first goal}
\sup_{0\le t\le T} \CQ(t)\ls 1.
\end{equation}
By Lemmas \ref{lem:decomposition of A} and \ref{lem:decomposition of B}, we obtain
\begin{align}
	\CQ(t) \le \|(A_1(t)+B_1(t))\bx^{-1}\|_\CB + \|(A_2(t)+B_2(t))\bx^{-1}\|_\CB + \|A_3(t)\bx^{-1}\|_\CB.
\end{align}
In the sequel, we estimate these three terms separately.

\subsubsection{Cancellation of $A_1(t)$ and $B_1(t)$}
In this section, we prove the cancellation of $A_1(t)$ and $B_1$, that is, we show
\begin{align}\label{eq:A1+B1}
		\boxed{\|(A_1(t) + B_1(t))\bx^{-1}\|_\CB \ls 1.}
\end{align}

\begin{proof}[Proof of \eqref{eq:A1+B1}]
By Lemmas \ref{lem:decomposition of A}, \ref{lem:decomposition of B}, and MDFM decomposition \eqref{eq:MDFM}, we obtain
\begin{equation*}
	\begin{aligned}
	&\No{(A_1(t) + B_1(t))\bx^{-1}}_{\CB}\\
	&\quad \le \int_0^t \ta \No{\Ck{(\na V)(\ta,-i\ta\hb\na_x) - \U(\ta)^* (\na V)(\ta) \U(\ta)}\bx^{-1}}_{\CB} d\ta \\
	&\quad \le \int_0^{1/\hb^2} \ta \CP[\na V](\ta) d\ta
	 + \int_{1/\hb^2}^\I \ta \CP[\na V](\ta) d\ta =: \SA+ \SB,
	\end{aligned}
\end{equation*}
where $\CP[f](t)$ is defined by \eqref{eq:Af}.
By Corollary \ref{cor:decay rate}, we have
\begin{equation}
\int_0^{1/\hb^2} \ta \CP[\na V](\ta) d\ta \ls \int_0^{1/\hb^2} \frac{1}{\bta^{4/3-\de/3}} d\ta \ls 1.
\end{equation}
Again, Corollary \ref{cor:decay rate} implies
\begin{equation}
\int_{1/\hb^2}^\I \ta \CP[\na V](\ta) d\ta
\ls \int_{1/\hb^2}^\I \frac{1}{\sqrt{\hb}|\ta|^{3/2}}d\ta+ \int_{1/\hb^2}^\I \frac{\sqrt{\hb}}{|\ta|^{3/2-\de}}d\ta \ls 1.
\end{equation} 
\end{proof}

\subsubsection{Cancellation of $A_2(t)$ and $B_2(t)$}
Next, we prove the cancellation of $A_2(t)$ and $B_2(t)$. Namely, we show
	\begin{align}\label{eq:A2+B2}
		\boxed{\|(A_2(t) + B_2(t))\bx^{-1}\|_\CB \ls 1.}
	\end{align}
	
\begin{proof}[Proof of \eqref{eq:A2+B2}]
\smallskip \noindent  \underline{\textbf{Step 0: Setup of the proof.}}
By Lemmas \ref{lem:easy}, \ref{lem:decomposition of A}, \ref{lem:decomposition of B}, and MDFM decomposition \eqref{eq:MDFM},  we have
\begin{align*}
	&\No{(A_2(t) + B_2(t))\bx^{-1}}_{\CB} \ls \frac{1}{\hb} \int_0^t d\ta \int_{0}^{\ta} d\ta' \log(\ta'+2)
	\Big\|\Big[V(\ta'), \ta(\na V)(\ta,-i\ta\hb\na_x)\\
	&\qquad \qquad \qquad \qquad  - \U(\ta'-\ta) \ta (\na V)(\ta) \U(\ta-\ta') \Big] \lg\J(\ta')\rg^{-1} \Big\|_\CB\\
	&\quad \le \frac{1}{\hb} \int_0^\I d\ta \int_0^{\ta} d\ta' \log(\ta'+2) \No{\Bk{V(\ta'), \U(\ta') \CV(\ta)\U(\ta')^* } \lg\J(\ta')\rg^{-1} }_\CB,
\end{align*}
where
\begin{equation}
	\CV(\ta) := \ta(\na V)(\ta,-i\ta\hb\na_x) - \M(-\ta) \ta (\na V)(\ta,-i\ta\hb\na_x) \M(\ta).
\end{equation}
Note that 
\begin{align}
	&\Bk{V(\ta'), \U(\ta') \CV(\ta)\U(\ta')^* }\lg\J(\ta')\rg^{-1} \\ 
	&\quad = \Bk{V(\ta'), \U(\ta') \CV(\ta)\U(\ta')^* \lg\J(\ta')\rg^{-1}}
	- \U(\ta')\CV(\ta)\U(\ta')^* \Bk{V(\ta'), \lg\J(\ta')\rg^{-1}} \\
	&\quad =: C(\ta,\ta') + D(\ta,\ta').
\end{align}
Hence, we have
\begin{align}
	\No{(A_2(t) + B_2(t))\bx^{-1}}_{\CB}
	&\ls \frac{1}{\hb} \int_0^\I \int_0^{\ta} \log(\ta'+2) \No{C(\ta,\ta')}_\CB d\ta' d\ta  \\
	&\qquad +\frac{1}{\hb} \int_0^\I \int_0^{\ta} \log(\ta'+2) \No{D(\ta,\ta')}_\CB d\ta' d\ta =: \SC + \SD.
\end{align}

\smallskip \noindent  \underline{\textbf{Step 1: Estimate of $\SD$.}}
We decompose $\SD$ into
\begin{align}
	\SD &= \frac{1}{\hb}\int_0^{1/\hb^2} \int_0^{\ta} \log(\ta'+2) \No{D(\ta,\ta')}_\CB d\ta' d\ta
	         +\frac{1}{\hb}\int_{1/\hb^2}^\I \int_0^{\ta} \log(\ta'+2)\No{D(\ta,\ta')}_\CB d\ta' d\ta\\
	    &=:\SD_1 + \SD_2.
\end{align}
\smallskip \noindent  \textbf{$\blacklozenge$ Estimate of $\SD_1$.}
Note that
\begin{align}
	\|D(\ta,\ta')\|_\CB
	&\le \No{\CV(\ta)\U(\ta')^*\lg\J(\ta')\rg^{-1}}_\CB \No{\lg\J(\ta')\rg \Bk{V(\ta'), \lg\J(\ta')\rg^{-1}}}_{\CB} \\
	&= \No{\CV(\ta)\bx^{-1}}_\CB \No{\bx \Bk{\U(\ta')^*V(\ta')\U(\ta'), \bx^{-1}}}_{\CB}.
\end{align}
On the one hand, by Corollary \ref{cor:decay rate}, we have
\begin{align}
	\No{\CV(\ta)\bx^{-1}}_\CB = \ta \CP[\na V](\ta) \ls \frac{1}{\bta^{4/3-\de}}
\end{align}
for $0\le \ta\le \hb^{-2}$.
On the other hand, it follows from the assumption that
\begin{equation}\label{eq:UVUx-1}
\begin{aligned}
&\No{\bx \Bk{\U(\ta')^*V(\ta')\U(\ta'), \bx^{-1}}}_{\CB} \\
&\quad \le \No{\Bk{\U(\ta')^*V(\ta')\U(\ta'), \bx^{-1}}}_{\CB} + \No{x\Bk{\U(\ta')^*V(\ta')\U(\ta'), \bx^{-1}}}_{\CB} \\
&\quad \le \No{\Bk{\U(\ta')^*V(\ta')\U(\ta'), \bx^{-1}}}_{\CB} + \No{\Bk{x, \U(\ta')^*V(\ta')\U(\ta')}\bx^{-1}}_{\CB} \\
&\qquad \qquad +  \No{\Bk{x\bx^{-1}, \U(\ta')^*V(\ta')\U(\ta')}}_{\CB} \\
&\quad \ls \No{\Bk{x, \U(\ta')^*V(\ta')\U(\ta')}}_{\CB} = \ta'\hb\No{\na V(\ta')}_{L^\I_x} \ls \frac{\hb}{\lg\ta'\rg}.
\end{aligned}
\end{equation}
Therefore, we obtain
\begin{equation}
	\SD_1 \ls \int_0^{1/\hb^2} \frac{1}{\bta^{4/3-\de}} \K{\int_0^{\ta} \frac{\log(\ta'+2)}{\lg\ta'\rg}d\ta' }d\ta \ls 1.
\end{equation}

\smallskip \noindent  \textbf{$\blacklozenge$ Estimate of $\SD_2$.}
By Corollary \ref{cor:decay rate} and \eqref{eq:UVUx-1}, we have
\begin{align}
	\SD_2 \ls \frac{1}{\sqrt{\hb}}\int_{1/\hb^2}^\I \frac{(\log(\ta+2))^2}{|\ta|^{3/2} } d\ta
	       + \int_{1/\hb^2}^\I \frac{\sqrt{\hb}(\log(\ta+2))^2}{\bta^{3/2-\de}} d\ta \ls 1.
\end{align}

\smallskip \noindent  \underline{\textbf{Step 2: Estimate of $\SC$.}}
First, we decompose $\SC$ into
\begin{align}
	\SC&=\frac{1}{\hb}\int_0^{1/\hb^2}\int_0^{\ta} \log(\ta'+2)\No{C(\ta,\ta')}_\CB d\ta' d\ta
	 + \frac{1}{\hb}\int_{1/\hb^2}^\I \int_0^{\ta} \log(\ta'+2)\No{C(\ta,\ta')}_\CB d\ta' d\ta\\
	 &=:\SC_1 + \SC_2.
\end{align}
By Lemma \ref{lem:commutator basic}, we have
\begin{align}
	\|C(\ta,\ta')\|_\CB&\le \No{\na V(\ta')}_{\CF L^1} \No{\Bk{x,\U(\ta')\CV(\ta)\U(\ta')^*\lg\J(\ta')\rg^{-1}}}_{\CB} \\
	&= \No{\na V(\ta')}_{\CF L^1} \No{\Bk{\J(-\ta'),\CV(\ta)\bx^{-1}}}_{\CB} \\
	&\le \No{\na V(\ta')}_{\CF L^1} \No{\Bk{x,\CV(\ta)\bx^{-1}}}_{\CB} + \ta'\hb\No{\na V(\ta')}_{\CF L^1} \No{\Bk{\na,\CV(\ta)\bx^{-1}}}_{\CB} \\
	&=: \SE(\ta,\ta') + \SF(\ta,\ta').
\end{align}

\smallskip \noindent  \textbf{$\blacklozenge$ Estimate of $\SC_1$.}
For $\SE(\ta,\ta')$, it follows from the assumption and Corollary \ref{cor:decay rate} that
\begin{align}
	\SE(\ta,\ta')
	&\ls \frac{\ta^2\hb}{\lg\ta'\rg^{2-\de}}\CP[\na^2 V](\ta)
	\ls \frac{\ta^2\hb}{\lg\ta'\rg^{2-\de}} \frac{1}{\bta^{19/6-\de}} \ls \frac{\hb}{\lg\ta'\rg^{2-\de}} \frac{1}{\bta^{7/6-\de}}
\end{align}
for all $0\le\ta\le\hb^{-2}$.
For $\SF(\ta,\ta')$, we obtain
\begin{equation}\label{eq:F bound}
\begin{aligned}
\SF(\ta,\ta') 	&\ls \ta'\hb\No{\na V(\ta')}_{\CF L^1} \K{\No{\Bk{\na,\CV(\ta)}\bx^{-1}}_{\CB} + \No{\CV(\ta)\bx^{-2}}_\CB } \\
	&\le \ta\ta'\hb\No{\na V(\ta')}_{\CF L^1} \K{\No{\na^2 V(\ta)}_{L^\I_x} + \CP[\na V](\ta) } \\
	&\ls \frac{\hb}{\lg\ta'\rg^{1-\de}} \K{\frac{1}{\bta^{2-\de}} + \frac{1}{\bta^{4/3-\de/3}}}.
\end{aligned}
\end{equation}
Hence, we obtain $\SC_1 \ls 1$.

\smallskip \noindent  \textbf{$\blacklozenge$ Estimate of $\SC_2$.}
For $\SE(\ta,\ta')$, it follows from the assumption and Corollary \ref{cor:decay rate} that
\begin{align}
	\SE(\ta,\ta')
	&\ls \frac{\ta^2\hb}{\lg\ta'\rg^{2-\de}}\CP[\na^2 V](\ta) 
	\ls \frac{\hb}{\lg\ta'\rg^{2-\de}} \K{\frac{1}{\sqrt{\hb}|\ta|^{3/2-\de}} + \frac{1}{|\ta|^{7/6-\de}}}.
\end{align}
Again, by using \eqref{eq:F bound}, we obtain $\SC_2\ls 1$.
\end{proof}

\subsubsection{Estimates of $A_3(t)$ and the conclusion}
Finally, we estimate $A_3$.
	By Lemmas \ref{lem:commutator basic} and \ref{lem:decomposition of A}, we obtain
	\begin{align}
		&\|A_3(t)\|_\CB
		\le \frac{1}{\hb} \int_0^t \int_{\ta}^t \No{\Bk{V(\ta'), \ta(\na V)(\ta,-i\ta\hb\na_x) }}_\CB d\ta'd\ta \\
		&\quad \le \frac{1}{\hb} \int_0^t \int_{\ta}^t
		     \No{\na V(\ta')}_{\CF L^1} \No{\Bk{x, \ta(\na V)(\ta,-i\ta\hb\na_x) }}_\CB d\ta'd\ta \\
		&\quad \le \int_0^t \int_0^{\ta'}
		\No{\na V(\ta')}_{\CF L^1} \ta^2 \No{\na^2 V(\ta) }_{L^\I_x}  d\ta d\ta'
		\ls \int_0^t  \frac{1}{\lg \ta'\rg^{2-\de}} \int_0^{\ta'} \frac{1}{\bta^{1-\de}} d\ta d\ta' \ls 1.
\end{align}

From the above argument, we obtain
\begin{align}
	\CQ(t) \le \|(A_1(t)+B_1(t))\bx^{-1}\|_\CB + \|(A_2(t)+B_2(t))\bx^{-1}\|_\CB + \|A_3(t)\bx^{-1}\|_\CB \ls 1,
\end{align}
which completes the proof of Proposition \ref{prop:wave operator boundedness} when $s=1$ (therefore, all $s$ such that $0\le s\le 1$).

\subsection{Proof of Proposition \ref{prop:wave operator boundedness} when \texorpdfstring{$1\le s \le 2$}{1=<s=<2}}\label{subsec:si=2}
\subsubsection{Setup of the proof}
Since we have already proved the case $s=1$, we can complete the proof if we prove Proposition \ref{prop:wave operator boundedness} when $s=2$.
We have
\begin{align}
	&\No{\bx^{2} e^{i\Ps(t,-i\hb\na_x)} \W_{V}(t) \bx^{-2}}_{\CB} \\
	&\quad \ls \No{\bx e^{i\Ps(t,-i\hb\na_x)} \W_{V}(t) \bx^{-1}}_{\CB}
	+ \No{\bx\Dk{x,e^{i\Ps(t,-i\hb\na_x)} \W_{V}(t)} \bx^{-2}}_{\CB} \\
	&\quad \ls 1 + \No{\bx\Dk{x,e^{i\Ps(t,-i\hb\na_x)} \W_{V}(t)} \bx^{-2}}_{\CB}=:1+\CR(t),
\end{align}
where we used the boundedness when $s=1$.
For the second term $\CR(t)$, we have
\begin{align}
	\CR(t) &\le \No{\bx (A_1(t) + B_1(t)) \bx^{-2}}_\CB + \No{\bx (A_2(t) + B_2(t))\bx^{-2}}_\CB + \No{\bx A_3(t) \bx^{-2}}_\CB\\
	& =: \CR_1(t) + \CR_2(t) + \CR_3(t),
\end{align}
where $A_1(t), A_2(t), A_3(t)$ and $B_1(t), B_2(t)$ are the same notations as in Lemmas \ref{lem:decomposition of A} and \ref{lem:decomposition of B}.

\subsubsection{Estimate of $\CR_1(t)$}
By Lemmas \ref{lem:decomposition of A}, \ref{lem:decomposition of B}, and MDFM decomposition \eqref{eq:MDFM}, we obtain
\begin{equation*}
	\begin{aligned}
		\CR_1(t)&\ls \int_0^t \ta \No{\bx \Ck{(\na V)(\ta,-i\ta\hb\na_x) - \U(\ta)^* (\na V)(\ta) \U(\ta)}\bx^{-2}}_{\CB} d\ta \\
		&\le \int_0^{\I} \ta \No{\Ck{(\na V)(\ta,-i\ta\hb\na_x) - \M(-\ta)(\na V)(\ta,-i\ta\hb\na_x)\M(\ta)}\bx^{-1}}_{\CB} d\ta \\
		&\quad + \int_0^\I \ta^2 \hb \No{\Ck{(\na^2 V)(\ta,-i\ta\hb\na_x) - \M(-\ta)(\na^2 V)(\ta,-i\ta\hb\na_x)\M(\ta)}\bx^{-2}}_{\CB} d\ta \\
		&=: \SA+\SB.
	\end{aligned}
\end{equation*}
By \eqref{eq:A1+B1}, we have $\SA \ls 1$.
For $\SB$, by Corollary \ref{cor:decay rate}, we obtain
\begin{align}
	&\ta^2 \hb \No{\Ck{(\na^2 V)(\ta,-i\ta\hb\na_x) - \M(-\ta)(\na^2 V)(\ta,-i\ta\hb\na_x)\M(\ta)}\bx^{-2}}_{\CB}\\
	&\quad \ls \frac{\sqrt{\hb}}{|\ta|^{1/2}\bta^{1-\de}} + \frac{1}{\bta^{3/2-\de}}
\end{align}
for all $\ta\in[0,\I)$, which implies $\SB \ls 1$.

\subsubsection{Estimate of $\CR_2(t)$.}
\smallskip \noindent  \underline{\textbf{Step 0: Setup.}}
By Lemmas \ref{lem:easy}, \ref{lem:decomposition of A}, and \ref{lem:decomposition of B}, we have
\begin{align*}
	&\No{\bx\K{A_2(t) + B_2(t)}\bx^{-2}}_{\CB} \ls \frac{1}{\hb} \int_0^\I d\ta \int_0^{\ta} d\ta' \lg\ta'\rg^{3\de} \\
	&\qquad \qquad \qquad \qquad \cdot \No{\lg\J(\ta')\rg \Bk{V(\ta'), \U(\ta') \CV(\ta)\U(\ta')^* } \lg\J(\ta')\rg^{-2} }_\CB,
\end{align*}
where
\begin{equation}
	\CV(\ta) := \ta(\na V)(\ta,-i\ta\hb\na_x) - \M(-\ta) \ta (\na V)(\ta,-i\ta\hb\na_x) \M(\ta).
\end{equation}
It suffices to consider
\begin{align}
	\SD &:= \frac{1}{\hb} \int_0^\I d\ta \int_0^{\ta} d\ta' \lg\ta'\rg^{3\de}\No{\Bk{\J(\ta'), \Bk{V(\ta'), \U(\ta') \CV(\ta)\U(\ta')^* }} \lg\J(\ta')\rg^{-2} }_\CB.
\end{align}
By the Jacobi identity \eqref{eq:Jacobi} for the commutator, we obtain
\begin{align}
	&\No{\Bk{\J(\ta'), \Bk{V(\ta'), \U(\ta') \CV(\ta)\U(\ta')^* }} \lg\J(\ta')\rg^{-2} }_\CB \\
	&\quad = \No{\Bk{x, \Bk{\U(\ta')^* V(\ta') \U(\ta'), \CV(\ta)} } \bx^{-2} }_\CB \\
	&\quad \le \No{\Bk{\Bk{x, \U(\ta')^* V(\ta') \U(\ta')}, \CV(\ta) }\bx^{-2}}_\CB
	+ \No{\Bk{\Bk{x, \CV(\ta)},\U(\ta')^* V(\ta') \U(\ta') }\bx^{-2}}_\CB \\
	&=: \SD_1 + \SD_2.
\end{align}

\smallskip \noindent \textbf{Step 1: Estimate of $\SD_1$.}
	We take $0<\delta<1/24$ and consider $0\leq\ta'\leq\ta\leq t\leq T$.
	By the definition of $\SD_1$ and the commutator identity
	\[
	[x,\U(\ta')^*V(\ta')\U(\ta')]=i\ta'\hb\U(\ta')^*(\na V)(\ta')\U(\ta'),
	\]
	we have
	\begin{align*}
	\SD_1&\leq \ta'\hb\bigl\|[\na V(\ta'),\U(\ta')\CV(\ta)\U(\ta')^*]\lg\J(\ta')\rg^{-1}\bigr\|_{\CB} \\
		&\leq \ta'\hb\bigl\|[\na V(\ta'),\U(\ta')\CV(\ta)\bx^{-1}\U(\ta')^*]\bigr\|_{\CB} \\
		&\quad+\ta'\hb\|\CV(\ta)\bx^{-1}\|_{\CB}\bigl\|\bx[\U(\ta')^*\na V(\ta')\U(\ta'),\bx^{-1}]\bigr\|_{\CB}.
	\end{align*}
	For the first term, Lemma \ref{lem:commutator basic} gives
	\begin{align*}
		\bigl\|[\na V(\ta'),\U(\ta')\CV(\ta)\bx^{-1}\U(\ta')^*]\bigr\|_{\CB}
		\lesssim \|\na^2V(\ta')\|_{\mathcal{F}L^1}\bigl\|[x,\U(\ta')\CV(\ta)\bx^{-1}\U(\ta')^*]\bigr\|_{\CB}.
	\end{align*}
	Since
	$\U(\ta')^*x\U(\ta')=x-i\ta'\hb\na$, we obtain
	\begin{align*}
		&\bigl\|[x,\U(\ta')\CV(\ta)\bx^{-1}\U(\ta')^*]\bigr\|_{\CB}
		=\bigl\|[x-i\ta'\hb\na,\CV(\ta)\bx^{-1}]\bigr\|_{\CB} \\
		&\qquad \leq\|[x,\CV(\ta)]\bx^{-1}\|_{\CB}+\ta'\hb\|[\na,\CV(\ta)]\bx^{-1}\|_{\CB}
		+\ta'\hb\|\CV(\ta)\na\bx^{-1}\|_{\CB}.
	\end{align*}
	Here, the definition of $\CV(\ta)$ and the MDFM decomposition imply
	\begin{align*}
		&[x,\CV(\ta)]=i\ta^2\hb\bigl\{(\na^2V)(\ta,-i\ta\hb\na_x) -\M(-\ta)(\na^2V)(\ta,-i\ta\hb\na_x)\M(\ta)\bigr\}, \\
		&[\na,\CV(\ta)]=-\ta\U(\ta)^*(\na^2V)(\ta)\U(\ta).
	\end{align*}
	Consequently,
	\begin{align*}
		&\|[x,\CV(\ta)]\bx^{-1}\|_{\CB}=\ta^2\hb\CP[\na^2V](\ta), \\
		&\|[\na,\CV(\ta)]\bx^{-1}\|_{\CB}\leq\ta\|\na^2V(\ta)\|_{L_x^\infty}, \\
		&\|\CV(\ta)\na\bx^{-1}\|_{\CB}\leq\|\CV(\ta)\bx^{-1}\|_{\CB}\|x\bx^{-2}\|_{L_x^\infty}\leq\ta\CP[\na V](\ta).
	\end{align*}
	For the second term, Lemma \ref{lem:commutator basic}, applied as in \eqref{eq:UVUx-1}, yields
	\begin{align*}
		&\bigl\|\bx[\U(\ta')^*\na V(\ta')\U(\ta'),\bx^{-1}]\bigr\|_{\CB} \\
		&\qquad \lesssim\bigl\|[\U(\ta')^*\na V(\ta')\U(\ta'),\bx^{-1}]\bigr\|_{\CB}+\bigl\|[x,\U(\ta')^*\na V(\ta')\U(\ta')]\bx^{-1}\bigr\|_{\CB} \\
		&\qquad\qquad +\bigl\|[\U(\ta')^*\na V(\ta')\U(\ta'),x\bx^{-1}]\bigr\|_{\CB} \\
		&\qquad \lesssim\bigl\|[x,\U(\ta')^*\na V(\ta')\U(\ta')]\bigr\|_{\CB}
		=\ta'\hb\|\na^2V(\ta')\|_{L_x^\infty}.
	\end{align*}
	Combining these estimates, we obtain
	\begin{align*}
		\SD_1
		&\lesssim\ta'\hb\|\na^2V(\ta')\|_{\mathcal{F}L^1}\bigl\{\ta^2\hb\CP[\na^2V](\ta)
		 +\ta\ta'\hb\|\na^2V(\ta)\|_{L_x^\infty}+\ta\ta'\hb\CP[\na V](\ta)\bigr\} \\
		&\quad+(\ta')^2\hb^2\|\na^2V(\ta')\|_{L_x^\infty}\ta\CP[\na V](\ta) \\
		&\lesssim\frac{\hb^2\ta^2}{\lg\ta'\rg^{2-\delta}}\CP[\na^2V](\ta)+\frac{\hb^2}{\lg\ta'\rg^{1-\delta}}
		\bigl\{\ta\|\na^2V(\ta)\|_{L_x^\infty}
		+\ta\CP[\na V](\ta)\bigr\}.
	\end{align*}
	By Corollary \ref{cor:decay rate} and
	$\hb\lg\ta\hb\rg^{-1/2}
	\leq\lg\ta\rg^{-1/2}$, we have
	\begin{align*}
		&\hb\ta^2\CP[\na^2V](\ta)
		\lesssim\frac{\hb\ta^2}{\lg\ta\hb\rg^{1/2}\lg\ta\rg^{3-\delta}}+\frac{\ta^2}{\lg\ta\rg^{7/2-\delta}}
		\lesssim\lg\ta\rg^{-3/2+\delta}, \\
		&\hb\ta\CP[\na V](\ta)
		\lesssim\frac{\hb\ta}{\lg\ta\hb\rg^{1/2}\lg\ta\rg^2}+\frac{\hb^{3/2}\ta}{\lg\ta\rg^{5/2-\delta}}
		\lesssim\lg\ta\rg^{-3/2+\delta}.
	\end{align*}
	Moreover,
	\[
	\int_0^\ta\lg\ta'\rg^{-2+4\delta}d\ta'\lesssim1,
	\qquad
	\int_0^\ta\lg\ta'\rg^{-1+4\delta}d\ta'\lesssim\lg\ta\rg^{4\delta}.
	\]
	Therefore,
	\begin{align*}
		\frac1\hb\int_0^td\ta\int_0^\ta d\ta'\lg\ta'\rg^{3\delta}\SD_1
		&\lesssim\hb\int_0^t\ta^2\CP[\na^2V](\ta)d\ta \\
		&\quad+\hb\int_0^t\lg\ta\rg^{4\delta}\bigl\{\ta\|\na^2V(\ta)\|_{L_x^\infty}+\ta\CP[\na V](\ta)\bigr\}d\ta \\
		&\lesssim\int_0^\infty\bigl(\lg\ta\rg^{-3/2+\delta}+\lg\ta\rg^{-2+5\delta}+\lg\ta\rg^{-3/2+5\delta}\bigr)d\ta
		\lesssim1.
	\end{align*}
	All implicit constants are independent of $\hb\in(0,1]$, $t$, and $T$.

\smallskip \noindent  \underline{\textbf{Step 2: Estimate of $\SD_2$.}}
Next, we consider $\SD_2$. We have
\begin{align}
	\SD_2&\le \Ck{\No{V(\ta')}_{L^\I_x} 
	+ \No{\lg\J(\ta')\rg V(\ta')\lg\J(\ta')\rg^{-1}}_{\CB} } \No{\Bk{x, \CV(\ta)}\bx^{-1}}_\CB \\
	&\ls \K{\No{V(\ta')}_{L^\I_x} + \ta'\hb \No{\na V(\ta')}_{L^\I_x} }
	     \ta^2 \hb\CP[\na^2 V](\ta).
\end{align}
When $0\le \ta \le \hb^{-2}$, then we have by Corollary \ref{cor:decay rate}
\begin{align}
	\SD_2 \ls \frac{\hb}{\lg\ta'\rg \bta^{7/6-\de}}.
\end{align}
When $\hb^{-2}\le \ta <\I$, then we have by Corollary \ref{cor:decay rate}
\begin{align}
	\SD_2 \ls \frac{\hb}{\lg\ta'\rg} \K{ \frac{1}{\sqrt{\hb}|\ta|^{3/2-\de}} + \frac{1}{|\ta|^{7/6-\de}} }.
\end{align}
From the above, we obtain
\begin{align}
	\frac{1}{\hb} \int_0^\I d\ta \int_0^{\ta} d\ta' \lg\ta'\rg^{3\de} \SD_2\ls 1.
\end{align}

\subsubsection{Estimate of $\CR_3(t)$}
Finally, we estimate $\CR_3(t)$.
By Lemmas \ref{lem:commutator basic}, \ref{lem:easy}, and \ref{lem:decomposition of A}, we have
\begin{align}
	&\CR_3(t)
	\ls \frac{1}{\hb} \int_0^t \int_{\ta}^t \lg\ta'\rg^{3\de} \No{\lg\J(\ta')\rg \Bk{V(\ta'), \ta(\na V)(\ta,-i\ta\hb\na_x) } \lg\J(\ta')\rg^{-2}}_\CB d\ta'd\ta.
\end{align}
Hence, it suffices to estimate
\begin{align}
	\SE(t) := \frac{1}{\hb} \int_0^t \int_{\ta}^t \lg\ta'\rg^{3\de}\No{\Bk{\J(\ta'), \Bk{V(\ta'), \ta(\na V)(\ta,-i\ta\hb\na_x) } }\lg\J(\ta')\rg^{-2}}_\CB d\ta'd\ta.
\end{align}
By the Jacobi identity \eqref{eq:Jacobi} and the assumption of Proposition \ref{prop:wave operator boundedness}, we have
\begin{align}
	&\No{\Bk{\J(\ta'), \Bk{V(\ta'), \ta(\na V)(\ta,-i\ta\hb\na_x) } }\lg\J(\ta')\rg^{-2}}_\CB \\
	&\quad \le \No{\Bk{\Bk{\J(\ta'), V(\ta')}, \ta(\na V)(\ta,-i\ta\hb\na_x) }}_\CB \\
	&\qquad \qquad + \No{\Bk{\Bk{\J(\ta'), \ta(\na V)(\ta,-i\ta\hb\na_x) }, V(\ta')} }_\CB \\
	&\quad = \ta'\hb \No{\Bk{(\na V)(\ta'), \ta(\na V)(\ta,-i\ta\hb\na_x) }}_\CB
	+ \ta\hb \No{\Bk{\ta(\na^2 V)(\ta,-i\ta\hb\na_x), V(\ta')} }_\CB \\
	&\quad \le \ta'\hb \No{\na^2 V(\ta')}_{\CF L^1} \No{\Bk{x, \ta(\na V)(\ta,-i\ta\hb\na_x) }}_\CB \\
	&\qquad \qquad + \ta\hb \No{\na V(\ta')}_{\CF L^1} \No{\Bk{x, \ta(\na^2 V)(\ta,-i\ta\hb\na_x)} }_\CB \\
	&\quad = \ta'\ta^2 \hb^2 \No{\na^2 V(\ta')}_{\CF L^1} \No{\na^2 V(\ta)}_{L^\I_x} 
	 + \ta^3\hb^2 \No{\na V(\ta')}_{\CF L^1} \No{\na^3 V(\ta)}_{L^\I_x} \\
	&\quad \ls \frac{\hb^2}{\lg\ta'\rg^{2-\de} \bta^{1-\de}}.
\end{align}
Therefore, we obtain $\SE(t)\ls 1$.
From the above, we complete the proof of Proposition \ref{prop:wave operator boundedness} with $s=2$. 

\section{Single commutator estimates with general Schatten--\texorpdfstring{$r$}{r} norm}\label{sec:single}
To prove the key estimate in Proposition \ref{prop:a priori}, we need to obtain some commutator estimates. This is because we need to estimate derivatives of the density function (see \eqref{eq:density deriv formula 0} and \eqref{eq:density deriv formula}). 
We will find a solution to \eqref{eq:NLH} such that the density function $\rhh(\gah(t))$ belongs to the function space $\CY^{a,b}_T$ with two small parameters $a,b>0$, where $\CY^{a,b}_T$ is defined by
\[
\CY^{a,b}_T := \{\ze: \ze \text{ is real-valued and } \|\ze\|_{\CY^{a,b}_T} < \I\},
\]
where 
\begin{equation}\label{eq:density condition}
	\begin{aligned}
		\|\ze\|_{\CY^{a,b}_T}&:=\sup_{0\le t \le T} \bigg\{\No{\ze(t)}_{L^1_x} + \bt^3\No{\ze(t)}_{L^\I_x} + \bt^{1-a}\No{\na \ze(t)}_{L^1_x} + \bt^{4-a}\No{\na \ze(t)}_{L^{\I}_x} \\
		&\qquad \qquad + \bt^{4-a} \hb^{3/2} \No{\na \ze(t)}_{\CF L^1} + \bt^{7/2-b} \No{\na^2 \ze(t)}_{L^2_x} 
		 \bigg\}.
	\end{aligned}
\end{equation}
Throughout this section, we assume the following.
\begin{assumption}\label{ass:1}
	In this section, we always assume that $d=3$, $\si\in(3/2,2)$ and $r \in [1,\I]$.
	Let $w(x)$ satisfy \eqref{eq:ass}.
	Assume that $\de>0$ is a small constant given in Proposition \ref{prop:wave operator boundedness}.
	Suppose that small numbers $a, b \in (0,\de/100)$ satisfy $7b/8<a<b$.
	Let $T>0$, $\ze\in \CY^{a,b}_T$ be real-valued, and $V := w \ast \ze(t)$. 
	Define $\gah(t):= \U_V(t)\gah_0 \U_V(t)^*$ for $\gah_0 \in \CX^{\si}_\hb$ (see \eqref{eq:initial data class} for the definition of $\CX^\si_\hb$).
\end{assumption}

\subsection{Bounds of densities and potentials}
Since we often need to estimate the potential $w\ast \ze(t)$,
we collect necessary estimates here.

\begin{lemma}\label{lem:Lp bound}
	Under Assumption \ref{ass:1}, we have
	\begin{align}\label{eq:Lp bound}
		&\No{\ze(t)}_{L^{r,1}_x} \ls \frac{\|\ze\|_{\CY^{a,b}_T}}{\bt^{3/r'}}, 
		\quad \No{\na \ze(t)}_{L^{r,1}_x} \ls \frac{\|\ze\|_{\CY^{a,b}_T}}{\bt^{1+3/r'-a}}
	\end{align}
	for all $0\le t \le T$ and $r \in (1,\I)$.
\end{lemma}
\begin{proof}
	By the assumptions, we have
	\begin{align}
		&\No{\ze(t)}_{L^1_x}\le \|\ze\|_{\CY^{a,b}_T}, \quad \No{\ze(t)}_{L^\I_x} \le \frac{\|\ze\|_{\CY^{a,b}_T}}{\bt^3}, \\
		&\No{\na \ze(t)}_{L^1_x}\le \frac{\|\ze\|_{\CY^{a,b}_T}}{\bt^{1-a}}, \quad \No{\na \ze(t)}_{L^\I_x} \le \frac{\|\ze\|_{\CY^{a,b}_T}}{\bt^{4-a}}.
	\end{align}
	Hence, by the interpolation, we obtain \eqref{eq:Lp bound} for all $1<r < \I$.
\end{proof}

\begin{lemma}\label{lem:potential deriv}
	Under Assumption \ref{ass:1}, we have
	\begin{align}
		\No{\pl^\al w\ast \zeta(t) }_{L^\infty} \ls_\ep
		\begin{dcases}
			&\bt^{-1}\|\ze\|_{\CY^{a,b}_T} \quad \mbox{if } |\al|=0, \\
			&\bt^{-2}\|\ze\|_{\CY^{a,b}_T}\quad \mbox{if } |\al|=1, \\
			&\bt^{-3+\ep}\|\ze\|_{\CY^{a,b}_T} \quad \mbox{if } |\al|=2, \\
			&\bt^{-4+a+\ep}\|\ze\|_{\CY^{a,b}_T} \quad \mbox{if } |\al|=3, \\
			&\bt^{-4+a+\ep}\|\ze\|_{\CY^{a,b}_T} \quad \mbox{if } |\al|=4
		\end{dcases}
	\end{align}
	for all $0\le t \le T$, where we can choose $\ep\in(0,\de/100)$ arbitrarily small, independently of $a,b$.
	Moreover, we sometimes need the estimate
	\begin{align}
		\No{\na^3 V(t)}_{L^{3}_x} \ls \bt^{-3+b/2+\ep}\|\ze\|_{\CY^{a,b}_T}
	\end{align}
	for all $0\le t \le T$.
	All implicit constants are independent of $T>0$.
\end{lemma}
\begin{proof}
	When $|\al|=0$, the O'Neil theorem \cite{O'Neil 1963} (the Young convolution inequality for Lorentz norms) and Lemma \ref{lem:Lp bound} imply
	\begin{align}
		\No{w\ast \ze(t)}_{L^\I_x} 
		\ls \No{\ze(t)}_{L^{3/2,1}_x}
		\ls \frac{\|\ze\|_{\CY_T^{a,b}}}{\bt}.
	\end{align}
	Similarly, when $|\al|=1$, we have 
	\begin{align}
		&\No{\pl^\al w\ast \ze(t)}_{L^\I_x} 
		\ls \No{\ze(t)}_{L^{3,1}_x}
		\ls \frac{\|\ze\|_{\CY^{a,b}_T}}{\bt^{2}}.
	\end{align}
	
	Note that $|\na|^2 w, |\na|^3 w \in L^p_x$ for all $p\in (1,\I)$.
	Therefore, we obtain $|\na|^{s} w \in L^p_x$ for all $p\in(1,\I)$ and $s \in [2,3]$.
	When $|\al|=2$, the O'Neil theorem and Lemma \ref{lem:Lp bound} imply
	\begin{align}
		&\No{\pl^\al w\ast \ze(t)}_{L^{\I}_x} 
		\ls \No{|\na|^{2+\ep/2} w\ast \ze(t)}_{L^{6/\ep,1}_x} 
		\ls \No{|\na|^{2+\ep/2}w}_{L^{\frac{6}{6-\ep}}_x} \No{\ze(t)}_{L^{3/\ep}_x}\ls \frac{\|\ze\|_{\CY^{a,b}_T}}{\bt^{3-\ep}}.
	\end{align}
	Similarly, when $|\al|=3$, we have
	\begin{align}\label{eq:b}
	   &\No{\pl^\al w\ast \ze(t)}_{L^\I_x}
	        \ls \No{|\na|^{3+\ep/2} w \ast \ze(t)}_{L^{6/\ep,1}_x}
	        \ls \No{\na \ze(t)}_{L^{3/\ep}_x} \le \frac{\|\ze\|_{\CY^{a,b}_T}}{\bt^{4-a-\ep}}.
	\end{align}
	The estimate \eqref{eq:b} works when $|\al|=4$.
	
	Finally, we have
	\begin{align}
		\No{\na^3 V(t)}_{L^{3}_x} \ls \No{|\na| \ze(t)}_{L^{\frac{3}{1+\ep}}_x}
		\ls \No{\ze(t)}_{L^{\frac{6}{1+4\ep}}_x}^{1/2} \No{|\na|^2 \ze(t)}_{L^2_x}^{1/2} \ls \frac{\|\ze\|_{\CY^{a,b}_T}}{\bt^{3-b/2-\ep}}.
	\end{align}
\end{proof}

\begin{lemma}\label{lem:potential deriv 2}
	Under Assumption \ref{ass:1}, we have
	\begin{align}
		\No{\pl^\al w\ast \zeta(t) }_{\CF L^1} \ls_\ep
		\begin{dcases}
			&\bt^{-2+a/2+\ep}\|\ze\|_{\CY^{a,b}_T} \quad \mbox{if } |\al|=1, \\
			&\bt^{-3+3b/4+\ep}\|\ze\|_{\CY^{a,b}_T} \quad \mbox{if } |\al|=2, \\
			&\min(\bt^{-7/2+b}, \hb^{-3/2}\bt^{-4+a} )\No{\ze}_{\CY^{a,b}_T}\quad \mbox{if } |\al|=3
		\end{dcases}
	\end{align}
	for all $0\le t\le T$, where we can choose $\ep\in(0,\de/100)$ arbitrarily small, independently of $a,b$.
	Moreover, the implicit constant is independent of $T>0$.
\end{lemma}

\begin{proof}
	Since $w\in L^p_x$ for $p>3$ and $\na w \in L^p_x$ for $p>3/2$, we have $|\na|^{1/2+\ep}w \in L^2_x$ by the complex interpolation.
When $|\al|=1$, we have
	\begin{align}
		\|\pl^\al w\ast \ze(t)\|_{\CF L^1}
		&\ls \No{|\xi|^{1/2+\ep}\wh{w}(\xi)|\xi|^{1/2-\ep}\wh{\ze}(\xi)}_{L^1_\xi}
		\ls \||\na|^{1/2-\ep}\ze(t)\|_{L^2_x}\\
		&\ls \No{\ze(t)}_{L^2_x}^{1/2+\ep} \No{\na\ze(t)}_{L^2_x}^{1/2-\ep} \le \frac{\|\ze\|_{\CY^{a,b}_T}}{\bt^{2-a/2-\ep}}.
	\end{align}
Similarly, when $|\al|=2$, we have by Lemma \ref{lem:Lp bound}
	\begin{align}
		\|\pl^\al w\ast \ze(t)\|_{\CF L^1}
		&\ls \No{\ze(t)}_{L^2_x}^{1/4+\ep/2} \No{|\na|^2 \ze(t)}_{L^2_x}^{3/4-\ep/2}
		\le \frac{\|\ze\|_{\CY^{a,b}_T}}{\bt^{3-3b/4-\ep}}.
	\end{align}
	Finally, when $|\al|=3$, we have
	\begin{align}
		\|\pl^\al w\ast \ze(t)\|_{\CF L^1}
		\ls \No{|\na|^2\ze(t)}_{L^{2}_x}
		\le \frac{\|\ze\|_{\CY^{a,b}_T}}{\bt^{7/2-b}}.
	\end{align}	
	Moreover, we have another estimate:
	\begin{align}
		\|\pl^\al w\ast \ze(t)\|_{\CF L^1} \ls \No{\na \ze(t)}_{\CF L^1} \le \frac{\|\ze\|_{\CY^{a,b}_T}}{\hb^{3/2}\bt^{4-a}}.
	\end{align}
\end{proof}

\subsection{Single commutator estimates}
In this section, we collect various kinds of single commutator estimates with the general Schatten--$r$ norm.
In Section \ref{sec:double}, we focus on the double commutator estimates for the Hilbert--Schmidt norm. The biggest difference between these two cases is clear in Lemma \ref{lem:commutator basic}. Namely, we can improve \eqref{eq:commutator basic} into \eqref{eq:commutator basic 2}.

\subsubsection{An easy case without weights or derivatives}
The simplest commutator estimate is the following.
\begin{lemma}\label{lem:commutator 1}
	Under Assumption \ref{ass:1}, there exists a sufficiently small $R>0$ such that $\|\ze\|_{\CY^{a,b}_T} \le R$ implies
	\begin{align}\label{eq:commutator first}
		&\sup_{0\le t\le T}\frac{1}{\bt} \No{\Dk{\frac{x}{\hb},\ga^\hb(t)}}_{\FS^r_\hb}
		+ \sup_{0\le t\le T}\No{\Dk{\na,\ga^\hb(t)}}_{\FS^r_\hb}\ls \No{\gah_0}_{\CX^\si_\hb}.
	\end{align}
\end{lemma}
\begin{proof}
	\smallskip \noindent  \underline{\textbf{Step 1: Estimate of the first term.}}
	We have
	\begin{align}
		\frac{1}{\bt}\No{\Dk{\xh, \gah(t)}}_{\FS^r_\hb}
		&\le \frac{1}{\bt}\bigg\|\Dk{\xh,\U(t)}\W_V(t) \ga_0^\hb \W_V(t)^* \U(t)^* \\
		&\qquad \qquad  + \U(t)\W_V(t) \ga_0^\hb \W_V(t)^* \Dk{\xh,\U(t)^*}\bigg\|_{\FS^r_\hb} \\
		&\quad +\frac{1}{\bt}\bigg\|\U(t) \Dk{\xh,\W_V(t)\ga_0^\hb \W_V(t)^*}  \U(t)^* \bigg\|_{\FS^r_\hb} =: \SA(t) + \SB(t).
	\end{align}
Clearly, we have
\begin{align}
		\SA(t)&\le \No{\Bk{\na,\gah(t) }}_{\FS^r_\hb}.
\end{align}
By Lemmas \ref{lem:commutator x W}, \ref{lem:commutator basic}, and \ref{lem:potential deriv 2}, we obtain
	\begin{align}
		\SB(t) &= \frac{1}{\bt}\No{\Dk{\xh,\W_V(t)\ga_0^\hb \W_V(t)^*}}_{\FS^r_\hb} \\
		&\le \frac{1}{\bt} \No{\Dk{\xh,\ga_0^\hb}}_{\FS^r_\hb}
		 + \frac{1}{\bt\hb} \No{\W_V(t)\int_0^t  \U_V(\ta)^*\Bk{\ta\na V(\ta), \gah(\ta)}\U_V(\ta)  d\ta\W_V(t)^*}_{\FS^r_\hb} \\
		&\le \No{\gah_0}_{\CX^\si_\hb} + \frac{1}{\bt} \int_0^t \bta^2 \|\na^2 V(\ta)\|_{\CF L^1} \frac{1}{\bta}\No{\Dk{\xh,\gah(\ta)}}_{\FS^r_\hb}d\ta \\
		&\ls \No{\gah_0}_{\CX^\si_\hb}
		       + \frac{1}{\bt} \int_0^t \frac{R}{\bta^{1-\ep-3b/4}} d\ta \sup_{0\le\ta\le t} \frac{1}{\bta}\No{\Dk{\xh,\gah(\ta)}}_{\FS^r_\hb} \\
		&\ls \No{\gah_0}_{\CX^\si_\hb}
		+ R \sup_{0\le \ta\le t} \frac{1}{\bta}\No{\Dk{\xh,\gah(\ta)}}_{\FS^r_\hb},
	\end{align}
	where we used sufficiently small $\ep>0$.
	From the above, we obtain
	\begin{align}
		\sup_{0\le \ta \le t} \frac{1}{\bta}\No{\Bk{\xh, \gah(\ta)}}_{\FS^r_\hb}
		\ls \No{\gah_0}_{\CX^\si_\hb}+ \sup_{0\le \ta\le t} \No{\Bk{\na,\gah(\ta) }}_{\FS^r_\hb}
	\end{align}
	if we choose sufficiently small $R>0$.
	
	\smallskip \noindent  \underline{\textbf{Step 2: Estimate of the second term.}}
By Lemmas \ref{lem:commutator nabla W}, \ref{lem:commutator basic}, and \ref{lem:potential deriv 2}, we have
\begin{align}
	&\No{\Dk{\na,\gah(t) }}_{\FS^r_\hb}
	= \No{\Bk{\na,\W_V(t)\gah_0\W_V(t)^* }}_{\FS^r_\hb}
 \le \No{\gah_0}_{\CX_\hb^\si} + \frac{1}{\hb}\int_0^t \No{\Bk{\na V(\ta), \gah(\ta)}}_{\FS^r_\hb} d\ta \\
 &\quad \le \No{\gah_0}_{\CX_\hb^\si} + \int_0^t \No{\na^2 V(\ta)}_{\CF L^1} \No{\Bk{\xh, \gah(\ta)}}_{\FS^r_\hb} d\ta \\
 &\quad \ls \No{\gah_0}_{\CX_\hb^\si} + \int_0^t \frac{R}{\bta^{2-\ep-3b/4}} d\ta 
 \sup_{0\le\ta\le t} \frac{1}{\bta}\No{\Bk{\xh, \gah(\ta)}}_{\FS^r_\hb} \\
 &\quad \ls \No{\gah_0}_{\CX_\hb^\si} + R \sup_{0\le\ta\le t} \frac{1}{\bta}\No{\Bk{\xh, \gah(\ta)}}_{\FS^r_\hb}.
\end{align}

\smallskip \noindent  \underline{\textbf{Step 3: Conclusion.}}
By the above argument, we obtain
\begin{align}
	\sup_{0\le\ta\le t} \frac{1}{\bta}\No{\Bk{\xh, \gah(\ta)}}_{\FS^r_\hb} \ls \No{\gah_0}_{\CX_\hb^\si} + R \sup_{0\le \ta\le t} \frac{1}{\bta}\No{\Bk{\xh, \gah(\ta)}}_{\FS^r_\hb}.
\end{align}
Therefore, choosing sufficiently small $R>0$, we obtain
\begin{align}
	\sup_{0\le \ta \le t} \frac{1}{\bta}\No{\Bk{\xh, \gah(\ta)}}_{\FS^r_\hb} \ls \No{\gah_0}_{\CX_\hb^\si}.
\end{align}
Since all implicit constants are independent of $t$, we obtain 
\begin{align}
	\sup_{0\le t\le T} \frac{1}{\bt}\No{\Bk{\xh, \gah(t)}}_{\FS^r_\hb}
	\ls \No{\gah_0}_{\CX_\hb^\si}.
\end{align}
Finally, by the calculation in \textbf{Step 2}, we obtain
\begin{align}
\No{\Dk{\na,\gah(t) }}_{\FS^r_\hb} \ls \No{\gah_0}_{\CX^\si_\hb} + \sup_{0\le t\le T} \frac{1}{\bt}\No{\Bk{\xh, \gah(t)}}_{\FS^r_\hb}
\ls \No{\gah_0}_{\CX^\si_\hb}.
\end{align}
\end{proof}

\begin{remark}\label{rmk:byproduct 1}
	As a byproduct of the above proof, we obtain
	\begin{align}
			&\No{\Dk{\xh,\W_V(t)\ga_0^\hb \W_V(t)^*}}_{\FS^r_\hb} \ls 
			\bt^{\ep+3b/4} \No{\gah_0}_{\CX^\si_\hb} \le \bt^{a} \No{\gah_0}_{\CX^\si_\hb}
	\end{align}
	for sufficiently small $\ep>0$, where we used $3b/4<a$. Moreover, we have
	\begin{align}
		\No{\Bk{\na,\W_V(t)\ga_0^\hb \W_V(t)^*}}_{\FS^r_\hb} \ls \No{\gah_0}_{\CX^\si_\hb}.
	\end{align}
\end{remark}

\subsubsection{A useful lemma}
The following lemma may seem tricky at first, but we will soon find it to be quite useful.
\begin{lemma}\label{lem:commutator 1.6}
	Under Assumption \ref{ass:1} with $r\ge 2$, there exists a sufficiently small $R>0$ such that $\|\ze\|_{\CY^{a,b}_T} \le R$ implies
	\begin{equation}\label{eq:hojo}
		\sup_{0\le t\le T}\frac{\bt^{2-7b/8}}{\hb} \No{\Jt^\si \Bk{\na V(t), \gah(t)} \Jt^\si}_{\FS^r_\hb} \ls \No{\gah_0}_{\CX^\si_\hb}.
	\end{equation}
\end{lemma}

\begin{proof}
	\smallskip \noindent  \underline{\textbf{Step 0: Setup of the proof.}}
	Note that
	\begin{align}
		&\No{\Jt^{\si} \Bk{\na V(t), \gah(t)}\Jt^{\si} }_{\FS^r_\hb}
		\ls \No{\Bk{\na V(t), \Jt^\si \gah(t) \Jt^\si}}_{\FS^r_\hb} \\
		&\qquad \qquad + \No{\Dk{\Jt^\si, \na V(t)}\Jt^{-\si}}_{\CB}
		\No{\Jt^\si \gah(t) \Jt^\si}_{\FS^r_\hb}
		=: \SA+ \SB.
	\end{align}
	By Lemma \ref{lem:potential deriv}, for any $\ep \in (0,\de/100)$,
	\begin{align}\label{eq:Hesssian V bound}
		\No{\na^2 V(t)}_{L^\I_x} \ls \frac{1}{\bt^{3-\ep}}.
	\end{align}
	
	\smallskip \noindent  \underline{\textbf{Step 1: Estimate of $\SA$.}}
	First, by Lemmas \ref{lem:commutator basic} and \ref{lem:potential deriv 2}, we have
	\begin{align}
		\SA &\le \hb \No{\na^2 V(t)}_{\CF L^1} \No{\Dk{\xh, \Jt^\si \gah(t)\Jt^\si}}_{\FS^r_\hb}\\
		&\ls \frac{R\hb}{\bt^{3-\ep-3b/4}} \No{\Dk{\xh, \Jt^\si \gah(t)\Jt^\si}}_{\FS^r_\hb}.
	\end{align}
	Note that by \eqref{eq:J}
	\begin{align}
		&\No{\Dk{\xh, \Jt^\si \gah(t) \Jt^\si} }_{\FS^r_\hb}
		= \frac{1}{\hb}\No{\Bk{\J(-t), \bx^\si \W_V(t)\gah_0\W_V(t)^* \bx^\si} }_{\FS^r_\hb} \\
		&\quad \le \No{\Bk{\xh, \bx^\si \W_V(t)\gah_0\W_V(t)^* \bx^\si} }_{\FS^r_\hb} \\ 
		&\qquad \qquad + \bt \No{\Bk{\na, \bx^\si \W_V(t)\gah_0\W_V(t)^* \bx^\si} }_{\FS^r_\hb}=\SA_1+\SA_2.
	\end{align}
	
	\smallskip \noindent  \textbf{$\blacklozenge$ Estimate of $\SA_1$.}
	For $\SA_1$, by Lemmas \ref{lem:commutator x W}, \ref{lem:easy}, and \eqref{eq:Hesssian V bound}, we have
	\begin{align}
		\SA_1 &\ls \No{\Bk{\xh, \bx^\si \W_V(t)\gah_0\W_V(t)^* \bx^\si} }_{\FS^r_\hb} \\
		&\ls \bt^{4\ep}\No{\gah_0}_{\CX^\si_\hb}
		+ \frac{\bt^{4\ep}}{\hb}\int_0^t \bta^{4\ep}
		 \No{\Jta^\si \Bk{\ta \na V(\ta), \gah(\ta)} \Jta^\si}_{\FS^r_\hb} d\ta \\
		&\le \bt^{4\ep}\No{\gah_0}_{\CX^\si_\hb}
		+ \frac{\bt^{4\ep}}{\hb}\int_0^t \frac{\hb}{\bta^{1-7b/8-4\ep}} d\ta \\
		&\qquad \qquad \cdot \sup_{0\le\ta\le t} \frac{\bta^{2-7b/8}}{\hb} \No{\Jta^\si \Bk{\na V(\ta), \gah(\ta)} \Jta^\si}_{\FS^r_\hb} \\
		&\ls \bt^{4\ep}\No{\gah_0}_{\CX^\si_\hb}
		+ \bt^{7b/8+8\ep} \sup_{0\le\ta\le t} \frac{\bta^{2-7b/8}}{\hb}
		\No{\Jta^\si \Bk{\na V(\ta), \gah(\ta)} \Jta^\si}_{\FS^r_\hb}.
	\end{align}
	
	\smallskip \noindent  \textbf{$\blacklozenge$ Estimate of $\SA_2$.}
	For $\SA_2$, by Lemmas \ref{lem:easy}, \ref{lem:commutator nabla W}, and \ref{lem:potential deriv}, we have
	\begin{align}
	\SA_2 &\ls \bt \No{\bx^{\si-1}\W_V(t)\gah_0\W_V(t)^* \bx^\si}_{\FS^r_\hb}  + \bt \No{\bx^\si \Bk{\na, \W_V(t)\gah_0\W_V(t)^*} \bx^\si}_{\FS^r_\hb} \\
	&\ls \bt^{1+4\ep}\No{\gah_0}_{\CX^\si_\hb} + \frac{\bt^{1+4\ep}}{\hb} \int_0^t \bta^{4\ep} \No{\Jta^{\si} \Bk{\na V(\ta), \gah(\ta)} \Jta^\si}_{\FS^r_\hb} d\ta \\
	&\ls \bt^{1+4\ep}\No{\gah_0}_{\CX^\si_\hb} + \bt^{1+4\ep} \int_0^t \frac{1}{\bta^{2-7b/8-4\ep}} d\ta \\
	&\qquad \qquad \cdot \sup_{0\le\ta\le t}\frac{\bta^{2-7b/8}}{\hb} \No{\Jta^{\si} \Bk{\na V(\ta), \gah(\ta)} \Jta^\si}_{\FS^r_\hb} \\
	&\ls \bt^{1+4\ep}\No{\gah_0}_{\CX^\si_\hb} + \bt^{1+4\ep} \sup_{0\le\ta\le t}\frac{\bta^{2-7b/8}}{\hb} \No{\Jta^{\si} \Bk{\na V(\ta), \gah(\ta)} \Jta^\si}_{\FS^r_\hb}.
	\end{align}
	
	\smallskip \noindent  \textbf{$\blacklozenge$ Conclusion.}
	Combining Lemma \ref{lem:potential deriv 2} with the above estimates, we obtain
	\begin{align}
		\SA &\ls R\hb \bt^{-3+\ep +3b/4} \Big\{\bt^{1+4\ep}\No{\gah_0}_{\CX_\hb^\si} \\
		&\quad+\bt^{1+4\ep}
		\sup_{0\le\ta\le t}\frac{\bta^{2-7b/8}}{\hb} \No{\Jta^{\si} \Bk{\na V(\ta), \gah(\ta)} \Jta^\si}_{\FS^r_\hb} \Big\} \\
		&= R\hb \bt^{-2+3b/4+5\ep} \Big\{\No{\gah_0}_{\CX_\hb^\si} +\sup_{0\le\ta\le t}\frac{\bta^{2-7b/8}}{\hb} \No{\Jta^{\si} \Bk{\na V(\ta), \gah(\ta)} \Jta^\si}_{\FS^r_\hb} \Big\}.
	\end{align}

	\smallskip \noindent  \underline{\textbf{Step 2: Estimate of $\SB$.}}
	By Lemmas \ref{lem:easy}, \ref{lem:commutator basic}, \ref{lem:potential deriv}, and \eqref{eq:f(J)}, we have
	\begin{align}
		\SB &\ls \bt^{4\ep} \No{\gah_0}_{\CX_\hb^\si}\Big\{\No{\Bk{\Jt^{\si/2}, \na V(t)}}_{\CB} \\
		&\qquad	\qquad + \No{\Bk{\Jt^{\si/2}, \Bk{\Jt^{\si/2}, \na V(t)\Jt^{-\si}}}}_{\CB}\Big\}\\
		&=\bt^{4\ep} \No{\gah_0}_{\CX_\hb^\si}\Big\{\No{\Bk{\bx^{\si/2}, \U(t)^*\na V(t)\U(t)}}_{\CB} \\
		&\qquad	\qquad + \No{\Bk{\bx^{\si/2}, \Bk{\bx^{\si/2}, \U(t)^*\na V(t)\U(t)\bx^{-\si} }}}_{\CB}\Big\} \\
		&\ls \bt^{4\ep+1}\hb \No{\gah_0}_{\CX_\hb^\si}\Ck{\No{\na^2 V(t)}_{L^\I_x}
			+ t\hb\No{\na^3 V(t)\U(t)\bx^{-\si}}_\CB} \\
		&\ls \bt^{4\ep+1}\hb \No{\gah_0}_{\CX_\hb^\si}
		\Ck{R\bt^{-3+\ep}
			+ \No{\na^3 V(t)}_{L^3_x}} 
			\ls R\hb\bt^{-2+b/2+5\ep}\No{\gah_0}_{\CX_\hb^\si}.
	\end{align}

	\smallskip \noindent \underline{\textbf{Step 3: Conclusion.}}
	Clearly, we have $3b/4+5\ep<7b/8$ and $b/2 < 7b/8$ for sufficiently small $\ep>0$.
Therefore, we have
	\begin{align}
		&\sup_{0\le \ta\le t}\frac{\bta^{2-7b/8}}{\hb} \No{\Jta^{\si} \Bk{\na V(\ta), \gah(\ta)}\Jta^{\si} }_{\FS^r_\hb} \\
		&\quad \ls \No{\gah_0}_{\CX_\hb^\si} + R\sup_{0\le\ta\le t}\frac{\bta^{2-7b/8}}{\hb} \No{\Jta^{\si} \Bk{\na V(\ta), \gah(\ta)} \Jta^\si}_{\FS^r_\hb}.
	\end{align}
	Hence, choosing sufficiently small $R>0$, we get
	\begin{align}
		\sup_{0\le \ta\le t}\frac{\bta^{2-7b/8}}{\hb} \No{\Jta^{\si} \Bk{\na V(\ta), \gah(\ta)}\Jta^{\si} }_{\FS^r_\hb} \ls \No{\gah_0}_{\CX_\hb^\si}.
	\end{align}
	Because all constants are independent of $t$, we obtain the desired estimate \eqref{eq:hojo}.
\end{proof}

\subsubsection{An application of Lemma \ref{lem:commutator 1.6}}
By using Lemma \ref{lem:commutator 1.6}, we obtain the following:
\begin{lemma}\label{lem:commutator 1.7}
	Under Assumption \ref{ass:1} with $r \ge 2$, there exists a sufficiently small $R>0$ such that $\|\ze\|_{\CY^{a,b}_T} \le R$ implies
	\begin{align}
		&\No{\sd^\si \Bk{\na, \W_V(t)\gah_0 \W_V(t)^*}\sd^\si}_{\FS^r_\hb}
		\ls \bt^{a} \No{\gah_0}_{\CX^\si_\hb}, \label{eq:commutator nabla} \\
		&\No{\bx^\si \Bk{\xh, \W_V(t)\gah_0 \W_V(t)^*} \bx^\si}_{\FS^r_\hb}
		\ls \bt^a \No{\gah_0}_{\CX^\si_\hb}, \label{eq:commutator weight'}\\
		&\No{\bx^\si e^{i\Ps(t,-i\hb\na_x)} \Bk{\xh, \W_V(t)\gah_0 \W_V(t)^*} e^{-i\Ps(t,-i\hb\na_x)}\bx^\si}_{\FS^r_\hb}
		\ls \bt^a \No{\gah_0}_{\CX^\si_\hb}. \label{eq:commutator weight}
	\end{align}
\end{lemma}

\begin{proof}
The proofs of \eqref{eq:commutator nabla} and \eqref{eq:commutator weight'} are easy; it follows from a standard argument that we did in this paper, applying Lemmas \ref{lem:commutator nabla W}, \ref{lem:easy}, \ref{lem:commutator basic}, etc.	
	Now we consider \eqref{eq:commutator weight}. By Lemmas \ref{lem:commutator x W}, \ref{lem:commutator 1.6}, \ref{lem:easy}, and Proposition \ref{prop:wave operator boundedness}, we have
	\begin{align}
		&\No{\bx^\si e^{i\Ps(t,-i\hb\na_x)}\Bk{\xh,  \W_V(t)\gah_0\W_V(t)^* } e^{-i\Ps(t,-i\hb\na_x)}\bx^\si}_{\FS^r_\hb} \\
		&\quad \ls \No{\gah_0}_{\CX^\si_\hb}
		+ \frac{1}{\hb}\int_0^t \bta^{4\ep} \No{\Jta^\si \Bk{\ta \na V(\ta), \gah(\ta)} \Jta^\si}_{\FS^r_\hb} d\ta \\
		&\quad \ls \No{\gah_0}_{\CX^\si_\hb}
		+ \int_0^t \frac{\No{\gah_0}_{\CX_\hb^\si}}{\bta^{1-7b/8-4\ep}} d\ta
		\ls \bt^{7b/8+4\ep} \No{\gah_0}_{\CX^\si_\hb}.
	\end{align}
	Finally, by Assumption \ref{ass:1}, we have $7b/8+4\ep<a$ for sufficiently small $\ep>0$.
\end{proof}

\section{Double commutator estimates with the Hilbert--Schmidt norm}\label{sec:double}
In this section, we give very important double commutator bounds for the Hilbert--Schmidt norm. When we work in $\FS^2_\hb$--space, we can improve \eqref{eq:commutator basic} to \eqref{eq:commutator basic 2}.
Throughout this section, we assume the following.
\begin{assumption}\label{ass:2}
	In this section, we always assume that $d=3$ and $\si\in(3/2,2)$.
	Assume that $\de>0$ is a constant given in Proposition \ref{prop:wave operator boundedness}.
	Suppose that small numbers $a, b \in (0,\de/100)$ satisfy $7b/8<a<b$.
	Assume that $w(x)$ satisfies \eqref{eq:ass}.
	Let $\ze\in \CY^{a,b}_T$ be real-valued and $V := w \ast \ze(t)$. 
	Define $\gah(t):= \U_V(t)\gah_0 \U_V(t)^*$ for $\gah_0 \in \CX^{\si}_\hb$.
\end{assumption}

The goal of this section is to prove the following estimates:
\begin{lemma}\label{lem:commutator 6}
Under Assumption \ref{ass:2}, for all sufficiently small $\ep=\ep(\de,a,b)>0$, there exists a small $R>0$ such that $\|\ze\|_{\CY^{a,b}_T} \le R$ implies
	\begin{align}
		&\sup_{0\le t\le T} \bt^{-a-\ep} \No{\bx^{{\si/2}} \Dk{\frac{x_k}{\hb},\Dk{\frac{x_j}{\hb},\W_V(t)\gah_0 \W_V(t)^*}}\bx^{{\si/2}}}_{\Sp} \ls \No{\gah_0}_{\CX^\si_\hb},  \\
		&\sup_{0\le t\le T} \bt^{-\ep} \No{\bx^{{\si/2}} \Dk{\pl_{x_k},\Dk{\frac{x_j}{\hb},\W_V(t)\gah_0 \W_V(t)^*}}\bx^{{\si/2}}}_{\Sp} \ls \No{\gah_0}_{\CX^\si_\hb} , \\
		&\sup_{0\le t\le T} \bt^{-\ep} \No{\bx^{{\si/2}} \Bk{\pl_{x_k},\Bk{\pl_{x_j},\W_V(t)\gah_0 \W_V(t)^*}}\bx^{{\si/2}}}_{\Sp} \ls \No{\gah_0}_{\CX^\si_\hb}
	\end{align}
	for any $j,k=1,\dots,3$.
\end{lemma}

\begin{lemma}\label{lem:commutator 6.1}
	Under Assumption \ref{ass:2}, for all sufficiently small $\ep=\ep(\de,a,b)>0$, there exists $R>0$ such that $\|\ze\|_{\CY^{a,b}_T} \le R$ implies
	\begin{align}
		&\sup_{0\le t\le T} \bt^{-a-\ep} \No{\sd^{{\si/2}} \Dk{\frac{x_k}{\hb},\Dk{\frac{x_j}{\hb},\W_V(t)\gah_0 \W_V(t)^*}}\sd^{{\si/2}}}_{\Sp} \ls \No{\gah_0}_{\CX^\si_\hb},  \\
		&\sup_{0\le t\le T} \bt^{-\ep} \No{\sd^{{\si/2}} \Dk{\pl_{x_k},\Dk{\frac{x_j}{\hb},\W_V(t)\gah_0 \W_V(t)^*}}\sd^{{\si/2}}}_{\Sp} \ls \No{\gah_0}_{\CX^\si_\hb} , \\
		&\sup_{0\le t\le T} \bt^{-\ep} \No{\sd^{{\si/2}} \Bk{\pl_{x_k},\Bk{\pl_{x_j},\W_V(t)\gah_0 \W_V(t)^*}}\sd^{{\si/2}}}_{\Sp} \ls \No{\gah_0}_{\CX^\si_\hb}
	\end{align}
	for any $j,k=1,\dots,3$.
\end{lemma}

In the sequel, we only prove Lemma \ref{lem:commutator 6} because the same argument works well for Lemma \ref{lem:commutator 6.1}.

\subsection{Preliminary estimates}
In this section, we collect various kinds of necessary estimates to prove Lemma \ref{lem:commutator 6}.

\subsubsection{A variant of Lemma \ref{lem:commutator 1.6}}
\begin{lemma}\label{lem:commutator 1.6.1}
	Under Assumption \ref{ass:2}, for all sufficiently small $\ep=\ep(\de,a,b)>0$, there exists $R>0$ such that $\|\ze\|_{\CY^{a,b}_T} \le R$ implies
	\begin{equation}\label{eq:hojo 2}
		\sup_{0\le t\le T} \frac{\bt^{2-\ep}}{\hb}\No{\Jt^{\si/2} \Bk{\na V(t), \gah(t)} \Jt^{\si/2}}_{\FS^2_\hb} \ls \No{\gah_0}_{\CX^\si_\hb}.
	\end{equation}
\end{lemma}

\begin{proof}
	The proof proceeds in the same manner as the proof of Lemma \ref{lem:commutator 1.6}; however, to clarify what has changed, we present the complete proof again. 
	
	\smallskip \noindent  \underline{\textbf{Step 0: Setup of the proof.}}
	Note that
	\begin{align}
		&\No{\Jt^{{\si/2}} \Bk{\na V(t), \gah(t)}\Jt^{{\si/2}} }_{\FS^2_\hb}
		\ls \No{\Bk{\na V(t), \Jt^{\si/2} \gah(t) \Jt^{\si/2}}}_{\FS^2_\hb} \\
		&\qquad + \No{\Bk{\Jt^{\si/2}, \na V(t)}\Jt^{-{\si/2}}}_{\CB}
		\No{\Jt^{\si/2} \gah(t) \Jt^{\si/2}}_{\FS^2_\hb}
		=: \SA+ \SB.
	\end{align}
		By Lemma \ref{lem:potential deriv}, for any $\ep \in (0,\de/100)$,
	\begin{align}\label{eq:Hesssian V bound'}
		\No{\na^2 V(t)}_{L^\I_x} \ls \frac{1}{\bt^{3-\ep}}.
	\end{align}
	
	\smallskip \noindent  \underline{\textbf{Step 1: Estimate of $\SA$.}}
	First, by Lemmas \ref{lem:commutator basic} and \ref{lem:potential deriv}, we have
	\begin{align}
		\SA &\le \hb \No{\na^2 V(t)}_{L^\I} \No{\Dk{\xh, \Jt^{\si/2} \gah(t)\Jt^{\si/2}}}_{\FS^2_\hb}\\
		&\ls \frac{\hb}{\bt^{3-\ep}} \No{\Dk{\xh, \Jt^{\si/2} \gah(t)\Jt^{\si/2}}}_{\FS^2_\hb}.
	\end{align}
	Note that
	\begin{align}
		&\No{\Dk{\xh, \Jt^{\si/2} \gah(t) \Jt^{\si/2}} }_{\FS^2_\hb}
		= \frac{1}{\hb}\No{\Bk{\J(-t), \bx^{\si/2} \W_V(t)\gah_0\W_V(t)^* \bx^{\si/2}} }_{\FS^2_\hb} \\
		&\quad \le \No{\Bk{\xh, \bx^{\si/2} \W_V(t)\gah_0\W_V(t)^* \bx^{\si/2}} }_{\FS^2_\hb} \\ 
		&\qquad \qquad + \bt \No{\Bk{\na, \bx^{\si/2} \W_V(t)\gah_0\W_V(t)^* \bx^{\si/2}} }_{\FS^2_\hb}=\SA_1+\SA_2.
	\end{align}
	
	\smallskip \noindent  \textbf{$\blacklozenge$ Estimate of $\SA_1$.}
	Define $L(t):= \log(t+2)$ to avoid wasting space.
	For $\SA_1$, by Lemmas \ref{lem:commutator x W} and \ref{lem:easy}, we have
	\begin{align}
		\SA_1 &=\No{ \bx^{\si/2} \Bk{\xh,\W_V(t)\gah_0\W_V(t)^* } \bx^{\si/2}}_{\FS^2_\hb} \\
		&\ls (L(t))^2\No{\gah_0}_{\CX^{\si}_\hb}
		+ \frac{(L(t))^2}{\hb}\int_0^t (L(\ta))^2
		\No{\Jta^{\si/2} \Bk{\ta \na V(\ta), \gah(\ta)} \Jta^{\si/2}}_{\FS^2_\hb} d\ta \\
		&\le (L(t))^2\No{\gah_0}_{\CX^{\si}_\hb}
		+ \frac{(L(t))^2}{\hb}\int_0^t \frac{\hb(L(\ta))^2}{\bta^{1-10\ep}} d\ta \\
		&\qquad \qquad \cdot \sup_{0\le\ta\le t} \frac{\bta^{2-10\ep}}{\hb} \No{\Jta^{\si/2} \Bk{\na V(\ta), \gah(\ta)} \Jta^{\si/2}}_{\FS^2_\hb} \\
		&\ls \bt^{100\ep}\No{\gah_0}_{\CX^{\si}_\hb}  + \bt^{100\ep} \sup_{0\le\ta\le t} \frac{\bta^{2-10\ep}}{\hb}
		\No{\Jta^{\si/2} \Bk{\na V(\ta), \gah(\ta)} \Jta^{\si/2}}_{\FS^2_\hb}.
	\end{align}
	
	\smallskip \noindent  \textbf{$\blacklozenge$ Estimate of $\SA_2$.}
	For $\SA_2$, by Lemmas \ref{lem:easy}, \ref{lem:commutator nabla W}, and \ref{lem:potential deriv}, we have
	\begin{align}
		\SA_2 &\ls \bt \No{\bx^{{\si/2}-1}\W_V(t)\gah_0\W_V(t)^* \bx^{\si/2}}_{\FS^2_\hb} \\
		&\qquad + \bt \No{\bx^{\si/2} \Bk{\na, \W_V(t)\gah_0\W_V(t)^*} \bx^{\si/2}}_{\FS^2_\hb} \\
		&\ls \bt (L(t))^2\No{\gah_0}_{\CX^{\si}_\hb} \\
		&\qquad + \frac{\bt (L(t))^2}{\hb} \int_0^t (L(\ta))^{2} \No{\Jta^{{\si/2}} \Bk{\na V(\ta), \gah(\ta)} \Jta^{\si/2}}_{\FS^2_\hb} d\ta \\
		&\ls \bt (L(t))^2\No{\gah_0}_{\CX^{\si}_\hb} + \bt(L(t))^2 \int_0^t \frac{1}{\bta^{2-11\ep}} d\ta \\
		&\qquad \qquad \cdot \sup_{0\le\ta\le t}\frac{\bta^{2-10\ep}}{\hb} \No{\Jta^{{\si/2}} \Bk{\na V(\ta), \gah(\ta)} \Jta^{\si/2}}_{\FS^2_\hb} \\
		&\ls \bt^{1+\ep}\No{\gah_0}_{\CX^{\si}_\hb} + \bt^{1+\ep} \sup_{0\le\ta\le t}\frac{\bta^{2-10\ep}}{\hb} \No{\Jta^{{\si/2}} \Bk{\na V(\ta), \gah(\ta)} \Jta^{\si/2}}_{\FS^2_\hb}.
	\end{align}
	
	\smallskip \noindent  \textbf{$\blacklozenge$ Conclusion.}
	Combining Lemma \ref{lem:potential deriv} with the above estimates, we obtain
	\begin{align}
		\SA &\ls R \hb \bt^{-3+\ep} \Big\{\bt^{1+\ep}\No{\gah_0}_{\CX_\hb^\si} \\
		&\quad+\bt^{1+\ep}
		\sup_{0\le\ta\le t}\frac{\bta^{2-10\ep}}{\hb} \No{\Jta^{{\si/2}} \Bk{\na V(\ta), \gah(\ta)} \Jta^{\si/2}}_{\FS^2_\hb} \Big\} \\
		&= R\hb \bt^{-2+2\ep} \Big\{\No{\gah_0}_{\CX_\hb^\si}
		+\sup_{0\le\ta\le t}\frac{\bta^{2-10\ep}}{\hb} \No{\Jta^{{\si/2}} \Bk{\na V(\ta), \gah(\ta)} \Jta^{\si/2}}_{\FS^2_\hb} \Big\}.
	\end{align}
	
	\smallskip \noindent  \underline{\textbf{Step 2: Estimate of $\SB$ and conclusion.}}
	By Lemmas \ref{lem:easy}, \ref{lem:commutator basic}, and \ref{lem:potential deriv}, we have
	\begin{align}
		\SB &\ls (L(t))^2 \No{\gah_0}_{\CX_\hb^\si} t\hb \No{\na^2 V(t)}_{L^\I_x} \ls R\hb \bt^{-2+2\ep} \No{\gah_0}_{\CX_\hb^\si}. 
	\end{align}
	Therefore, we have
	\begin{align}
		&\sup_{0\le \ta\le t}\frac{\bta^{2-10\ep}}{\hb} \No{\Jta^{{\si/2}} \Bk{\na V(\ta), \gah(\ta)}\Jta^{{\si/2}} }_{\FS^2_\hb} \\
		&\quad \ls \No{\gah_0}_{\CX_\hb^\si} + R\sup_{0\le\ta\le t}\frac{\bta^{2-10\ep}}{\hb} \No{\Jta^{{\si/2}} \Bk{\na V(\ta), \gah(\ta)} \Jta^{\si/2}}_{\FS^2_\hb}.
	\end{align}
	Hence, choosing a sufficiently small $R>0$, we get
	\begin{align}
		\sup_{0\le \ta\le t}\frac{\bta^{2-10\ep}}{\hb} \No{\Jta^{{\si/2}} \Bk{\na V(\ta), \gah(\ta)}\Jta^{{\si/2}} }_{\FS^2_\hb} \ls \No{\gah_0}_{\CX_\hb^{\si}}.
	\end{align}
	Because all constants are independent of $t$, we obtain the desired estimate \eqref{eq:hojo 2}.
\end{proof}

\subsubsection{A variant of Remark \ref{rmk:byproduct 1}}
\begin{lemma}\label{lem:commutator 5}
	Under Assumption \ref{ass:2}, for all sufficiently small $\ep=\ep(\de,a,b)>0$, there exists $R>0$ such that $\|\ze\|_{\CY^{a,b}_T} \le R$ implies
	\begin{align}
		&\No{\Bk{\frac{x}{\hb},\bx^{\si/2}\W_V(t)\gah_0 \W_V(t)^* \bx^{\si/2}}}_{\FS^2_\hb} \ls \bt^\ep \No{\gah_0}_{\CX^\si_\hb},\\
		&\No{\Bk{\na,\bx^{\si/2}\W_V(t)\gah_0 \W_V(t)^*\bx^{\si/2} }}_{\FS^2_\hb} \ls \bt^\ep \No{\gah_0}_{\CX^\si_\hb}.
	\end{align}
\end{lemma}

\begin{proof}
Define $L(t):= \log(t+2)$ to avoid wasting space.
By Lemmas \ref{lem:commutator x W}, \ref{lem:easy}, and \ref{lem:commutator 1.6.1}, we have
	\begin{align}
		&\No{\bx^{\si/2}\Bk{\frac{x}{\hb},\W_V(t)\gah_0 \W_V(t)^*} \bx^{\si/2}}_{\FS^2_\hb} \\
		&\quad \ls (L(t))^2 \No{\gah_0}_{\CX^\si_\hb} + \frac{(L(t))^2}{\hb} \int_0^t (L(\ta))^2 \No{\Jta^{\si/2} \Bk{\ta\na V(\ta), \gah(\ta)} \Jta^{\si/2} }_{\FS^2_\hb} d\ta \\
		&\quad \ls  (L(t))^2 \No{\gah_0}_{\CX^\si_\hb} + (L(t))^2 \int_0^t \frac{(L(\ta))^2 }{\bta^{1-\ep}} d\ta \No{\gah_0}_{\CX^\si_\hb}
		\ls \bt^{2\ep} \No{\gah_0}_{\CX^\si_\hb}.
	\end{align}
	Hence, we obtained the first estimate. We can prove the second one in the same way.
\end{proof}

\subsubsection{Miscellaneous technical inequalities}
\begin{lemma}\label{lem:commutator 1.6.2}
	Under Assumption \ref{ass:2}, for all sufficiently small $\ep=\ep(\de,a,b)>0$, there exists a sufficiently small $R>0$ such that $\|\ze\|_{\CY^{a,b}_T} \le R$ implies
	\begin{align}\label{eq:hojo 3}
		&\No{\Jt^{\si/2} \Bk{f(t), \gah(t)} \Jt^{\si/2}}_{\FS^2_\hb} 
		\ls \bt^{1+\ep} \hb\No{\na f(t)}_{L^\I_x} \No{\gah_0}_{\CX^\si_\hb} 
	\end{align}
	for all $f \in W^{1,\I}(\R^3)$.
\end{lemma}

\begin{proof}
	Note that
	\begin{align}
		&\No{\Jt^{\si/2} \Bk{f(t), \gah(t)} \Jt^{\si/2}}_{\FS^2_\hb} \\
		&\quad \ls \No{\Bk{\Jt^{\si/2},f(t)}\Jt^{-\si/2}}_{\CB} \No{\Jt^{\si/2} \gah(t)\Jt^{\si/2}}_{\FS^2_\hb} \\
		&\qquad \qquad + \No{ \Bk{f(t), \Jt^{\si/2}\gah(t)\Jt^{\si/2}} }_{\FS^2_\hb}=: \SA + \SB.
	\end{align}
	For $\SA$, by Lemma \ref{lem:commutator basic}, we have
	\begin{align}
		\SA\ls \bt^{1+\ep} \hb\No{\na f(t)}_{L^\I_x} \No{\gah_0}_{\CX^\si_\hb}.
	\end{align}
	For $\SB$, again by Lemma \ref{lem:commutator basic}, we have
	\begin{align}
		\SB &\le \No{\na f(t)}_{L^\I_x} \No{\Bk{x,\Jt^{\si/2}\gah(t)\Jt^{\si/2}}}_{\FS^2_\hb} \\
		&\le \No{\na f(t)}_{L^\I_x} \No{\Bk{\J(-t),\bx^{\si/2}\W_V(t)\gah_0 \W_V(t)^* \bx^{\si/2}}}_{\FS^2_\hb} \\
		&\le \hb \No{\na f(t)}_{L^\I_x} \No{\Bk{\xh,\bx^{\si/2}\W_V(t)\gah_0 \W_V(t)^* \bx^{\si/2}}}_{\FS^2_\hb} \\
		&\qquad \qquad + t\hb \No{\na f(t)}_{L^\I_x} \No{\Bk{\na,\bx^{\si/2}\W_V(t)\gah_0 \W_V(t)^* \bx^{\si/2}}}_{\FS^2_\hb}=:\SB_1 + \SB_2.
	\end{align}
	For $\SB_1$, by Lemma \ref{lem:commutator 5}, we obtain
	\begin{align}
		\SB_1 \ls \hb\bt^{\ep} \No{\na f(t)}_{L^\I_x} \No{\gah_0}_{\CX^\si_\hb}.
	\end{align}
	To estimate $\SB_2$, note that
	\begin{align}
		&\No{\Bk{\na,\bx^{\si/2}\W_V(t)\gah_0 \W_V(t)^* \bx^{\si/2}}}_{\FS^2_\hb} \\
		&\quad \ls \No{\bx^{\si/2}\Bk{\na,\W_V(t)\gah_0 \W_V(t)^* }\bx^{\si/2}}_{\FS^2_\hb} + \bt^\ep \No{\gah_0}_{\CX^\si_\hb}.
	\end{align}
	Applying Lemma \ref{lem:commutator 5} to the above estimates, we get the desired result.
\end{proof}

\begin{lemma}\label{lem:commutator 5.9}
		Under Assumption \ref{ass:2}, for all sufficiently small $\ep=\ep(\de,a,b)>0$, there exists a sufficiently small $R>0$ such that $\|\ze\|_{\CY^{a,b}_T} \le R$ implies
	\begin{align}
		&\No{\Jt^{{\si/2}}\Bk{f(t) , \Bk{ \J(t), \gah(t)}}\Jt^{{\si/2}}}_\Sp \ls \hb^2 \bt^{1+\ep} \No{\na f(t)}_{L^\I} \No{\gah_0}_{\CX^\si_\hb}\\
	    &\qquad \qquad + \hb^2 \|\na f(t)\|_{L^\I_x}  \No{\bx^{{\si/2}} \Dk{\xh, \Bk{ \xh, \W_V(t)\gah_0\W_V(t)^*} }\bx^{{\si/2}} }_{\FS^2_\hb} \\
	    &\qquad \qquad + \hb^2\bt \|\na f(t)\|_{L^\I_x} \No{\bx^{{\si/2}} \Dk{\na, \Bk{ \xh, \W_V(t)\gah_0\W_V(t)^*} }\bx^{{\si/2}} }_{\FS^2_\hb}
	\end{align}
	for all $f \in W^{1,\I}(\R^3)$.
\end{lemma}

\begin{proof}
	First, note that
	\begin{align}
		&\No{\Jt^{{\si/2}}\Bk{f(t), \Bk{ \J(t), \gah(t)}}\Jt^{{\si/2}}}_\Sp\\
		&\quad \ls \No{\Bk{\Jt^{{\si/2}}, f(t)} \Jt^{-\si/2}}_\CB
		\No{\Jt^{{\si/2}} \Bk{ \J(t), \gah(t)} \Jt^{{\si/2}}}_{\FS^2_\hb}\\
		&\qquad + \No{\Dk{f(t), \Jt^{{\si/2}} \Bk{ \J(t), \gah(t)} \Jt^{{\si/2}}} }_{\FS^2_\hb}=: \SA + \SB.
	\end{align}
For $\SA$, by Lemma \ref{lem:commutator 5}, we have
\begin{align}
	\SA \ls t\hb^2 \No{\na f(t)}_{L^\I} \No{\bx^{{\si/2}} \Bk{ \xh, \W_V(t)\gah_0 \W_V(t)^*} \bx^{{\si/2}}}_{\FS^2_\hb}
	\ls \hb^2 \bt^{1+\ep}  \No{\na f(t)}_{L^\I} \No{\gah_0}_{\CX^\si_\hb}.
\end{align}
For $\SB$, by Lemma \ref{lem:commutator basic}, we have
\begin{align}
	\SB&\le \|\na f(t)\|_{L^\I}  \No{\Dk{x, \Jt^{{\si/2}} \Bk{ \J(t), \gah(t)} \Jt^{{\si/2}}} }_{\FS^2_\hb} \\
	&=  \|\na f(t)\|_{L^\I}  \No{\Dk{\J(-t), \bx^{{\si/2}} \Bk{ x, \W_V(t)\gah_0\W_V(t)^*} \bx^{{\si/2}}} }_{\FS^2_\hb} \\
	&\le \hb^2 \|\na f(t)\|_{L^\I}  \No{\Dk{\xh, \bx^{{\si/2}} \Bk{ \xh, \W_V(t)\gah_0\W_V(t)^*} \bx^{{\si/2}}} }_{\FS^2_\hb} \\
	&\quad + t\hb^2 \|\na f(t)\|_{L^\I}  \No{\Dk{\na, \bx^{{\si/2}} \Bk{ \xh, \W_V(t)\gah_0\W_V(t)^*} \bx^{{\si/2}}} }_{\FS^2_\hb}.
\end{align}
Finally, note that Lemma \ref{lem:commutator 5} implies
\begin{align}
	&\No{\Dk{\na, \bx^{{\si/2}} \Bk{ \xh, \W_V(t)\gah_0\W_V(t)^*} \bx^{{\si/2}}} }_{\FS^2_\hb} \\
	&\quad \le \No{\bx^{{\si/2}} \Dk{\na, \Bk{ \xh, \W_V(t)\gah_0\W_V(t)^*} }\bx^{{\si/2}} }_{\FS^2_\hb}  + \No{\bx^{{\si/2}} \Bk{ \xh, \W_V(t)\gah_0\W_V(t)^*} \bx^{{\si/2}}}_{\FS^2_\hb} \\
	&\quad \ls \No{\bx^{{\si/2}} \Dk{\na, \Bk{ \xh, \W_V(t)\gah_0\W_V(t)^*} }\bx^{{\si/2}} }_{\FS^2_\hb} 
	+ \bt^\ep \No{\gah_0}_{\CX^\si_\hb}.
\end{align}
Collecting the above estimates, we obtain the desired result.
\end{proof}

\begin{lemma}\label{lem:commutator 5.9.1}
		Under Assumption \ref{ass:2}, for all sufficiently small $\ep=\ep(\de,a,b)>0$, there exists a sufficiently small $R>0$ such that $\|\ze\|_{\CY^{a,b}_T} \le R$ implies
	\begin{align}
		&\No{\Jt^{{\si/2}}\Bk{f(t) , \Bk{ \na, \gah(t)}}\Jt^{{\si/2}}}_\Sp \ls \hb \bt^{1+\ep} \No{\na f(t)}_{L^\I} \No{\gah_0}_{\CX^\si_\hb}\\
		&\qquad \qquad + \hb \|\na f(t)\|_{L^\I_x}  \No{\bx^{{\si/2}} \Dk{\na, \Bk{ \xh, \W_V(t)\gah_0\W_V(t)^*} }\bx^{{\si/2}} }_{\FS^2_\hb} \\
		&\qquad \qquad + \hb\bt \|\na f(t)\|_{L^\I_x} \No{\bx^{{\si/2}} \Dk{\na, \Bk{ \na, \W_V(t)\gah_0\W_V(t)^*} }\bx^{{\si/2}} }_{\FS^2_\hb}
	\end{align}
	for all $f \in W^{1,\I}(\R^3)$.
\end{lemma}

\begin{proof}
	First, note that
	\begin{align}
		&\No{\Jt^{{\si/2}}\Bk{f(t), \Bk{ \na, \gah(t)}}\Jt^{{\si/2}}}_\Sp\\
		&\quad \ls \No{\Bk{\Jt^{{\si/2}}, f(t)}\Jt^{-\si/2}}_\CB
		\No{\Jt^{{\si/2}} \Bk{ \na, \gah(t)} \Jt^{{\si/2}}}_{\FS^2_\hb}\\
		&\qquad + \No{\Dk{f(t), \Jt^{{\si/2}} \Bk{ \na, \gah(t)} \Jt^{{\si/2}}} }_{\FS^2_\hb}=: \SA + \SB.
	\end{align}
	For $\SA$, by Lemma \ref{lem:commutator 5}, we have
	\begin{align}
		\SA \ls t\hb \No{\na f(t)}_{L^\I} \No{\bx^{{\si/2}} \Bk{ \na, \W_V(t)\gah_0 \W_V(t)^*} \bx^{{\si/2}}}_{\FS^2_\hb}
		\ls \hb \bt^{1+\ep}  \No{\na f(t)}_{L^\I} \No{\gah_0}_{\CX^\si_\hb}.
	\end{align}
	For $\SB$, by Lemma \ref{lem:commutator basic}, we have
	\begin{align}
		\SB&\le \|\na f(t)\|_{L^\I}  \No{\Dk{x, \Jt^{{\si/2}} \Bk{ \na, \gah(t)} \Jt^{{\si/2}}} }_{\FS^2_\hb} \\
		&=  \|\na f(t)\|_{L^\I}  \No{\Dk{\J(-t), \bx^{{\si/2}} \Bk{ \na, \W_V(t)\gah_0\W_V(t)^*} \bx^{{\si/2}}} }_{\FS^2_\hb} \\
		&\le \hb \|\na f(t)\|_{L^\I}  \No{\Dk{\xh, \bx^{{\si/2}} \Bk{ \na, \W_V(t)\gah_0\W_V(t)^*} \bx^{{\si/2}}} }_{\FS^2_\hb} \\
		&\quad + t\hb \|\na f(t)\|_{L^\I}  \No{\Dk{\na, \bx^{{\si/2}} \Bk{ \na, \W_V(t)\gah_0\W_V(t)^*} \bx^{{\si/2}}} }_{\FS^2_\hb}.
	\end{align}
	By the Jacobi identity \eqref{eq:Jacobi}, we have
	\begin{equation}\label{eq:commute}
		\begin{aligned}
		&\No{\bx^{{\si/2}} \Dk{\xh, \Bk{ \na, \W_V(t)\gah_0\W_V(t)^*} } \bx^{{\si/2}}}_{\FS^2_\hb} \\
		&\qquad = \No{\bx^{{\si/2}} \Dk{ \na, \Bk{\xh, \W_V(t)\gah_0\W_V(t)^*} } \bx^{{\si/2}}}_{\FS^2_\hb}.
		\end{aligned}
	\end{equation}
	Finally, note that Lemma \ref{lem:commutator 5} implies
	\begin{align}
		&\No{\Dk{\na, \bx^{{\si/2}} \Bk{ \na, \W_V(t)\gah_0\W_V(t)^*} \bx^{{\si/2}}} }_{\FS^2_\hb} \\
		&\quad \le \No{\bx^{{\si/2}} \Dk{\na, \Bk{ \na, \W_V(t)\gah_0\W_V(t)^*} }\bx^{{\si/2}} }_{\FS^2_\hb} 
		 + \No{\bx^{{\si/2}} \Bk{ \na, \W_V(t)\gah_0\W_V(t)^*} \bx^{{\si/2}}}_{\FS^2_\hb} \\
		&\quad \ls \No{\bx^{{\si/2}} \Dk{\na, \Bk{ \na, \W_V(t)\gah_0\W_V(t)^*} }\bx^{{\si/2}} }_{\FS^2_\hb} 
		+ \bt^\ep \No{\gah_0}_{\CX^\si_\hb}.
	\end{align}
	Collecting the above estimates, we obtain the desired result.
\end{proof}

\subsection{Proof of Lemma \ref{lem:commutator 6}}
Now we give a complete proof of Lemma \ref{lem:commutator 6}.
Let $L(t):= \log(t+2)$ to avoid wasting space.

\subsubsection{Estimate of the first term}\label{subsubsec:first}
By Lemmas \ref{lem:commutator x W} and \ref{lem:easy}, we have
\begin{align}
	&\No{\bx^{{\si/2}} \Dk{\frac{x_k}{\hb},\Dk{\frac{x_j}{\hb},\W_V(t)\gah_0 \W_V(t)^*}}\bx^{{\si/2}}}_{\Sp} \\ 
	&\quad \ls (L(t))^2 \No{\bx^{{\si/2}} \Dk{\frac{x_k}{\hb}, \Dk{\frac{x_j}{\hb}, \gah_0 }} \bx^{{\si/2}}}_\Sp \\
	&\qquad + \frac{(L(t))^2 }{\hb^2} \int_0^t (L(\ta))^2
	\No{\Jta^{{\si/2}}\Bk{\ta \pl_{x_j} V(\ta) , \Bk{ \J_k(\ta), \gah(\ta)}}\Jta^{{\si/2}}}_\Sp d\ta \\
	&\qquad + \frac{(L(t))^2 }{\hb^2} \int_0^t (L(\ta))^2
	\No{\Jta^{{\si/2}}\Bk{\ta \pl_{x_k} V(\ta), \Bk{\J_j(\ta), \gah(\ta)}}\Jta^{{\si/2}}}_\Sp d\ta\\
	&\qquad + \frac{(L(t))^2 }{\hb} \int_0^t (L(\ta))^2
	\No{\Jta^{{\si/2}} \Bk{\ta^2 \pl_{x_k,x_j}^2 V(\ta), \gah(\ta)}\Jta^{{\si/2}}}_{\FS^2_\hb} d\ta\\
	&\quad =: (L(t))^2 \No{\bx^{{\si/2}} \Dk{\frac{x_k}{\hb}, \Dk{\frac{x_j}{\hb}, \gah_0 }} \bx^{{\si/2}}}_\Sp + \SA + \SB + \SC.
\end{align}

\smallskip \noindent  {\textbf{$\blacklozenge$ Estimate of $\SA$ and $\SB$.}}
By Lemmas \ref{lem:commutator 5.9} and \ref{lem:potential deriv}, we have
\begin{align}
	&\No{\Jta^{{\si/2}}\Bk{\ta \pl_{x_j} V(\ta), \Bk{ \J_k(\ta), \gah(\ta)}}\Jta^{{\si/2}}}_\Sp\\
	&\quad \ls \hb^2 \bta^{1+\ep} \No{\ta \na^2 V(\ta)}_{L^\I}\No{\gah_0}_{\CX^\si_\hb}\\
	&\qquad \qquad + \hb^2 \|\ta \na^2 V(\ta)\|_{L^\I_x}  \No{\bx^{{\si/2}} \Dk{\xh, \Bk{ \xh, \W_V(\ta)\gah_0\W_V(\ta)^*} }\bx^{{\si/2}} }_{\FS^2_\hb} \\
	&\qquad \qquad + \hb^2\bta \|\ta \na^2 V(\ta)\|_{L^\I_x} \No{\bx^{{\si/2}} \Dk{\na, \Bk{ \xh, \W_V(\ta)\gah_0\W_V(\ta)^*} }\bx^{{\si/2}} }_{\FS^2_\hb} \\
	&\quad \ls \frac{\hb^2 R \No{\gah_0}_{\CX^\si_\hb}}{\bta^{1-2\ep}} 
	 + \frac{R \hb^2}{\bta^{2-\ep}} \No{\bx^{{\si/2}} \Dk{\xh, \Bk{ \xh, \W_V(\ta)\gah_0\W_V(\ta)^*} }\bx^{{\si/2}} }_{\FS^2_\hb} \\
	&\qquad \qquad + \frac{R\hb^2}{\bta^{1-\ep}} \No{\bx^{{\si/2}} \Dk{\na, \Bk{ \xh, \W_V(\ta)\gah_0\W_V(\ta)^*} }\bx^{{\si/2}} }_{\FS^2_\hb}.
\end{align}
Therefore, we have
\begin{align}
	\SA &\ls \frac{(L(t))^2}{\hb^2} \int_0^t d\ta (L(\ta))^2 \bigg\{\frac{R\hb^2\No{\gah_0}_{\CX^\si_\hb}}{\bta^{1-2\ep}} \\
	&\qquad \qquad + \frac{R \hb^2}{\bta^{2-\ep}} \No{\bx^{{\si/2}} \Dk{\xh, \Bk{ \xh, \W_V(\ta)\gah_0\W_V(\ta)^*} }\bx^{{\si/2}} }_{\FS^2_\hb} \\
	&\qquad \qquad + \frac{R\hb^2}{\bta^{1-\ep}} \No{\bx^{{\si/2}} \Dk{\na, \Bk{ \xh, \W_V(\ta)\gah_0\W_V(\ta)^*} }\bx^{{\si/2}} }_{\FS^2_\hb}\bigg\}  \\
	&\ls \bt^{4\ep}\No{\gah_0}_{\CX^\si_\hb} 
	    + R\bt^\ep \sup_{0\le\ta\le t} \bta^{-a-3\ep} \No{\bx^{{\si/2}} \Dk{\xh, \Bk{ \xh, \W_V(\ta)\gah_0\W_V(\ta)^*} }\bx^{{\si/2}} }_{\FS^2_\hb} \\
	&\qquad \qquad + R\bt^{100\ep}\sup_{0\le\ta\le t} \bta^{-10\ep} \No{\bx^{{\si/2}} \Dk{\na, \Bk{ \xh, \W_V(\ta)\gah_0\W_V(\ta)^*} }\bx^{{\si/2}} }_{\FS^2_\hb}.
\end{align}
By symmetry, it is clear that $\SB$ satisfies the same estimate.

\smallskip \noindent  {\bf $\blacklozenge$ Estimate of $\SC$.}
By Lemmas \ref{lem:commutator 1.6.2} and \ref{lem:potential deriv} we have
\begin{align}
	\SC &\ls \frac{(L(t))^2}{\hb} \int_0^t 
	(L(\ta))^2\bta^{1+\ep} \hb \bta^2 \No{\na^3 V(\ta)}_{L^\I_x} \No{\gah_0}_{\CX^\si_\hb} d\ta \\
	&\ls R(L(t))^2 \No{\gah_0}_{\CX^\si_\hb}\int_0^t \frac{(L(\ta))^2}{\bta^{1-a-2\ep}}d\ta \ls R\bt^{a+3\ep}\No{\gah_0}_{\CX^\si_\hb}.
\end{align}

\smallskip \noindent  {\bf $\blacklozenge$ Conclusion.}
Collecting the above argument, we obtain
\begin{align}
	&\sup_{0\le\ta\le t}\bta^{-a-3\ep}\No{\bx^{{\si/2}} \Dk{\frac{x_k}{\hb},\Dk{\frac{x_j}{\hb},\W_V(\ta)\gah_0 \W_V(\ta)^*}}\bx^{{\si/2}}}_{\Sp} \\
	&\qquad \ls \No{\gah_0}_{\CX^\si_\hb} 
	+ R \sup_{0\le\ta\le t} \bta^{-a-3\ep} \No{\bx^{{\si/2}} \Dk{\xh, \Bk{ \xh, \W_V(\ta)\gah_0\W_V(\ta)^*} }\bx^{{\si/2}} }_{\FS^2_\hb} \\
	&\qquad \qquad \qquad \qquad + R\sup_{0\le\ta\le t} \bta^{-10\ep} \No{\bx^{{\si/2}} \Dk{\na, \Bk{ \xh, \W_V(\ta)\gah_0\W_V(\ta)^*} }\bx^{{\si/2}} }_{\FS^2_\hb}.
\end{align}

\subsubsection{Estimate of the second term}\label{subsub:second}
By Lemma \ref{lem:commutator mixed}, we have
\begin{align}
	&\No{\bx^{{\si/2}} \Dk{\pl_{x_k},\Dk{\frac{x_j}{\hb},\W_V(t)\gah_0 \W_V(t)^*}}\bx^{{\si/2}}}_{\Sp}  \\
	&\quad \ls (L(t))^2 \No{\bx^{\si/2} \Dk{\na, \Dk{\frac{x}{\hb}, \gah_0 }} \bx^{\si/2}}_\Sp \\
	&\qquad + \frac{(L(t))^2}{\hb} \int_0^t (L(\ta))^2\No{\Jta^{\si/2} \Bk{\ta \na V(\ta), \Bk{\na, \gah(\ta)}} \Jta^{\si/2}}_\Sp d\ta \\
			&\qquad + \frac{(L(t))^2}{\hb^2} \int_0^t (L(\ta))^2\No{\Jta^{\si/2} \Bk{\na V(\ta), \Bk{\J(\ta), \gah(\ta)}} \Jta^{\si/2}}_\Sp d\ta\\
			&\qquad + \frac{(L(t))^2}{\hb} \int_0^t (L(\ta))^2\No{\Jta^{\si/2} \Bk{\ta \na^2 V(\ta), \gah(\ta)} \Jta^{\si/2}}_\Sp d\ta \\
	&\quad =: (L(t))^2 \No{\bx^{\si/2} \Dk{\na, \Dk{\frac{x}{\hb}, \gah_0 }} \bx^{\si/2}}_\Sp + \SD+ \SE + \SF.
\end{align}

\smallskip \noindent  {\bf $\blacklozenge$ Estimate of $\SD$.}
By Lemmas \ref{lem:commutator 5.9.1} and \ref{lem:potential deriv}, we have
\begin{align}
	&\No{\Jta^{\si/2} \Bk{\ta \na V(\ta), \Bk{\na, \gah(\ta)}} \Jta^{\si/2}}_\Sp\\
	&\quad  \ls \hb \bta^{1+\ep} \No{\ta \na^2 V(\ta)}_{L^\I} \No{\gah_0}_{\CX^\si_\hb}\\
	&\qquad \qquad + \hb \|\ta \na^2 V(\ta)\|_{L^\I_x}  \No{\bx^{{\si/2}} \Dk{\na, \Bk{ \xh, \W_V(\ta)\gah_0\W_V(\ta)^*} }\bx^{{\si/2}} }_{\FS^2_\hb} \\
	&\qquad \qquad + \hb\bta \|\ta \na^2 V(\ta)\|_{L^\I_x} \No{\bx^{{\si/2}} \Dk{\na, \Bk{ \na, \W_V(\ta)\gah_0\W_V(\ta)^*} }\bx^{{\si/2}} }_{\FS^2_\hb} \\
	&\quad  \ls \frac{R \hb \No{\gah_0}_{\CX^\si_\hb} }{\bta^{1-2\ep}} 
	+ \frac{R\hb}{\bta^{2-\ep}}   \No{\bx^{{\si/2}} \Dk{\na, \Bk{ \xh, \W_V(\ta)\gah_0\W_V(\ta)^*} }\bx^{{\si/2}} }_{\FS^2_\hb} \\
	&\qquad \qquad + \frac{R\hb}{\bta^{1-\ep}}  \No{\bx^{{\si/2}} \Dk{\na, \Bk{ \na, \W_V(\ta)\gah_0\W_V(\ta)^*} }\bx^{{\si/2}} }_{\FS^2_\hb}.
\end{align}
Therefore, we have
\begin{align}
	\SD &\ls \frac{(L(t))^2}{\hb} \int_0^t d\ta (L(\ta))^2 \bigg\{\frac{R \hb \No{\gah_0}_{\CX^\si_\hb} }{\bta^{1-2\ep}} \\
	&\qquad \qquad + \frac{R\hb}{\bta^{2-\ep}} \No{\bx^{{\si/2}} \Dk{\xh, \Bk{ \na, \W_V(\ta)\gah_0\W_V(\ta)^*} }\bx^{\si/2} }_{\FS^2_\hb} \\
	&\qquad \qquad + \frac{R\hb}{\bta^{1-\ep}}  \No{\bx^{{\si/2}} \Dk{\na, \Bk{ \na, \W_V(\ta)\gah_0\W_V(\ta)^*} }\bx^{\si/2} }_{\FS^2_\hb} \bigg\}\\
	&\ls R\bt^{3\ep}\No{\gah_0}_{\CX^\si_\hb} + R\bt^{3\ep} \sup_{0\le\ta\le t} \bta^{-3\ep} \No{\bx^{{\si/2}} \Dk{\na, \Bk{ \xh, \W_V(\ta)\gah_0\W_V(\ta)^*} }\bx^{{\si/2}} }_{\FS^2_\hb} \\
	&\qquad \qquad + R\bt^{3\ep} \sup_{0\le\ta\le t} \bta^{-\ep} \No{\bx^{{\si/2}} \Dk{\na, \Bk{ \na, \W_V(\ta)\gah_0\W_V(\ta)^*} }\bx^{{\si/2}} }_{\FS^2_\hb},
\end{align}
where we used \eqref{eq:commute} in the last inequality.

\smallskip \noindent  {\bf $\blacklozenge$ Estimate of $\SE$.}
By Lemmas \ref{lem:commutator 5.9}, \ref{lem:potential deriv} and \ref{lem:commutator 5}, we have
\begin{align}
	&\No{\Jta^{\si/2} \Dk{\na V(\ta), \Dk{\J(\ta), \gah(\ta)}} \Jta^{\si/2}}_\Sp\\
	&\quad  \ls \hb^2 \bta^{1+\ep} \No{\na^2 V(\ta)}_{L^\I} \No{\gah_0}_{\CX^\si_\hb}\\
	&\qquad \qquad + \hb^2 \|\na^2 V(\ta)\|_{L^\I_x}  \No{\bx^{{\si/2}} \Dk{\xh, \Bk{ \xh, \W_V(\ta)\gah_0\W_V(\ta)^*} }\bx^{{\si/2}} }_{\FS^2_\hb} \\
	&\qquad \qquad + \hb^2\bta \|\na^2 V(\ta)\|_{L^\I_x} \No{\bx^{{\si/2}} \Dk{\na, \Bk{ \xh, \W_V(\ta)\gah_0\W_V(\ta)^*} }\bx^{{\si/2}} }_{\FS^2_\hb} \\
	&\quad  \ls \frac{R \hb^2 \No{\gah_0}_{\CX^\si_\hb} }{\bta^{2-2\ep}} 
	 + \frac{R\hb^2}{\bta^{3-\ep}}   \No{\bx^{{\si/2}} \Dk{\xh, \Bk{ \xh, \W_V(\ta)\gah_0\W_V(\ta)^*} }\bx^{{\si/2}} }_{\FS^2_\hb} \\
	&\qquad \qquad + \frac{R\hb^2}{\bta^{2-\ep}}  \No{\bx^{{\si/2}} \Dk{\na, \Bk{ \xh, \W_V(\ta)\gah_0\W_V(\ta)^*} }\bx^{{\si/2}} }_{\FS^2_\hb}.
\end{align}
Therefore, we have
\begin{align}
	\SE &\ls R\bt^{\ep}\No{\gah_0}_{\CX^\si_\hb} 
	 + R\bt^{\ep} \sup_{0\le\ta\le t} \bta^{-a-3\ep} \No{\bx^{{\si/2}} \Dk{\xh, \Bk{ \xh, \W_V(\ta)\gah_0\W_V(\ta)^*} }\bx^{{\si/2}} }_{\FS^2_\hb} \\
	&\qquad \qquad + R\bt^{\ep} \sup_{0\le\ta\le t} \bta^{-3\ep} \No{\bx^{{\si/2}} \Dk{\na, \Bk{ \xh, \W_V(\ta)\gah_0\W_V(\ta)^*} }\bx^{{\si/2}} }_{\FS^2_\hb}.
\end{align}

\smallskip \noindent  {\bf$\blacklozenge$ Estimate of $\SF$.}
By Lemmas \ref{lem:commutator 1.6.2} and \ref{lem:potential deriv} we have
\begin{align}
	\SF &\ls \frac{(L(t))^2}{\hb} \int_0^t 
	(L(\ta))^2\bta^{1+\ep} \hb \bta \No{\na^3 V(\ta)}_{L^\I_x} \No{\gah_0}_{\CX^\si_\hb} d\ta \\
	&\ls R(L(t))^2 \No{\gah_0}_{\CX^\si_\hb}\int_0^t \frac{(L(\ta))^2}{\bta^{2-a-2\ep}}d\ta \ls R\bt^{\ep}\No{\gah_0}_{\CX^\si_\hb}.
\end{align}

\smallskip \noindent  \textbf{$\blacklozenge$ Conclusion.}
Collecting all the above, we obtain
\begin{align}
	&\No{\bx^{{\si/2}} \Dk{\na,\Dk{\frac{x}{\hb},\W_V(t)\gah_0 \W_V(t)^*}}\bx^{{\si/2}}}_{\Sp} \\
	&\quad \ls \bt^{3\ep}\No{\gah_0}_{\CX^\si_\hb}
            +R\bt^{3\ep} \sup_{0\le\ta\le t} \bta^{-a-3\ep} \No{\bx^{{\si/2}} \Dk{\xh, \Bk{ \xh, \W_V(\ta)\gah_0\W_V(\ta)^*} }\bx^{{\si/2}} }_{\FS^2_\hb} \\	
	&\qquad \qquad + R\bt^{3\ep} \sup_{0\le\ta\le t} \bta^{-3\ep} \No{\bx^{{\si/2}} \Dk{\na, \Bk{ \xh, \W_V(\ta)\gah_0\W_V(\ta)^*} }\bx^{{\si/2}} }_{\FS^2_\hb} \\
	&\qquad \qquad + R\bt^{3\ep} \sup_{0\le\ta\le t} \bta^{-\ep} \No{\bx^{{\si/2}} \Dk{\na, \Bk{ \na, \W_V(\ta)\gah_0\W_V(\ta)^*} }\bx^{{\si/2}} }_{\FS^2_\hb}.
\end{align}

\subsubsection{Estimate of the third term.}\label{subsub:third}
By Lemma \ref{lem:commutator nabla W}, we have
\begin{align}
	&\No{\bx^{{\si/2}} \Bk{\pl_{x_k},\Bk{\pl_{x_j},\W_V(t)\gah_0 \W_V(t)^*}}\bx^{{\si/2}}}_{\Sp}  \\
	&\quad \ls (L(t))^2 \No{\bx^{\si/2} \Dk{\na, \Dk{\na, \gah_0 }} \bx^{\si/2}}_\Sp \\
	&\qquad + \frac{(L(t))^2}{\hb} \int_0^t (L(\ta))^2\No{\Jta^{\si/2} \Bk{\na V(\ta), \Dk{\na, \gah(\ta)}} \Jta^{\si/2}}_\Sp d\ta \\
	&\qquad + \frac{(L(t))^2}{\hb} \int_0^t (L(\ta))^2\No{\Jta^{\si/2} \Bk{\na V(\ta), \Dk{\na, \gah(\ta)}} \Jta^{\si/2}}_\Sp d\ta\\
	&\qquad + \frac{(L(t))^2}{\hb} \int_0^t (L(\ta))^2\No{\Jta^{\si/2} \Bk{\na^2 V(\ta), \gah(\ta)} \Jta^{\si/2}}_\Sp d\ta \\
	&\quad =: (L(t))^2 \No{\bx^{\si/2} \Bk{\na, \Bk{\na, \gah_0 }} \bx^{\si/2}}_\Sp + \SG+ \SH + \SI.
\end{align}

\smallskip \noindent  {\bf $\blacklozenge$ Estimate of $\SG$ and $\SH$.}
By the symmetry, it suffices to estimate $\SG$.
By Lemmas \ref{lem:commutator 5.9.1} and \ref{lem:potential deriv}, we have
\begin{align}
	&\No{\Jta^{\si/2} \Bk{\na V(\ta), \Bk{\na, \gah(\ta)}} \Jta^{\si/2}}_\Sp\\
	&\quad  \ls \hb \bta^{1+\ep} \No{\na^2 V(\ta)}_{L^\I} \No{\gah_0}_{\CX^\si_\hb}\\
	&\qquad \qquad + \hb \|\na^2 V(\ta)\|_{L^\I_x}  \No{\bx^{{\si/2}} \Bk{\na, \Bk{ \xh, \W_V(\ta)\gah_0\W_V(\ta)^*} }\bx^{{\si/2}} }_{\FS^2_\hb} \\
	&\qquad \qquad + \hb\bta \|\na^2 V(\ta)\|_{L^\I_x} \No{\bx^{{\si/2}} \Bk{\na, \Bk{ \na, \W_V(\ta)\gah_0\W_V(\ta)^*} }\bx^{{\si/2}} }_{\FS^2_\hb} \\
	&\quad  \ls \frac{R \hb \No{\gah_0}_{\CX^\si_\hb} }{\bta^{2-2\ep}} 
	+ \frac{R\hb}{\bta^{3-\ep}}   \No{\bx^{{\si/2}} \Dk{\na, \Bk{ \xh, \W_V(\ta)\gah_0\W_V(\ta)^*} }\bx^{{\si/2}} }_{\FS^2_\hb} \\
	&\qquad \qquad + \frac{R\hb}{\bta^{2-\ep}}  \No{\bx^{{\si/2}} \Dk{\na, \Bk{ \na, \W_V(\ta)\gah_0\W_V(\ta)^*} }\bx^{{\si/2}} }_{\FS^2_\hb}.
\end{align}
Therefore, we have
\begin{align}
	\SG &\ls R\bt^{\ep}\No{\gah_0}_{\CX^\si_\hb} + R\bt^{\ep} \sup_{0\le\ta\le t} \bta^{-3\ep} \No{\bx^{{\si/2}} \Dk{\na, \Bk{ \xh, \W_V(\ta)\gah_0\W_V(\ta)^*} }\bx^{{\si/2}} }_{\FS^2_\hb} \\
	&\qquad \qquad + R\bt^{\ep} \sup_{0\le\ta\le t} \bta^{-\ep} \No{\bx^{{\si/2}} \Dk{\na, \Bk{ \na, \W_V(\ta)\gah_0\W_V(\ta)^*} }\bx^{{\si/2}} }_{\FS^2_\hb}.
\end{align}

\smallskip \noindent  {\bf$\blacklozenge$ Estimate of $\SI$.}
By Lemmas \ref{lem:commutator 1.6.2} and \ref{lem:potential deriv} we have
\begin{align}
	\SI &\ls \frac{(L(t))^2}{\hb} \int_0^t 
	(L(\ta))^2\bta^{1+\ep} \hb \No{\na^3 V(\ta)}_{L^\I_x} \No{\gah_0}_{\CX^\si_\hb} d\ta \\
	&\ls R(L(t))^2 \No{\gah_0}_{\CX^\si_\hb}\int_0^t \frac{(L(\ta))^2}{\bta^{3-a-2\ep}}d\ta \ls R\bt^{\ep}\No{\gah_0}_{\CX^\si_\hb}.
\end{align}

\smallskip \noindent  \textbf{$\blacklozenge$ Conclusion.}
Collecting all the arguments above, we obtain
\begin{align}
	&\No{\bx^{{\si/2}} \Bk{\na,\Bk{\na,\W_V(t)\gah_0 \W_V(t)^*}}\bx^{{\si/2}}}_{\Sp} \\
	&\quad \ls \bt^{\ep}\No{\gah_0}_{\CX^\si_\hb}
	+R\bt^{\ep} \sup_{0\le\ta\le t} \bta^{-a-3\ep} \No{\bx^{{\si/2}} \Dk{\xh, \Bk{ \xh, \W_V(\ta)\gah_0\W_V(\ta)^*} }\bx^{{\si/2}} }_{\FS^2_\hb} \\	
	&\qquad \qquad + R\bt^{\ep} \sup_{0\le\ta\le t} \bta^{-3\ep} \No{\bx^{{\si/2}} \Dk{\na, \Bk{ \xh, \W_V(\ta)\gah_0\W_V(\ta)^*} }\bx^{{\si/2}} }_{\FS^2_\hb} \\
	&\qquad \qquad + R\bt^{\ep} \sup_{0\le\ta\le t} \bta^{-\ep} \No{\bx^{{\si/2}} \Dk{\na, \Bk{ \na, \W_V(\ta)\gah_0\W_V(\ta)^*} }\bx^{{\si/2}} }_{\FS^2_\hb}.
\end{align}

\subsubsection{Conclusion}
Collecting the estimates obtained in Sections \ref{subsubsec:first}, \ref{subsub:second}, and \ref{subsub:third}, we have
\begin{align}
	&\sup_{0\le\ta\le t}\bta^{-a-3\ep}\No{\bx^{{\si/2}} \Dk{\xh,\Dk{\frac{x}{\hb},\W_V(\ta)\gah_0 \W_V(\ta)^*}}\bx^{\si/2}}_{\Sp} \\
	&\qquad \qquad + \sup_{0\le\ta\le t}\bta^{-3\ep}\No{\bx^{\si/2} \Dk{\na,\Dk{\frac{x}{\hb},\W_V(\ta)\gah_0 \W_V(\ta)^*}}\bx^{\si/2}}_{\Sp} \\
	&\qquad \qquad + \sup_{0\le\ta\le t}\bta^{-\ep}\No{\bx^{\si/2} \Bk{\na,\Bk{\na,\W_V(\ta)\gah_0 \W_V(\ta)^*}}\bx^{{\si/2}}}_{\Sp} \\
	&\quad \ls \No{\gah_0}_{\CX^\si_\hb} + R \sup_{0\le\ta\le t}\bta^{-a-3\ep}\No{\bx^{{\si/2}} \Dk{\xh,\Dk{\frac{x}{\hb},\W_V(\ta)\gah_0 \W_V(\ta)^*}}\bx^{\si/2}}_{\Sp} \\
	&\qquad \qquad + R\sup_{0\le\ta\le t}\bta^{-3\ep}\No{\bx^{\si/2} \Dk{\na,\Dk{\frac{x}{\hb},\W_V(\ta)\gah_0 \W_V(\ta)^*}}\bx^{\si/2}}_{\Sp} \\
	&\qquad \qquad + R\sup_{0\le\ta\le t}\bta^{-\ep}\No{\bx^{\si/2} \Bk{\na,\Bk{\na,\W_V(\ta)\gah_0 \W_V(\ta)^*}}\bx^{{\si/2}}}_{\Sp}.
\end{align}
Therefore, choosing sufficiently small $R>0$, we have
\begin{align}
	&\sup_{0\le\ta\le t}\bta^{-a-3\ep}\No{\bx^{{\si/2}} \Dk{\xh,\Dk{\frac{x}{\hb},\W_V(\ta)\gah_0 \W_V(\ta)^*}}\bx^{\si/2}}_{\Sp} \\
	&\qquad \qquad + \sup_{0\le\ta\le t}\bta^{-3\ep}\No{\bx^{\si/2} \Dk{\na,\Dk{\frac{x}{\hb},\W_V(\ta)\gah_0 \W_V(\ta)^*}}\bx^{\si/2}}_{\Sp} \\
	&\qquad \qquad + \sup_{0\le\ta\le t}\bta^{-\ep}\No{\bx^{\si/2} \Bk{\na,\Bk{\na,\W_V(\ta)\gah_0 \W_V(\ta)^*}}\bx^{{\si/2}}}_{\Sp} \ls \No{\gah_0}_{\CX^\si_\hb}.
\end{align}
Since all the implicit constants are independent of $t$, we obtain the desired estimates.

\section{Proof of the main result}\label{sec:proof}
In this section, we give a proof of Theorem \ref{th:main} and Remark \ref{rmk:modified scattering}.
\subsection{Key a priori estimate}
Now, we give a key a priori estimate.
\begin{proposition}\label{prop:a priori}
	Let $d=3$ and $3/2<\si<2$.
	Let $a,b>0$ be small numbers such that $7b/8<a<b$.
	Assume that $w(x)$ satisfies \eqref{eq:ass}.
	Then, there exist $C\ge 1$ and a small $R>0$ such that the following holds: If $\ze$ is real-valued and satisfies $\|\ze\|_{\CY^{a,b}_T}\le R$, then we have
	\begin{align}\label{eq:a priori}
		\No{\rhh(\U_{w\ast\ze}(t)\gah_0\U_{w\ast\ze}(t)^*)}_{\CY^{a,b}_T} \le C\|\gah_0\|_{\CX^\si_\hb}
	\end{align}
	for any self-adjoint $\gah_0 \in \CX^\si_\hb$, where $C, R$ are independent of $\hb\in(0,1]$ and $T>0$.
\end{proposition}

\begin{remark}
	A solution $\gah(t)$ to \eqref{eq:NLH} satisfies 
	\[
	\varrho = \rhh\K{\U_{w\ast\varrho}(t)\gah_0 \U_{w\ast \varrho}(t)^*},
	\]
	where $\varrho(t):= \rhh(\gah(t))$.
	Hence, the a priori bound \eqref{eq:a priori} gives us a good estimate for the solution.
\end{remark}

\begin{proof}[Proof of Proposition \ref{prop:a priori}]
	Let $V:= w\ast \ze$.
	
\smallskip \noindent  \underline{\textbf{Step 1: Without derivative terms.}}
By \eqref{eq:density L1 bound}, we have
\begin{align}
	\No{\rhh(\U_{V}(t)\gah_0\U_{V}(t)^*)}_{L^1_x} \le \No{\gah_0}_{\FS^1_\hb} \le \No{\gah_0}_{\CX^\si_\hb}.
\end{align}
By Corollary \ref{cor:free} and Proposition \ref{prop:wave operator boundedness}, we have
\begin{align}
	\No{\rhh(\gah(t))}_{L^\I_x} 
	&\ls \frac{1}{\bt^3} \Big\{ \No{\sd^\si \W_V(t) \gah_0 \W_V(t)^* \sd^\si}_{\CB}  \\
	&\qquad \qquad + \No{\bx^\si e^{i\Ps(t,-i\hb \na_x)} \W_V(t)\gah_0 \W_V(t)^* e^{-i\Ps(t,-i\hb\na_x)} \bx^\si }_\CB \Big\}  \\
	&\ls \frac{1}{\bt^3} \Big\{ \No{\sd^\si \gah_0 \sd^\si}_{\CB}
	 + \No{\bx^\si \gah_0 \bx^\si }_\CB \Big\}
	\le \frac{1}{\bt^{3}} \|\gah_0\|_{\CX_\hb^\si}.
\end{align}
In the sequel, we fix a sufficiently small number $\ep\in(0,\de/100)$.

\smallskip \noindent  \underline{\textbf{Step 2: With one derivative (i).}}
By \eqref{eq:density deriv formula 0} and \eqref{eq:density deriv formula}, we have
\begin{align}
	\No{\na \rhh(\gah(t))}_{L^1_x}
	&\le \frac{1}{\bt} \bigg\{\No{\Dk{\na,\W_V(t)\gah_0\W_V(t)^*}}_{\FS^1_\hb} + \No{\Dk{\xh,\W_V(t)\gah_0\W_V(t)^*}}_{\FS^1_\hb} \bigg\}.
\end{align}
By Remark \ref{rmk:byproduct 1}, we have
\begin{align}
	&\No{\Dk{\xh,\W_V(t)\gah_0\W_V(t)^*}}_{\FS^1_\hb}
	+ \No{\Dk{\na,\W_V(t)\gah_0\W_V(t)^*}}_{\FS^1_\hb} \ls \bt^a \No{\gah_0}_{\CX^\si_\hb}.
\end{align}
Therefore, we obtained
\begin{align}
\sup_{0\le t \le T} \bt^{1-a} \No{\na \rhh(\gah(t))}_{L^1_x}
\ls \No{\gah_0}_{\CX^\si_\hb}.
\end{align}

\smallskip \noindent  \underline{\textbf{Step 3: With one derivative (ii).}}
By Corollary \ref{cor:free}, \eqref{eq:density deriv formula 0}, and \eqref{eq:density deriv formula}, we have
\begin{align}
	\No{\na \rhh(\gah(t))}_{L^{\I}_x}
	&\ls \frac{1}{\bt^4} \bigg\{ \No{\sd^{\si}\Bk{\na,\W_V(t)\gah_0\W_V(t)^*}\sd^{\si} }_{\CB} \\
	&\qquad \qquad + \No{\bx^{\si} e^{i\Ps(t,-i\hb\na_x)}\Dk{\xh,\W_V(t)\gah_0\W_V(t)^*} e^{-i\Ps(t,-i\hb\na_x)}\bx^{\si} }_{\CB} \bigg\}.
\end{align}
By Lemma \ref{lem:commutator 1.7}, we obtain
\begin{align}
	\sup_{0\le t \le T} \bt^{4-a}\No{\na \rhh(\gah(t))}_{L^{\I}_x}
	\ls \No{\gah_0}_{\CX^\si_\hb}.
\end{align}

\smallskip \noindent  \underline{\textbf{Step 4: With one derivative (iii).}}
By Lemmas \ref{lem:free Fourier} and \ref{lem:commutator 1.7}, we have
\begin{align}
\No{\na \rhh(\gah(t))}_{\CF L^1}
&\ls \frac{1}{\hb^{3/2}\bt^{4}} \Big\{\No{\sd^{\si} \Bk{\na,\W_V(t)\gah_0 \W_V(t)^*} \sd^{\si} }_{\FS^2_\hb} \\
&\qquad+ \No{\bx^{\si}\Bk{\xh, \W_V(t)\gah_0 \W_V(t)^*} \bx^{\si} }_{\FS^2_\hb} \Big\} \ls \frac{1}{\hb^{3/2}\bt^{4-a}}  \No{\gah_0}_{\CX^\si_\hb}.
\end{align}
Therefore, we obtain
\begin{align}
	\sup_{0\le t\le T} \hb^{3/2}\bt^{4-a} \No{\na \rhh(\gah(t))}_{\CF L^1} \ls \No{\gah_0}_{\CX_\hb^\si}.
\end{align}

\smallskip \noindent  \underline{\bf Step 5: With two derivatives.}
By \eqref{eq:density deriv formula 0}, \eqref{eq:density deriv formula}, and Lemmas \ref{lem:commutator 6}, \ref{lem:commutator 6.1}, we have
\begin{align}
	\No{|\na|^2 \rhh(\gah(t))}_{L^2}
	&\ls \frac{1}{\bt^{7/2}} \Big\{\No{\sd^{\si/2} \Bk{\na,\Bk{\na,\W_V(t)\gah_0 \W_V(t)^*} }\sd^{\si/2} }_{\FS^2_\hb} \\
	&\qquad \qquad \qquad + \No{\bx^{\si/2}\Bk{\xh, \Bk{\xh, \W_V(t)\gah_0 \W_V(t)^*}} \bx^{\si/2} }_{\FS^2_\hb} \Big\} \\
	&\ls \frac{1}{\bt^{7/2}} \K{\bt^\ep + \bt^{a+\ep}} \|\gah_0\|_{\CX^\si_\hb}.
\end{align}
Since $a<b$, we can choose $\ep=\ep(a,b)$ such that $a+\ep < b$. Therefore, we conclude that
\begin{align}
	\sup_{0\le t\le T} \bt^{7/2-b} \No{|\na|^2 \rhh(\gah(t))}_{L^2}
	&\ls \No{\gah_0}_{\CX^\si_\hb},
\end{align}
which completes the proof.
\end{proof}

\subsection{Local well-posedness}
In this section, we prove the local well-posedness of \eqref{eq:NLH} as a preparation for the bootstrap argument in the next section.
Define the space for initial data $\CX_0$ by
\begin{equation}
	\begin{aligned}
	\CX_0 &:= \Big\{ A \in \CB(L^2(\R^3)) \mid A \mbox{ is self-adjoint and } \|A\|_{\CX_0} < \I\Big\},
	\end{aligned}
\end{equation}
where the norm $\|\cdot\|_{\CX_0}$ is defined by
\begin{align}
	\No{A}_{\CX_0} &:= \|A\|_{\FS^1_\hb} + \No{\sd^{\si} A \sd^{\si}}_{\FS^2_\hb \cap \CB} + \No{\Bk{\na,A}}_{\FS^1_\hb} \\
	 &\quad + \No{\sd^{\si} \Bk{\na,A}\sd^{\si}}_{\FS^2_\hb \cap \CB} + \No{\sd^{\si/2}\Bk{\na,\Bk{\na,A}}\sd^{\si/2}}_{\FS^2_\hb}.
\end{align}
We define the space for density functions $\CY_{0,T}$ by
\[
\CY_{0,T} := \Ck{\ze: \ze \text{ is real-valued and } \|\ze\|_{\CY_{0,T}} < \I},
\]
where
\begin{equation}
	\No{\ze}_{\CY_{0,T}} := \|\ze\|_{C([0,T];L^1_x\cap L^\I_x)} + \|\na \ze\|_{C([0,T];L^1_x \cap L^\I_x)}
	+ \hb^{3/2}\No{\na \ze}_{C([0,T];\CF L^1)} +  \No{\na^2\ze}_{C([0,T];L^2_x)}.
\end{equation}

\begin{remark}\label{rmk:continuity}
	We have $\rhh(\U(t) \gah_0 \U(t)^*) \in \CY_{0,T}$ for any $T>0$, and this fact will be implicitly used in the proof of Lemma \ref{lem:LWP}.
	Note that the continuity is required, but we can prove it in the same way as \cite[Proof of Lemma 4.4]{Hadama Hong 2025}.
	In the definition of $\CX_0$ and $\CX_\hb^\si$ (see \eqref{eq:initial data class}), we imposed the compactness of $\sd^\si \gah_0 \sd^\si$, but this condition was added to ensure that
	$$\R\ni t\mapsto \rhh(\U(t)\gah_0\U(t)^*)  \in L^\I_x \mbox{ is continuous}.$$
\end{remark}

\begin{lemma}\label{lem:LWP}
	Let $d=3$ and $3/2<\si<2$.
	Assume that $w(x)$ satisfies \eqref{eq:ass}.
	Then, we have the following:
	\begin{enumerate}
		\item For any $\gah_0 \in \CX_0$,
		there exist (small) $T>0$ and a unique solution $\gah(t)\in C([0,T];\CX_0)$ to \eqref{eq:NLH} with $\gah(0)=\gah_0$ such that 
		\begin{align}
			\No{\rhh(\gah(t))}_{\CY_{0,T}} \le C_0 \No{\gah_0}_{\CX_0},
		\end{align}	
		where $C_0$ depends only on $w$ and $\si$.
		\item Let $T_\mx^\hbar>0$ be the maximal forward existence time of the solution obtained in (1). Then, we have
		\begin{align}
			T_\mx^\hbar < \I \implies \lim_{t \to T_\mx^\hbar} \No{\gah(t)}_{\CX_0} = \I.
		\end{align} 
		\end{enumerate}
\end{lemma}

\begin{remark}
	The small time $T>0$ given in Lemma \ref{lem:LWP} (1) might depend on $\hb$, but it never causes a problem (see Section \ref{subsec:proof}). 
\end{remark}

\begin{proof}
	\smallskip \noindent  \underline{\textbf{Setup of the proof.}}
	Let $A:= C_0 \No{\gah_0}_{\CX_0}$, where $C_0$ is a sufficiently large constant depending on $w$ and $\si$. Define
	\begin{equation}
	E_{T,A}:= \Ck{(\gah,\ze)\in C([0,T];\CX_0)\times \CY_{0,T} : \|\gah\|_{C([0,T];\CX_0)} \le A,\quad \|\ze\|_{\CY_{0,T}} \le A},
	\end{equation}
	where $T>0$ will be chosen later.
	Set $\Ph_1$ and $\Ph_2$ by
	\begin{align}
		&\Ph_1[\gah,\ze](t):= \U(t)\gah_0\U(t)^* - \frac{i}{\hb}\int_0^t \U(t-\ta) \Bk{w\ast \ze(\ta), \gah(\ta)} \U(\ta-t)d\ta,\\
		&\Ph_2[\gah,\ze](t):= \rhh\K{\Ph_1[\gah,\ze](t)}.
	\end{align}
	
	\smallskip \noindent  \underline{\bf Estimate 1.}
	First, we have
	\begin{align}
		\No{\Ph_1[\gah,\ze]}_{C([0,T];\FS^1_\hb)}
		&\le \No{\gah_0}_{\FS^1_\hb} + \frac{2}{\hb}\int_0^T \No{w}_{L^\I_x} \|\ze(\ta)\|_{L^1_x} \No{\gah(\ta)}_{\FS^1_\hb} d\ta \\
		&\le \No{\gah_0}_{\FS^1_\hb} + \frac{C_wT}{\hb} \No{\gah}_{C([0,T];\CX_0)} \|\ze\|_{\CY_{0,T}}.
	\end{align}
	Hence, inequality \eqref{eq:density L1 bound} implies
	\begin{align}
		\No{\Ph_2[\gah,\ze]}_{C([0,T];L^1_x)} &\le \No{\Ph_1[\gah,\ze]}_{C([0,T];\FS^1_\hb)} 
		\le \No{\gah_0}_{\FS^1_\hb} + \frac{C_wT}{\hb} \No{\gah}_{C([0,T];\CX_0)} \|\ze\|_{\CY_{0,T}}.
	\end{align}
	
	\smallskip \noindent  \underline{\bf Estimate 2.}
	Second, we have
	\begin{align}
		&\No{\sd^\si \Ph_1[\gah,\ze]\sd^\si }_{C([0,T];\CB)}
		\le \No{\sd^\si \gah_0 \sd^\si}_{\CB} \\
		&\qquad + \frac{2}{\hb}\int_0^T \No{\sd^\si (w\ast \ze)(\ta)\gah(\ta) \sd^\si}_{\CB} d\ta \\
		&\quad \le \No{\sd^\si \gah_0 \sd^\si}_{\CB} + \frac{2T}{\hb} \No{\sd^\si(w\ast \ze)(t) \sd^{-\si}}_{C([0,T];\CB)} \No{\gah}_{C([0,T];\CX_0)}.
	\end{align}
	Note that the fractional Leibniz rule implies
	\begin{align}
		&\No{\sd^\si(w\ast \ze)(t) \sd^{-\si}}_\CB
		\le \No{(w\ast \ze)(t) \sd^{-\si}}_{\CB} + \No{|\hb\na|^\si (w\ast \ze)(t) \sd^{-\si}}_\CB \\
		&\quad \ls \|w\|_{L^\I_x} \|\ze(t)\|_{L^1_x} + \hb^{\si-1/2}\No{|\na|^\si (w\ast \ze)(t)}_{L^6_x}
		\ls \|\ze(t)\|_{L^1_x}.
	\end{align}
	Hence, we have
	\begin{align}
		&\No{\sd^\si \Ph_1[\gah,\ze]\sd^\si }_{C([0,T];\CB)}
		\le \No{\sd^\si \gah_0 \sd^\si}_{\CB} + \frac{C_{w,\si}T}{\hb} \No{\gah}_{C([0,T];\CX_0)}\|\ze\|_{\CY_{0,T}} .
	\end{align}
	Moreover, by Lemma \ref{lem:Sobolev}, we have
	\begin{align}
		\|\Ph_2[\gah,\ze]\|_{C([0,T];L^\I_x)} &\ls \No{\sd^\si \Ph_1[\gah,\ze] \sd^\si}_{C([0,T];\CB)} \\
		&\ls \No{\sd^\si \gah_0 \sd^\si}_{\CB} + \frac{C_{w,\si}T}{\hb} \No{\gah}_{C([0,T];\CX_0)}\|\ze\|_{\CY_{0,T}}.
	\end{align}
	
	\smallskip \noindent  \underline{\bf Estimate 3.}
	By the Jacobi identity \eqref{eq:Jacobi} for commutators, we have
	\begin{align}
		&\No{\Bk{\na, \Ph_1[\gah,\ze]}}_{C([0,T];\FS^1_\hb)}
		\le \No{\Bk{\na,\gah_0}}_{\FS^1_\hb} + \frac{1}{\hb}\int_0^T \No{\Bk{\na, \Bk{w \ast \ze(\ta), \gah(\ta)}}}_{\FS^1_\hb} d\ta \\
		&\qquad \le \No{\Bk{\na,\gah_0}}_{\FS^1_\hb} + \frac{1}{\hb}\int_0^T \K{\No{\Bk{w \ast \ze(\ta), \Bk{\na, \gah(\ta)}}}_{\FS^1_\hb} + \No{ \Bk{\Bk{\na,w \ast \ze(\ta)}, \gah(\ta)}}_{\FS^1_\hb}} d\ta \\
		&\qquad \le \No{\Bk{\na,\gah_0}}_{\FS^1_\hb} + \frac{C_w}{\hb}\int_0^T \|\ze(\ta)\|_{L^1_x} \K{\No{\Bk{\na, \gah(\ta)}}_{\FS^1_\hb} + \No{\gah(\ta)}_{\FS^1_\hb}} d\ta \\
		&\qquad \le \No{\Bk{\na,\gah_0}}_{\FS^1_\hb} + \frac{C_wT}{\hb} \No{\gah}_{C([0,T];\CX_0)}\|\ze\|_{\CY_{0,T}}.
	\end{align}
	Hence, by \eqref{eq:density L1 bound}, we obtain
	\begin{align}
		\No{\na \Ph_2[\gah,\ze]}_{C([0,T];L^1_x)}
		&\le \No{\Bk{\na, \Ph_1[\gah,\ze]}}_{C([0,T];\FS^1_\hb)} \\
		&\le \No{\Bk{\na,\gah_0}}_{\FS^1_\hb} + \frac{C_wT}{\hb} \No{\gah}_{C([0,T];\CX_0)}\|\ze\|_{\CY_{0,T}}.
	\end{align}
	
	\smallskip \noindent  \underline{\bf Estimate 4.}
	By the Jacobi identity \eqref{eq:Jacobi} for the commutator, we have
	\begin{align}
		&\No{\sd^\si\Bk{\na,\Ph_1[\gah,\ze]}\sd^\si }_{C([0,T];\CB)}
		\le \No{\sd^\si \Bk{\na,\gah_0} \sd^\si}_{\CB} \\
		&\qquad \qquad  + \frac{1}{\hb}\int_0^T \No{\sd^\si \Bk{\na, \Bk{w \ast \ze(\ta), \gah(\ta)}}\sd^\si }_{\CB} d\ta \\
		&\qquad \le \No{\sd^\si \Bk{\na,\gah_0} \sd^\si }_{\CB} + \frac{1}{\hb}\int_0^T \Big(\No{\sd^\si \Bk{w \ast \ze(\ta), \Bk{\na, \gah(\ta)}}\sd^\si}_{\CB}\\
		&\qquad \qquad \qquad \qquad \qquad \qquad \qquad 
		\qquad \qquad  + \No{\sd^\si \Bk{\Bk{\na,w \ast \ze(\ta)}, \gah(\ta)}\sd^\si}_{\CB}\Big) d\ta \\
		&\qquad \le \No{\sd^\si\Bk{\na,\gah_0}\sd^\si}_{\CB} + \frac{C_{w,\si}}{\hb}\int_0^T \|\ze(\ta)\|_{L^1_x} \Big(\No{\sd^\si\Bk{\na, \gah(\ta)}\sd^\si}_{\CB} \\
		&\qquad \qquad \qquad \qquad \qquad \qquad \qquad\qquad \qquad + \No{\sd^\si \gah(\ta) \sd^\si}_{\CB}\Big) d\ta \\
		&\qquad \le \No{\sd^\si \Bk{\na,\gah_0} \sd^\si }_{\CB} + \frac{C_{w,\si}T}{\hb} \No{\gah}_{C([0,T];\CX_0)}\|\ze\|_{\CY_{0,T}}.
	\end{align}
	Hence, by Lemma \ref{lem:Sobolev}, we have
	\begin{align}
		\No{\na \Ph_2[\gah,\ze]}_{C([0,T];L^\I_x)}
		&\ls \No{\sd^\si\Bk{\na,\Ph_1[\gah,\ze]}\sd^\si }_{C([0,T];\CB)} \\
		&\le \No{\sd^\si \Bk{\na,\gah_0} \sd^\si }_{\CB} + \frac{C_{w,\si}T}{\hb} \No{\gah}_{C([0,T];\CX_0)}\|\ze\|_{\CY_{0,T}}.
	\end{align}
	
	\smallskip \noindent  \underline{\bf Estimate 5.}
	In the same way as \textbf{Estimate 4}, we have
	\begin{align}
		&\No{\sd^{\si} \Bk{\na,\Ph_1[\gah,\ze]}\sd^{\si}}_{C([0,T];\FS^2_\hb)} \\
		&\qquad \ls \No{\sd^\si \Bk{\na,\gah_0} \sd^\si }_{\FS^2_\hb} + \frac{C_{w,\si}T}{\hb} \No{\gah}_{C([0,T];\CX_0)}\|\ze\|_{\CY_{0,T}}.
	\end{align}
	Hence, by Lemma \ref{lem:free Fourier}, we have
	\begin{align}
		\hb^{3/2} \No{\na \Ph_2[\gah,\ze]}_{C([0,T];\CF L^1_x)}
		&\ls \No{\sd^\si\Bk{\na,\Ph_1[\gah,\ze]}\sd^\si }_{C([0,T];\FS^2_\hb)} \\
		&\le \No{\sd^\si \Bk{\na,\gah_0} \sd^\si }_{\FS^2_\hb} + \frac{C_{w,\si}T}{\hb} \No{\gah}_{C([0,T];\CX_0)}\|\ze\|_{\CY_{0,T}}.
	\end{align}

    \smallskip \noindent  \underline{\bf Estimate 6.}
    Applying the Jacobi identity \eqref{eq:Jacobi} for the commutator twice, we obtain
    \begin{align}
    	&\No{\sd^{\si/2}\Bk{\na,\Bk{\na,\Ph_1[\gah,\ze]}}\sd^{\si/2} }_{C([0,T];\FS^2_\hb)}
    	\le \No{\sd^{\si/2} \Bk{\na,\Bk{\na,\gah_0}} \sd^{\si/2}}_{\FS^2_\hb} \\
    	&\qquad \qquad  + \frac{1}{\hb}\int_0^T \No{\sd^{\si/2} \Bk{\na, \Bk{\na, \Bk{w \ast \ze(\ta), \gah(\ta)}}}\sd^{\si/2} }_{\FS^2_\hb} d\ta \\
      	&\qquad \le \No{\sd^{\si/2}\Bk{\na,\Bk{\na,\gah_0}}\sd^{\si/2}}_{\FS^2_\hb} \\
      	&\qquad \qquad + \frac{C_w}{\hb}\int_0^T \|\ze(\ta)\|_{L^1_x} \Big\{\No{\sd^{\si/2}\Bk{\na,\Bk{\na, \gah(\ta)}}\sd^{\si/2}}_{\FS^2_\hb} \\
      	&\qquad \qquad \qquad \qquad
      	+ \No{\sd^{\si/2}\Bk{\na, \gah(\ta)}\sd^{\si/2}}_{\FS^2_\hb} 
      	+ \No{\sd^{\si/2} \gah(\ta) \sd^{\si/2}}_{\FS^2_\hb}\Big\} d\ta \\
    	&\qquad \le \No{\sd^{\si/2} \Bk{\na, \Bk{\na,\gah_0}} \sd^{\si/2} }_{\FS^2_\hb} + \frac{C_wT}{\hb} \No{\gah}_{C([0,T];\CX_0)}\|\ze\|_{\CY_{0,T}}.
    \end{align}
    Hence, it follows from Lemma \ref{lem:Sobolev} that
    \begin{align}
    	\No{\na^2 \Ph_2[\gah,\ze]}_{C([0,T];L^2_x)} &\ls \No{\sd^{\si/2}\Bk{\na,\Bk{\na,\Ph_1[\gah,\ze]}}\sd^{\si/2} }_{C([0,T];\FS^2_\hb)} \\
    	&\ls \No{\sd^{\si/2} \Bk{\na, \Bk{\na,\gah_0}} \sd^{\si/2} }_{\FS^2_\hb} + \frac{C_wT}{\hb} \No{\gah}_{C([0,T];\CX_0)}\|\ze\|_{\CY_{0,T}}.
    \end{align}
    
    \smallskip \noindent  \underline{\bf Conclusion.}
    Collecting \textbf{Estimates 1--6}, we have
    \begin{align}
    	&\No{\Ph_1[\gah,\ze]}_{C([0,T];\CX_0)} \le \No{\gah_0}_{\CX_0} + \frac{C_{w,\si} T}{\hb} \No{\gah}_{C([0,T];\CX_0)}\|\ze\|_{\CY_{0,T}}, \\
    	&\No{\Ph_2[\gah,\ze]}_{\CY_{0,T}} \le C_{w,\si} \No{\gah_0}_{\CX_0} + \frac{C_{w,\si} T}{\hb} \No{\gah}_{C([0,T];\CX_0)}\|\ze\|_{\CY_{0,T}}.
    \end{align}
    Therefore, by the standard argument with the Banach fixed point theorem, we obtain the desired result.
\end{proof}

\subsection{Proof of Theorem \ref{th:main} and Remark \ref{rmk:modified scattering}}\label{subsec:proof}
Let $C\ge 1$ and small $R>0$ be constants given in Proposition \ref{prop:a priori}.
Let $C_0\ge 1$ be a constant given in Lemma \ref{lem:LWP}.
Define $\ep_0:= R/(100CC_0)$.
Then, under the assumptions in Theorem \ref{th:main}, we obtain the maximal solution $\gah(t) \in C([0,T_\mx^\hbar);\CX_0)$ given in Lemma \ref{lem:LWP}.
Now, we assume
\begin{align}
	T_*^\hbar:= \sup \Ck{0<T < T_\mx^\hbar : \No{\rhh(\gah)}_{\CY^{a,b}_T} \le 100CC_0 \No{\gah_0}_{\CX^\si_\hb}}<\I.
\end{align}
Note that
$$
\Ck{T>0 : \No{\rhh(\gah)}_{\CY^{a,b}_T} \le 100CC_0  \No{\gah_0}_{\CX^\si_\hb}}\neq\varnothing
$$
because Lemma \ref{lem:LWP} implies, when $T\le 1$, 
\begin{equation}
	\No{\rhh(\gah)}_{\CY^{a,b}_T} \le 100 \No{\rhh(\gah)}_{\CY_{0,T}} \le 100 C_0 \No{\gah_0}_{\CX_0} \le  100CC_0  \No{\gah_0}_{\CX^\si_\hb}.
\end{equation}
Since we have the continuity (see also Remark \ref{rmk:continuity})
\begin{align}
	&\rhh\K{\gah(t)} \in C([0,T_\mx^\hb); L^1_x \cap L^\I_x), \quad 
	\na \rhh\K{\gah(t)} \in C([0,T_\mx^\hb); L^1_x \cap \CF L^1),\\ 
	&\na^2 \rhh\K{\gah(t)} \in C([0,T_\mx^\hb); L^2_x),
\end{align}
it follows that $\No{\rhh(\gah)}_{\CY^{a,b}_{T_*^\hbar}} \le 100CC_0\No{\gah_0}_{\CX^\si_\hb}$.
Since $\No{\gah_0}_{\CX^\si_\hb}\le R/(100CC_0)$, it holds that $\No{\rhh(\gah)}_{\CY^{a,b}_{T_*^\hbar}}\le R$.
Hence, by Proposition \ref{prop:a priori}, we obtain
\begin{align}
	\No{\rhh(\gah)}_{\CY_{T_*^\hbar}^{a,b}} = \No{\rhh\K{\U_{w\ast \rhh(\gah)}(t) \gah_0 \U_{w\ast\rhh(\gah)}(t)^*}}_{\CY_{T_*^\hbar}^{a,b}}
	\le C \No{\gah_0}_{\CX^{\si}_\hb}.
\end{align}
When $T_*^\hbar < T_\mx^\hbar$, it contradicts the maximality of $T_*^\hbar$.
When $T_*^\hbar = T_\mx^\hbar < \I$, it contradicts the blowup alternative (see Lemma \ref{lem:LWP}). 
Therefore, we obtain $T_*^\hbar=\I$ and complete the proof of Theorem \ref{th:main}.

Finally, we give a proof of Remark \ref{rmk:modified scattering}.
Since
\begin{equation}
	i\hb\pl_t e^{i\Ps(t,-i\hb\na_x)} \W_V(t) = e^{i\Ps(t,-i\hb\na_x)} \Ck{-V(t,-it\hb\na_x)+\M(-t)V(t,-it\hb\na_x) \M(t)} \W_V(t),
\end{equation}
we have
\begin{equation}
	\begin{aligned}
	&e^{i\Ps(t,-i\hb\na_x)} \W_V(t) \\
	&\quad = \Id_{L^2_x} + \frac{i}{\hb}\int_0^t e^{i\Ps(\ta,-i\hb\na_x)} \Ck{V(\ta,-i\ta\hb\na_x) - \M(-\ta)V(\ta,-i\ta\hb\na_x) \M(\ta)} \W_V(\ta) d\ta.
	\end{aligned}
\end{equation}
Therefore, by Lemmas \ref{lem:easy}, \ref{lem:key estimate}, \ref{lem:potential deriv}, and \ref{lem:potential deriv 2}, we have
\begin{equation}\label{eq:modified wave convergence}
\begin{aligned}
	&\No{\Ck{e^{i\Ps(t,-i\hb\na_x)}\W_V(t) - e^{i\Ps(s,-i\hb\na_x)}\W_V(s)}\bx^{-1}}_\CB \\
	&\quad \ls \frac{1}{\hb}\int_s^t \No{\Ck{V(\ta,-i\ta\hb\na_x) - \M(-\ta)V(\ta,-i\ta\hb\na_x) \M(\ta)} \bx^{-1}}_\CB \log(\ta+2)d\ta \\
	&\quad \ls \frac{1}{\hb} \int_s^t \No{\na V(\ta)}_{\CF L^1_x}^{1/3} \K{\No{V(\ta)}_{L^\I_x}^{2/3} + |\ta\hb|^{2/3}\No{\na V(\ta)}_{\CF L^1}^{2/3}}\log(\ta+2)d\ta \\
	&\quad \ls \frac{1}{\hb}\int_s^t \frac{1}{\bta^{4/3-a/2-\ep}} d\ta \to 0 \quad \mbox{as } t,s \to \I.
\end{aligned}
\end{equation}
From the above, we obtain
\begin{align}
	&\No{e^{i\Ps(t,-i\hb\na_x)} \U(t)^* \gah(t) \U(t) e^{-i\Ps(t,-i\hb\na_x)} - e^{i\Ps(s,-i\hb\na_x)} \U(s)^* \gah(s) \U(s) e^{-i\Ps(s,-i\hb\na_x)} }_{\FS^2_\hbar} \\
	&\quad \ls \No{\Ck{e^{i\Ps(t,-i\hb\na_x)}\W_V(t) - e^{i\Ps(s,-i\hb\na_x)}\W_V(s)}\bx^{-1}}_{\CB}
	\No{\bx^\si \gah_0 \bx^\si}_{\FS^2_\hb}  \to 0 \mbox{ as } t,s \to \I.
\end{align}
By the completeness of $\FS^2_\hb$, we conclude that there exists $\gah_+ \in \FS_\hb^2$ such that
\begin{align}
	&\No{\gah_+ - e^{i\Ps(t,-i\hb\na_x)} \U(t)^* \gah(t) \U(t) e^{-i\Ps(t,-i\hb\na_x)}}_{\FS^2_\hb} 
	 \to 0 \mbox{ as } t \to \I.
\end{align}
Since
\[
\sup_{t \ge 0}\|e^{i\Ps(t,-i\hb\na_x)} \U(t)^* \gah(t) \U(t) e^{-i\Ps(t,-i\hb\na_x)}\|_{\FS^1_\hb} = \|\gah_0\|_{\FS^1_\hb}<\I,
\]
we have $\gah_+ \in \FS^1_\hb$.
\qed

\begin{remark}\label{rmk:final}
	You can find a $1/\hb$--factor in \eqref{eq:modified wave convergence}. This is the reason we cannot prove the modified scattering uniformly with respect to $\hb$.
\end{remark}

\section*{Declarations}
\subsection*{Funding}
The author was supported by JSPS KAKENHI Grant Number 24KJ1338.

\subsection*{AI use}
The first version of this manuscript was written without the use of AI and uploaded on arXiv. During revision, ChatGPT Sol 6.1 and Astra 6 were used to refine the mathematical arguments and identify typographical errors.

\end{document}